\documentclass[11pt]{amsart}
\usepackage{graphicx} % Required for inserting images

\usepackage[english]{babel}
\usepackage{xcolor}
\definecolor{myteal}{RGB}{53,150,166}
\usepackage[colorlinks,citecolor=myteal,urlcolor=blue,linkcolor=purple]{hyperref}

\usepackage[letterpaper,top=1.25in,bottom=1in,left=1in,right=1in,marginparwidth=2cm]{geometry}
\usepackage{lmodern}

\usepackage{amsmath}
\usepackage{graphicx}
\usepackage{amssymb}
\usepackage{wasysym}
\usepackage{bbm}
\usepackage{todonotes}
\usepackage{float}
\usepackage{enumerate}
\usepackage{dsfont}

\usepackage{imakeidx}%adding indices
\usepackage{amsthm}
\newtheorem{thm}{Theorem}[section]
\newtheorem{lemma}[thm]{Lemma}
\newtheorem{prop}[thm]{Proposition}
\newtheorem{cor}[thm]{Corollary}

\theoremstyle{definition}
\newtheorem{definition}[thm]{Definition}

\newtheorem{rem}[thm]{Remark}
\newtheorem{question}[thm]{Question}

\numberwithin{equation}{section}
\newcommand{\mc}[1]{\mathcal{#1}}
\newcommand{\m}[1]{\mathbb{#1}}
\newcommand{\bcc}{\text{BCC}}
\newcommand{\dd}{\mathrm{d}}

\newcommand{\pd}[2]{\frac{\partial #1}{\partial #2}}

\newcommand{\be}{\begin{equation}}
\newcommand{\ee}{\end{equation}} 

\newcommand{\old}[1]{}

\newcommand{\Hess}{\text{Hess}}
\newcommand{\Ent}{\mathrm{Ent}}

\newcommand{\R}{{\mathbb R}}
\newcommand{\Z}{{\mathbb Z}}
\newcommand{\T}{{\mathbb T}}

\newcommand{\NN}{{\mathbf N}}

\newcommand{\AD}{\text{AD}}

\newcommand{\GN}{G_{\NN}}

\renewcommand{\H}{\mathbb{H}}

\title{The multinomial dimer model}
\author{Richard Kenyon \and Catherine Wolfram}

\address{Richard Kenyon\\Department of Mathematics, Yale University, New Haven, 06511}
\email{richard dot kenyon at yale.edu}
\address{Catherine Wolfram\\Department of Mathematics, Yale University, New Haven, 06511}
\email{catherine dot wolfram at yale.edu}

\date{\today}
\begin{document}

\begin{abstract}

The dimer model is a classical statistical mechanics model which is exactly solvable in two dimensions, but about which little is known in higher dimensions. In analogy with large $N$ limits in lattice gauge theory, we study a large $N$ limit of the dimer model in any dimension $d$. The dependence on $N$ comes from the multinomial tiling model introduced by Kenyon and Pohoata, which gives a general framework for adding a dependence on $N$ to a tiling model. We study the behavior of this model on periodic bipartite graphs in ${\mathbb R}^d$, in the scaling limit as the multiplicity $N$ and then the size of the graph go to infinity. 

In this iterated limit, in any dimension $d$, we prove a variational principle and show that random configurations concentrate on a limit shape which is the unique solution to an associated system of Euler-Lagrange equations. The rate function of the variational principle is the integral of a surface tension function, which we can compute explicitly for lattices in any dimension $d$ as the Legendre dual of the free energy for the model on the torus. 
%Variational principles of an analogous form are known when $d=2$ and $d=3$ for the standard dimer model, but the surface tension and Euler-Lagrange equations can only be computed when $d=2$. In the large $N$ framework, 
We give a unified methodology for computing the surface tension and Euler-Lagrange equations in any dimension $d$.

A new structure called the \textit{critical gauge} also emerges in the large $N$ limit. We show that the critical gauge functions converges in the scaling limit to a limiting gauge function which is the unique solution to a dual Euler-Lagrange equation. This limiting gauge function determines the limit shape and vice versa.

We further use our techniques to compute explicit limit shapes in some two and three dimensional examples, such as the Aztec diamond and ``Aztec cuboid". This is one of the first stat mech models in dimensions $d\ge3$ where limit shapes can be computed explicitly. 
\end{abstract}
\maketitle

\tableofcontents

\section{Introduction}

A \textit{dimer cover} of a graph $G=(V,E)$ is a collection of edges $\tau\subset E$ such that every vertex of $G$ is covered by exactly one edge in $\tau$. The \textit{dimer model} on a lattice $\Lambda$ is the study of random dimer covers of subgraphs $G\subset \Lambda$. The lattices we consider in this paper are always bipartite. This is ``$d$-dimensional'' if $\Lambda$ is a lattice which spans $\m R^d$. 

A classical question in statistical mechanics is to understand the behavior of the model in scaling limits, i.e.\ the behavior of random dimer covers on a sequence of graphs $R_n\subset \frac{1}{n}\Lambda$ as $n\to\infty$.

\begin{figure}
    \centering
    \includegraphics[width=0.45\linewidth]{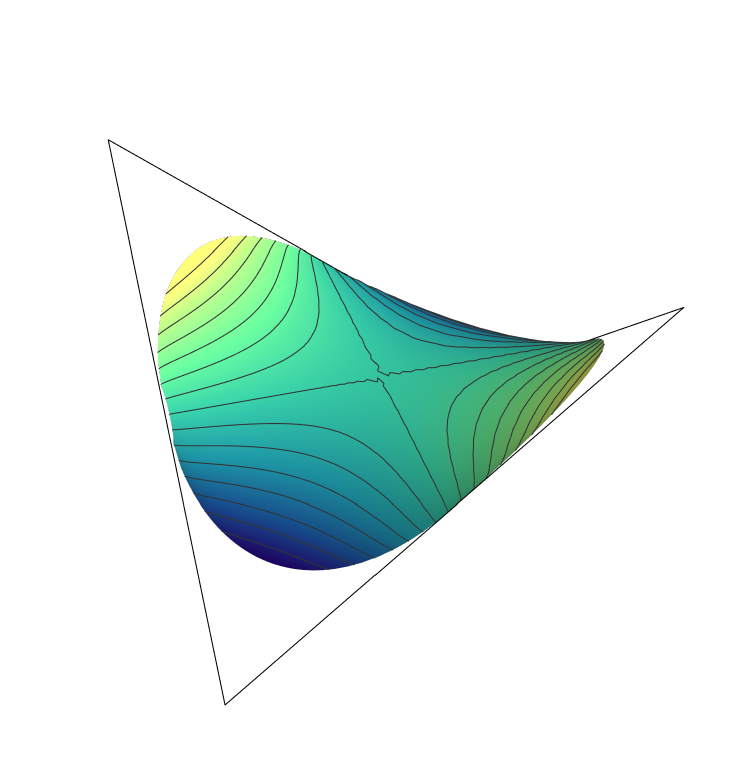}\includegraphics[width=0.45\linewidth]{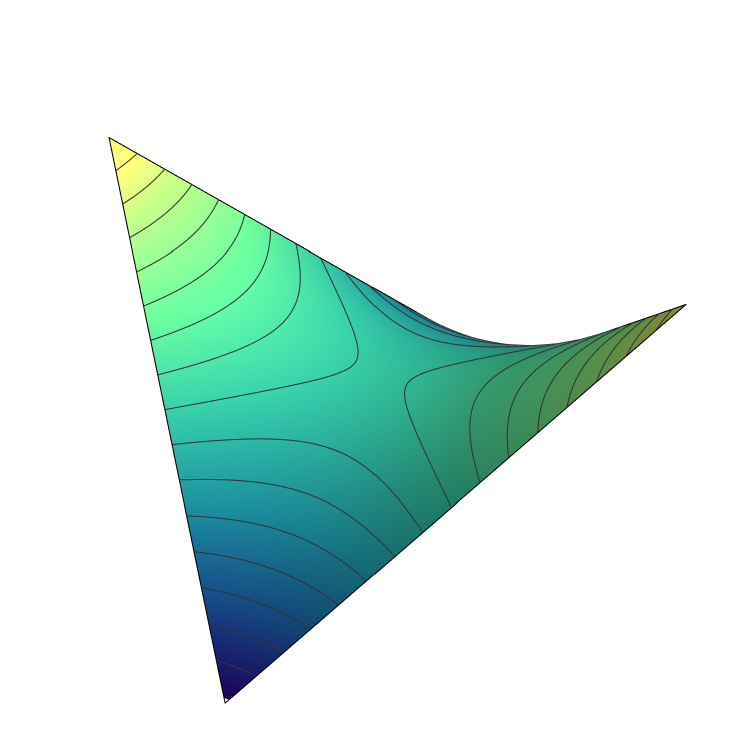}
    \caption{The limit shape height functions on the Aztec diamond for the dimer model (left) and the multinomial dimer model in the large $N$ limit (right).}
    \label{fig:aztecs}
\end{figure}

In two dimensions, there is a correspondence between dimer covers and Lipschitz \textit{height functions} \cite{Thurston}, and these scaling limit questions can be made precise in terms of the corresponding height functions. In 2001, Cohn, Kenyon and Propp showed that if $R_n\subset \frac{1}{n}\Z^2$ approximate a compact region $R\subset \m R^2$ and boundary condition as $n\to \infty$, then interpreted as a random function, uniform random dimer covers concentrate on a unique \textit{limit shape} height function given the boundary conditions \cite{cohn2001variational}. They further prove a full large deviation principle, and explicitly compute the rate function which measures the exponentially small probability that a scaling limit lies close to any other possible limiting height function. See the left half of Figure~\ref{fig:aztecs} for a graph of the limit shape height function on the Aztec diamond; this function is smooth inside the circle (``artic curve'') and linear in a region around the four corners \cite{jockusch1998random}.

Understanding the limit shapes for the two dimensional dimer model, their fluctuations, and other properties touches on many different areas of mathematics, including combinatorics, complex analysis, hyperbolic geometry, PDEs, and algebraic geometry, see e.g.\ \cite{CohnElkiesPropp,jockusch1998random,kenyon2000conformal,dimersGFF2000,cohn2001variational,AST_2005__304__R1_0,burgers2007,kenyon2006dimers,kenyon2014conformal,dimerscircles,kuchumov22}, and connects it to the Gaussian free field, Schramm-Loewner evolutions, and conformal field theory, see e.g.\  \cite{dimersGFF2000,kenyon2014conformal,Dubedat2014DoubleDC,Berestycki-imaginary,HonglerCFT}. See also \cite{kenyon2009lectures,gorin2021lectures} for more general overviews. The dimer model played a pivotal role in the development of the field of random conformal geometry in two dimensions, as it was one of the first statistical mechanics models where conformal invariance in its scaling limit was rigorously proved \cite{kenyon2000conformal}, confirming physics predictions from the 1980s \cite{BPZ}. 

One of the things that makes the two dimensional dimer model so tractable is that it is \textit{exactly solvable}. This stems from the Kasteleyn determinant formula, due to Kasteleyn \cite{Kasteleyn} and independently \cite{TemperleyFisher} in the 1960s, that the number of dimer covers of a bipartite planar graph can be computed as a determinant. This has in turn made it possible to compute many other key statistical quantities in the model, including the rate function for the large deviation principle for two dimensional dimers.

Like many statistical mechanics models, much less is known about the dimer model in higher dimensions. In dimensions $d>2$, the dimer model does not appear to be exactly solvable, and the correspondence with height functions also does not extend to higher dimensions. In fact it is provably known that the Kasteleyn formula does not hold for almost any non-planar graph \cite{LITTLE1975187}. Despite these challenges, some things are known about dimers in higher dimensions. See e.g. \cite{taggi2022uniformly,quitmann2022macroscopicdimer} on correlation functions on the torus, \cite{freedman2011weakly,milet2014domino,freire2022connectivity,klivans2020domino,demarreiros2025} for a small subset of works on ``local moves'' and topological invariants that arise from the failure of determinants formulas, and \cite{3Ddimers} on scaling limits for dimer tilings of $\m Z^3$. See also \cite{LindaMooreNordahl,arctic_oct_MC,kozma,SheffieldYadin,Lammers,quitmann2022macroscopicloops} for a small selection of works on related three-dimensional models.

To make sense of the scaling limit questions without height functions, one can use a general correspondence, which holds in any dimension, between dimer covers and \textit{discrete flows}. Using discrete flows, Chandgotia, Sheffield and Wolfram \cite{3Ddimers} proved a three-dimensional analog of \cite{cohn2001variational}, i.e., they showed that the scaling limit of random dimer covers of subgraphs $R_n\subset \frac{1}{n}\m Z^3$, interpreted as a random flow, also concentrate on a unique limit shape given boundary conditions, and prove a large deviation principle. However, no explicit formulas for the rate function or any other statistical quantities are known, and nothing is known rigorously about what these limit shapes look like. 

The main subject of this paper is a a \textit{large $N$ limit} of the dimer model, which we call the multinomial dimer model, which is exactly solvable in any dimension $d$ in the large $N$ limit.  This can be seen in analogy with large $N$ limits in lattice gauge theory, see e.g.\ \cite{thooft,gaugetheorylargeNsurvey,ChatterjeelargeN,BCSlargeN}. Combinatorially, an $N$-dimer cover is a collection of edges such that every vertex is contained in exactly $N$ edges (with multiplicity) in the collection. We call this the multinomial dimer model because it is an instance of the mulitnomial tiling model introduced by Kenyon and Pohoata in 2022, who give a general framework for adding a dependence on $N$ to a tiling model \cite{KenyonPohoata}.  These formulas do \textit{not} involve determinants, and instead come from the fact that the partition functions for multinomial dimer measures have a generating function in $N$ which can be expanded around infinity. 
%We remark that the existence of any sequence of measures which have this property is nontrivial, and that this was established by the explicit construction of \cite{KenyonPohoata}. 

Our main goal will be to prove variational principles for dimers in this large $N$ limit, analogous to \cite{cohn2001variational} and \cite{3Ddimers}, but for lattices in any dimension, in the iterated limit as $N$ and then the size of the graph go to infinity, and to go beyond this to develop the theory of limit shapes and Euler-Lagrange equations for this model.  

See the right half of Figure~\ref{fig:aztecs} for the limit shape height function of this model on the Aztec diamond, which happens to be especially simple:  $h(x,y) = \frac{1}{4}(2x-1)(2y-1)$. The analogous solution in the weighted case (see (\ref{wAD})) is significantly more interesting. 

To compute limit shapes from our explicit formulas, the most straightforward method would be to solve the family of Euler-Lagrange equations, which we can write explicitly. In two dimensions, where this reduces to the usual gradient variational problem, this can be done in substantial generality using the \textit{tangent plane method} of \cite{KenyonPrause}. In higher dimensions, this remains challenging. 

In fact, our current examples of limit shapes primarily come from new structure in the large $N$ limit, given in terms of dual \textit{gauge} variables and their scaling limits. See Section~\ref{sec:critical_gauge_intro} (this gauge structure is especially simple for the Aztec diamond; this is related to why this limit shape formula is so simple in this case) and Section~\ref{sec:limit_shapes}.

%While some aspect of the model simplify in the large $N$ limit (fluctuations are suppressed, and the limit shapes have more regularity), we show via large deviations that the model still has the same ``mesoscopic'' behavior in the large $N$ limit, see Section~\ref{} for a more detailed overview. New structure also emerges in the large $N$ limit, which gives a new perspective on limit shapes and Euler-Lagrange equations, see Section~\ref{}. 

\subsection{Scalings limits in general dimension} In any dimension, there is a correspondence between dimer covers and \textit{discrete flows.} Discrete flows will play the role analogous to height functions in two dimensions. This correspondence with discrete flows was used in \cite{3Ddimers} to establish the variational principle for dimer tilings of $\m Z^3$, and we use this same general framework here. Recall that all the graphs we consider are bipartite, with bipartition into ``white'' and ``black'' vertices. Given an $N$-dimer cover $M$, the corresponding \textit{discrete flow} $\omega_M$ is given by 
\begin{align*}
    \omega_M(e) = M_e/N \in [0,1]
\end{align*}
if $e$ is oriented from its white vertex to its black vertex (and $\omega_M(-e) = -\omega_M(e)$ if $-e$ denotes $e$ with reversed orientation). Here $M_e$ is the multiplicity of $e$ in $M$, so if $M$ is a single dimer cover, then $\omega_M(e)\in \{0,1\}$. By construction, the divergences $\sum_{e\ni v} \omega_M(e)$ of $\omega_M$ are $\pm 1$ for all vertices $v$. If $M,M'$ are dimer covers on the same graph, then we can compare them using the sup norm on the edges of the graph. 

We now specialize to the case of dimer covers of subgraphs of lattices. Throughout, we will assume that the bipartite lattices $\Lambda$ that we consider have various symmetries that simplify our exposition. In particular we assume that $\Lambda$ has a ``bipartite transitive" $\Z^d$ action by translations, that is, transitive on white vertices and transitive on black vertices. We let
$D$ denote the common degree, and incident edge directions at white vertices are denoted $e_1,\dots, e_D$. See Section~\ref{sec:subgraphs_boundaryconditions}.

To compare dimer covers of a sequence of graphs $R_n\subset \frac{1}{n}\Lambda\subset \m R^d$, we use the \textit{weak topology on flows}, which is described in Section~\ref{sec:vector_fields}. This is the same topology used in \cite{3Ddimers} to compare dimer covers of $\m Z^3$ (and obtain a large deviations principle), but we streamline its description in adapting it from $\m Z^3$ to more general lattices. In this topology, the $n\to \infty$ scaling limits of discrete flows on $R_n$ are \textit{asymptotic flows} (see Theorem~\ref{thm:scaling_limits_asymptotic}), namely, they are measurable divergence-free vector fields supported in $R$ and valued in the \textit{Newton polytope} $\mc N(\Lambda)$, defined in our setting by
\begin{align*}
    \mc N(\Lambda) = \text{convex hull}\{e_1,\dots,e_D\}.
\end{align*}
We call a point $s\in \mc N(\Lambda)$ a \textit{slope.} See Section~\ref{sec:slope}.

This flow description makes sense in any dimension, in particular it still makes sense in two dimensions where the height function does exist. In this case, the relationship between divergence-free flows and height functions is simple: if $\omega =(\omega_1,\omega_2)$ is divergence-free, then its curl-free dual is $(-\omega_2,\omega_1)=\nabla h$, where $h$ is the height function. In order to give a unified description, we use the discrete flow correspondence to formulate the large deviation principle for multinomial dimers in any dimensions $d$.

\subsection{Large deviations and variational principle}\label{sec:ldp_intro}

Fix a lattice $\Lambda$ in $\m R^d$ as above. Let $R\subset \m R^d$ be a compact region which is the closure of its interior and has piecewise smooth boundary, and fix a boundary condition $b$ on $\partial R$ (e.g., the boundary value of an asymptotic flow). Finally, let $R_n\subset \frac{1}{n}\Lambda$ be a sequence of graphs which approximate $R$ in Hausdorff distance
and whose boundaries approximate $b$ in the appropriate sense. 
%%%
\old{
We define the boundary of $R_n$ in $\Lambda$ to be 
\begin{align*}
    \partial R_n = \{ v\in R_n : \exists u\in \Lambda \setminus R_n, (u,v)\in E(\Lambda)\}.
\end{align*}
The interior of $R_n$ is $R_n^\circ := R_n \setminus \partial R_n$. We take the vertex multiplicities $\mathbf N =(N_v)_{v\in R_n}$ on $R_n$ such that (i) for $v\in R_n^\circ$, $N_v=N$, (ii) for $v\in \partial R_n$, $N_v$ is fixed between $1$ and $N$, (iii) the pair $(R_n,\mathbf N)$ is \textit{feasible}, i.e., there exists an $\mathbf N$ dimer cover of $R_n$, and further for all edges $e$, there exists an $\mathbf N$ dimer cover of $R_n$ containing $e$. 

Finally, to simplify some technicalities, we choose suitable thresholds $\theta_n = O(1/n)$, and define $\rho_{n,N}$ to be the multinomial dimer measure (with uniform edge weights) on $R_n$ with multiplicites $N$ for $v\in R_n^\circ$, and in $[N_v-N \theta_n,N_v+N\theta_n]\cap [1,N]$ for $v\in \partial R_n$. See Section~\ref{sec:background} for further elaboration on the definition of multinomial measures, and the beginning of Section~\ref{sec:ldp} for the definition of $\rho_{n,N}$ in particular.  
}
%%%
Let $\rho_{n,N}$ be the multinomial dimer measure on $R_n$.
We then have
\begin{thm}[Rough statement; see Theorem~\ref{thm:ldp}]\label{thm:ldp_intro}
    Let $(R,b)$, $R_n, \rho_{n,N}$ be as above. In the iterated limit as $N$ and then $n$ go to infinity, the measures $\rho_{n,N}$ satisfy a large deviation principle in the weak topology on flows with good rate function $I_b(\cdot)$ which up to an additive constant is
    $$\int_R \sigma(\omega(x))\, \dd x$$
    if $\omega$ is an asymptotic flow, where $\sigma$ is the surface tension determined by $\Lambda$, and is given explicitly as the Legendre dual of the \textit{free energy} $F$ for $\Lambda$ given in \eqref{Fdef}. If $\omega$ is not an asymptotic flow, then $I_b(\omega) = \infty$. 
\end{thm}
For general background on the theory of large deviations, see e.g.\ \cite{dembo2009large,varadhan2016large}. 
We also show that $\sigma$ is strictly convex (Theorem~\ref{thm:strict_convex}), and hence the rate function has a unique minimizer. This minimizer is the \textit{limit shape} for $(R,b)$. As such, the large deviation principle implies a variational principle, that random samples from these multinomial measures concentrate the limit shape, which is the unique solution of a minimization problem.

\begin{cor}[Rough statement; see Corollary~\ref{cor:limit_shape}]\label{cor:limit_shape_intro}
Fix $(R,b)$, $R_n$, and $\rho_{n,N}$ as above. As $n,N$ go to infinity, random $N$-dimer covers of $R_n$ sampled from $\rho_{n,N}$ concentrate exponentially fast on the limit shape, i.e., the unique asymptotic flow $\omega$ with boundary condition $b$ that minimizes $I_b$. 
\end{cor}
\begin{figure}
    \centering
    \includegraphics[width=0.4\linewidth]{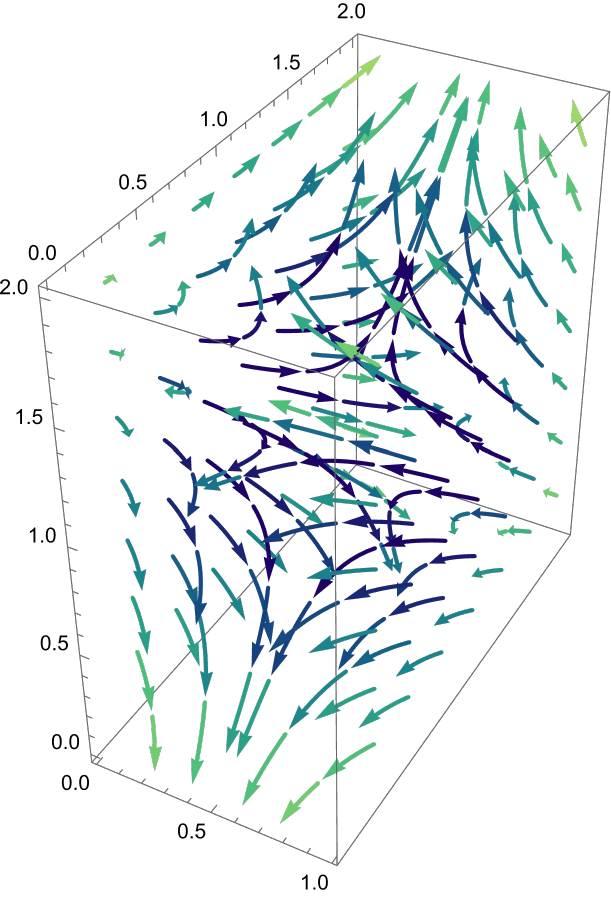}
    \caption{Flow lines for the Aztec cuboid limit shape divergence free flow $\omega = (-\frac{2x}A+1, \frac{2y}B-1, \frac{2z}C-1)$ with $A=1$ and $B=C=2$.}
    \label{fig:cuboid}
\end{figure}
See Figure~\ref{fig:aztecs} for the multinomial limit shape height function on the Aztec diamond, and Figure~\ref{fig:cuboid} for an example of a three-dimensional example called the \textit{Aztec cuboid} in the body-centered cubic ($\bcc)$ lattice. See Section~\ref{sec:limit_shapes} for more on these and other examples. 

The form of Theorem~\ref{thm:ldp_intro} and Corollary~\ref{cor:limit_shape_intro}---that the rate function is the integral of a surface tension function which depends only on slope, and that this rate function has a unique minimizer and hence configurations concentrate on a limit shape---is exactly the same as the form of the large deviation principles for single dimer covers of $\m Z^2$ \cite{cohn2001variational} and of $\m Z^3$ \cite{3Ddimers}. While the microscopic fluctuations of the model are suppressed in the large $N$ limit, this theorem shows that at the ``mesoscopic'' scale, the multinomial dimer model is governed in the same way, where the growth rate of the number of tilings depends on the slope $s\in \mc N(\Lambda)$. 

For the standard dimer model, the surface tension function for dimer tilings of $\m Z^2$ has an explicit formula, also given by taking the Legendre dual of the corresponding free energy on the torus \cite{cohn2001variational,kenyon2006dimers}. For dimer tilings of $\m Z^3$, the surface tension is only known abstractly as the maximum specific entropy of an ergodic Gibbs measure for the model of slope $s$ \cite{3Ddimers}, and correspondingly it is very difficult to say more about this function, or its minimizers (the limit shapes).

The proof of Theorem~\ref{thm:ldp_intro} uses a variety of different tools, which make use of a mixture of techniques, including explicit formulas, and and general combinatorial constructions that hold for the dimer model in any dimension developed in \cite{3Ddimers}, but where the details were worked out only in dimension three. We can substantially simplify these combinatorial arguments using the large $N$ structure, thereby allowing us to make the arguments rigorous in any dimension $d$. In particular:
\begin{itemize}
    \item We use the generating function for multinomial dimer partition functions, found in \cite{KenyonPohoata}, to explicitly compute the free energy $F$ for multinomial dimer measures on the torus in the large $N$ limit (see Theorem~\ref{thm:free_energy_torus}).
    \item We show that the surface tension $\sigma$ is the Legendre dual of $F$ (see Theorem~\ref{thm:legendre_duality}). The general outline of this argument mimics the one in \cite{cohn2001variational}, but the key variance estimate (Lemma~\ref{lem:mult_concentration}) instead uses the large $N$ structure. %This gives us a way to get explicit formulas for the surface tension, see below for examples.
    \item We prove a version of the \textit{patching theorem} for the large $N$ limit. Analogous to the patching theorem of \cite{3Ddimers}, the main tool is Hall's matching theorem. The patching theorem of \cite{3Ddimers} is only proved for $\m Z^3$, is quite technical, and relies on analysis of the geometry certain ``minimal discrete surfaces'' built out of lattice squares. Using the large $N$ structure, we can prove a simpler, unified version of this result that holds for more general lattices and dimensions. See Section~\ref{sec:patching}.
    \item We prove that any asymptotic flow can be approximated by a discrete flow corresponding to an $N$-dimer cover with a new construction, which makes use of the large $N$ structure (see Theorem~\ref{thm:discrete_approximation}). This replaces the very involved ``shining light construction'' of \cite{3Ddimers}, which is specific to $\m Z^3$. 
\end{itemize}

\subsection{Explicit formulas: free energy and surface tension}

\begin{figure}
    \centering
    \includegraphics[width=0.45\linewidth]{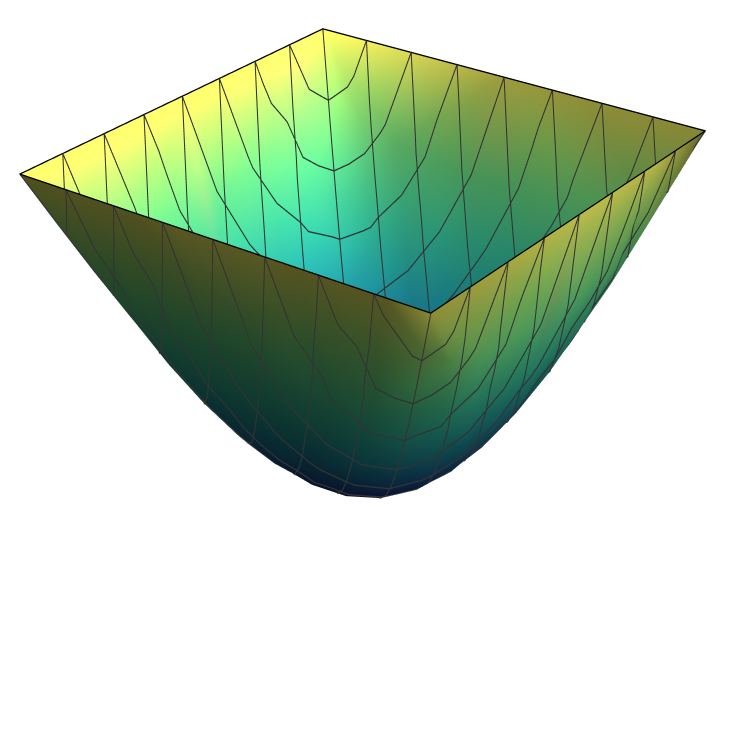} \includegraphics[width=0.45\linewidth]{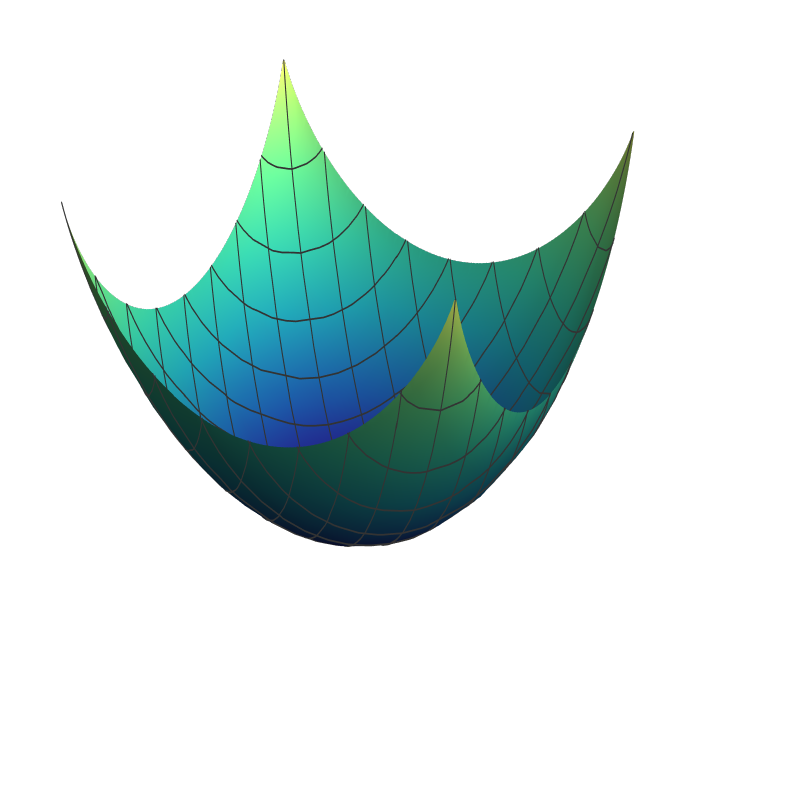}
    \caption{The surface tension functions for $\m Z^2$ for the standard dimer model (left) and the multinomial dimer model in the large $N$ limit (right). Notice that the multinomial surface tension is strictly convex everywhere, whereas the standard dimer one is strictly convex in the interior. Both have maximum value zero.}
    \label{fig:surface_tensions}
\end{figure}

The formulas for the free energy $F$ are especially simple and explicit. Namely, if $\Lambda$ has edges in directions $e_1,\dots,e_D$ at white vertices, then $F:\m R^d\to \m R$ is 
\begin{align}\label{Fdef}
    F(\alpha) = \log \bigg(\sum_{j=1}^D \exp(e_j \cdot \alpha)\bigg).
\end{align}
See Section~\ref{sec:torus_asymptotics}. The \textit{amoeba} $\mc A=\mc A(\Lambda)$ of $\Lambda$ is the largest subset of $\m R^d$ on which $F$ is strictly convex. Unlike for the dimer model in two dimensions, where $\mc A$ is a strict subset, here $\mc A = \m R^d$ for any $d$-dimensional lattice $\Lambda$, and $F$ is analytic. The surface tension $\sigma$ is the Legendre dual of $F$, i.e., 
\begin{align*}
    \sigma(s) = \max_{\alpha\in \m R^d} - F(\alpha) + \sum_{i=1}^d s_i \alpha_i. 
\end{align*}
The domain of the surface tension is the Newton polytope $\mc N(\Lambda)$. 

Unlike the free energy, the surface tension is sometimes simple and sometimes very complicated, depending on the lattice. For $\m Z^2$ it has a simple formula: 

\begin{align}\label{eq:multinomial_Z2}
    \sigma(s_1,s_2) &= \frac{1+2s_1}{2}\log\bigg(\frac{1+2s_1}{2}\bigg) + \frac{1-2s_1}{2}\log\bigg(\frac{1-2s_1}{2}\bigg) \\
    &+ \frac{1+2s_2}{2}\log\bigg(\frac{1+2s_2}{2}\bigg) + \frac{1-2s_2}{2}\log\bigg(\frac{1-2s_2}{2}\bigg),\nonumber
\end{align}
See also Figure~\ref{fig:surface_tensions} for graphs comparing this surface tension with the one for standard dimers on $\m Z^2$. 

For $\m Z^3$, the free energy is simple, while the surface tension is very complicated. The free energy is 
\begin{align*}
    F(\alpha_1,\alpha_2,\alpha_3) = \log (\mathrm{e}^{\alpha_1/2}+\mathrm{e}^{-\alpha_1/2}+\mathrm{e}^{\alpha_2/2}+\mathrm{e}^{-\alpha_2/2}+\mathrm{e}^{\alpha_3/2}+\mathrm{e}^{-\alpha_3/2}).
\end{align*}
In this case the Legendre dual $\sigma$ is valued in the Newton polytope $\mc N(\m Z^3) = \{s\in \m R^3: |s_1|+|s_2|+|s_3| \leq 1/2\}$, which is an octahedron. Restricted to any triangular face of the octahedron, it is equal to the multinomial surface tension for the honeycomb lattice. However, the formula on all of $\mc N(\m Z^3)$ is computationally complicated, involving the root of a degree 8 polynomial. For a numerically-generated discrete approximation of the multinomial limit shape for the \textit{Aztec octahedron} graph in $\m Z^3$, see Figure~\ref{fig:aztec_octahedron}.

However, for the three-dimensional \textit{body-centered cubic lattice} ($\bcc$), the surface tension does have a simple closed formula:
\begin{align*}
    F(\alpha_1,\alpha_2,\alpha_3) &= \log \bigg( (\mathrm{e}^{\alpha_1/2}+\mathrm{e}^{-\alpha_1/2}) (\mathrm{e}^{\alpha_2/2}+\mathrm{e}^{-\alpha_2/2})(\mathrm{e}^{\alpha_3/2}+\mathrm{e}^{-\alpha_3/2})\bigg)\\
    \sigma(s_1,s_2,s_3) &= \sum_{i=1}^3 \frac{1+2s_i}{2}\log\bigg(\frac{1+2s_i}{2}\bigg) + \frac{1-2s_i}{2}\log\bigg(\frac{1-2s_i}{2}\bigg).
\end{align*}

\subsection{Limit shape regularity} 

The limit shape $\omega$, which exists by Corollary~\ref{cor:limit_shape_intro}, is a measurable divergence-free vector field on $\m R^d$. In two dimensions, where the height function is defined, the limit shape height function $h$ is related to $\omega=(\omega_1,\omega_2)$ by $\nabla h = (-\omega_2,\omega_1)$.

One of the iconic properties of limit shapes for standard dimers is the existence of \textit{facets}. In terms of slopes, a facet corresponds to a region where the limit shape divergence-free flow $\omega$ takes values in $\partial \mc N(\Lambda)$, the boundary of the set of allowed slopes (and/or certain integer interior points). In Figure~\ref{fig:aztecs}, the facets are the four regions, one at each corner, where the height function is linear. Single dimer covers of regions like the Aztec diamond in $\m Z^3$ also appear to have facets in simulations, but it is not known how to prove this rigorously (see \cite[Problem 9.0.12]{3Ddimers}). In contrast, we show the following for the multinomial model in the large $N$ limit:
\begin{thm}[See Theorem~\ref{thm:no_facets}]
    For any $(R,b)\subset \m R^d$ flexible (a mild regularity condition, see Definition~\ref{def:flexible}), the $N\to \infty$ multinomial dimer limit shape does not have facets.
\end{thm}
This result holds in any dimension. The proof of this theorem comes from a perturbative argument, and uses the fact that the surface tension gradient blows up at $\partial \mc N(\Lambda).$

This ``surface tension blow up'' can also be used to explain why, e.g.\ for uniform random dimer tilings of $\m Z^2$, we see only facets with \textit{one} color of tile (corresponding to a vertex of $\partial \mc N$), and never two (corresponding to a point on an edge of $\partial \mc N$). (With doubly-periodic weights, for example, the surface tension changes and it is possible to see two colors.) We conjecture that facets exist in this model for all finite $N$, with size shrinking as $N\to \infty$; see see Question~\ref{q:multinomial_finiteN}. For any dimension $d>2$, this existence is in fact a conjecture even for $N=1$. 

In two dimensions, elliptic regularity allows us to conclude the following.
\begin{cor}[See Corollary~\ref{cor:2Dsmooth}]\label{cor:smooth_intro}
    In two dimensions, $N\to \infty$ multinomial dimer limit shapes are smooth. 
\end{cor}
Note that for the standard dimer model, limit shapes are often not smooth, and are not (or are not expected to be) smooth in interesting examples such as the Aztec diamond (or its three-dimensional analogs). Indeed, if the limit shape has facets then it has a singular locus (along the boundary of the facet) where $\omega$ and $\nabla h$ need not even be continuous. For discussion of this, see e.g.\ \cite[Section 1.3]{astala2023dimermodelsconformalstructures}. We expect $N\to \infty$ multinomial limit shapes to be smooth in higher dimensions as well, see Question~\ref{q:always_smooth}.  

\subsection{Euler-Lagrange equations}

Subject to boundary conditions, the limit shape divergence-free flow $\omega$ on $(R,b)$ is the unique minimizer of 
\begin{align*}
    \int_{R}\sigma(\omega(x))\, \dd x
\end{align*}
and so can also be described as the solution to a system of Euler-Lagrange equations. 

For a divergence-free flow in dimension $d$, the Euler-Lagrange equations have the following general form.
\begin{thm}[See Theorem~\ref{thm:EL_equations}]\label{thm:EL_intro}
    The limit shape flow $\omega$ for $(R,b)$ is the solution to the following system of differential equations:
    \begin{itemize}
        \item the divergence equation, 
        \begin{align*}
            \text{div}(\omega(x)) = \sum_{i=1}^d \pd{\omega_i}{x_i}(x) = 0;
        \end{align*}
        \item a collection of ${d\choose 2}$ equations:
        \begin{align}\label{eq:flow_mixedpartials}
            \pd{}{x_i} \pd{}{s_j}\sigma(\omega(x)) - \pd{}{x_j}\pd{}{s_i}\sigma(\omega(x)) =0 
        \end{align}
        for all $i\neq j$ between $1$ and $d$.
    \end{itemize}
\end{thm}
Since $\sigma$ has a known formula, these PDEs are completely explicit. For example, for the $\bcc$ lattice, they are
\begin{align*}
      \frac{4}{1-4\omega_1^2} \pd{\omega_1}{y} - \frac{4}{1-4 \omega_2^2} \pd{\omega_2}{x}= 0, \quad \frac{4}{1-4\omega_1^2} \pd{\omega_1}{z} - \frac{4}{1-4 \omega_3^2} \pd{\omega_3}{x}= 0, \quad \frac{4}{1-4\omega_2^2} \pd{\omega_2}{z} - \frac{4}{1-4 \omega_3^2} \pd{\omega_3}{y}= 0.
\end{align*}
When $d=2$, we can immediately derive the familiar Euler-Lagrange equations for a height function $h$ from Theorem~\ref{thm:EL_equations} using the relationship $\nabla h= (-\omega_2,\omega_1)$. For example, for $\m Z^2$, the Euler-Lagrange equation for the height function is 
$$\frac{4h_{xx}}{1-4h_x^2} + \frac{4h_{yy}}{1-4h_y^2} =0.
$$
A recent work of Prause and the first author \cite{KenyonPrause} gives a general method for describing the solutions to gradient variational problems in terms of their tangent planes; we apply this to 2D multinomial dimers in Section~\ref{sec:tangent_plane}.

We also observe a general principle, stemming from Legendre duality, which transforms the equations in Theorem~\ref{thm:EL_intro} into a new gradient variational problem in any dimension. In general, Legendre duality implies that $\nabla F$ and $\nabla \sigma$ are inverses of each other, and we have the relationships:
\begin{align*}
    \nabla F(a) = s \iff \nabla \sigma(s) = a.
\end{align*}
Hence if we define $\alpha:= \nabla \sigma(\omega)$, where $\omega$ is the limit shape divergence-free flow, then the equations in \eqref{eq:flow_mixedpartials} imply that $\alpha=\nabla H$ is the gradient of a function. We call $\alpha$ the \textit{gauge flow} and $H$ the \textit{limiting gauge function}. The divergence free equation for $\omega$ then corresponds to a PDE for $H$.

\begin{thm}[See Theorem~\ref{thm:EL_equations_gauge}, Corollary~\ref{cor:EL_gauge_cor}]
     Fix $R\subset \m R^d$ and a suitable boundary function $H_b$ on $\partial R$. If $H$ is a weak solution to the differential equation 
    \begin{align}\label{eq:gauge_PDE_intro}
         \text{div} (\nabla F(\nabla H)) = 0
    \end{align}
    with $H\mid_{\partial R} = H_b$, the limiting gauge flow is $\alpha = \nabla H$.
\end{thm}
This relationship between divergence-free and gradient variational problems seems to be well-known to experts in PDEs. The relationship between these equations is purely formal, and also holds for other Legendre dual pairs. It appears that this dual gauge perspective was also studied in the two-dimensional dimer context in \cite{bobenko2024dimersmcurves}, and we suggest that it may be useful in studying the standard dimer model in any dimension as well. 

The gauge PDE for the $\bcc$ lattice in $\R^3$ is 
\begin{align*}
     \frac{\mathrm{e}^{H_{x}}H_{xx}}{(1+\mathrm{e}^{H_{x}})^2} +  \frac{\mathrm{e}^{H_{y}}H_{yy}}{(1+\mathrm{e}^{H_{y}})^2} + \frac{\mathrm{e}^{H_{z}}H_{zz}}{(1+\mathrm{e}^{H_{z}})^2} =0.
\end{align*}
A particular benefit of the gauge PDE is that it is in terms of the free energy $F$ (which has a simple expression (\ref{Fdef})), so it is simple even if the surface tension is not. Unlike the two dimensional height function $h$, which is Lipschitz, the limiting gauge function $H$ does not have any a priori regularity (it follows, however, from Corollary~\ref{cor:smooth_intro} that $H$ is smooth in two dimensions). The limiting gauge function $H$ for the Aztec cuboids $\mathrm{AC}(a,b,c)$, with $(a/n,b/n,c/n)\to (A,B,C)$ as $n\to \infty$, is
\begin{align*}
    H(x,y,z)=\;\;&A \log A- x\log x - (A-x)\log(A-x) \\
    - &B\log B + y \log y + (B-y) \log (B-y) \\
    - &C\log C + z\log z + (C-z)\log(C-z).
\end{align*}
The limit shape for a sequence of Aztec cuboids is shown in Figure~\ref{fig:cuboid}. See Section~\ref{sec:aztec_cuboid} for more details on this example. 

\subsection{Critical gauge}\label{sec:critical_gauge_intro} The critical gauge is a new structure which arises in the large $N$ limit: the behavior of the multinomial dimer model as $N\to \infty$ on a fixed finite graph is captured by the \textit{critical gauge} \cite{KenyonPohoata}.  We show that the critical gauge has a very simple relationship to the limiting gauge function $H$ that we saw exists above. This is how we compute the limit shapes for the Aztec diamond and Aztec cuboid above, which we show in Section~\ref{sec:limit_shapes} have especially clean closed formulas for their critical gauge.

Suppose that $G = (V=B\cup W,E)$ is a (nice) bipartite graph, with all vertex multiplicities $N$. Then there exist \textit{critical gauge functions} $f:W \to \m R_+$ and $g:B\to \m R_+$ such that as $N\to \infty$, random $N$-dimer covers sampled from multinomial dimer measures concentrate on the discrete flow proportional to
\begin{align*}
    \{c(e) = f(u)g(v)\}_{e=(u,v)\in E}.
\end{align*}
These are called the \textit{critical edge weights}. The gauge functions $f,g$ are unique up to the transformation $f\mapsto C f, g\mapsto g/C$ for some constant $C>0$, and are the unique solutions (up to this transformation) to the \textit{critical gauge equations}: 
\begin{align*}
    &\sum_{u\in W:(u,v)\in E} f(u) g(v) = 1 \qquad \forall v\in B\\
    &\sum_{v\in B:(u,v)\in E} f(u) g(v) = 1 \qquad \forall u\in W.
\end{align*}
For more detail, see Section~\ref{sec:KP_results}. The critical gauge can be efficiently computed numerically using Sinkhorn's algorithm \cite{Sinkhorn}. To see this, note that if $G$ is a finite graph that admits $N$-dimer covers, then it has the same number of white and black vertices. We let $A$ denote the corresponding bipartite adjacency matrix of $G$. Solving the critical gauge equations is then the same as finding diagonal matrices $D_1,D_2$ such that $D_1 A D_2$ is bistochastic, i.e., has both row and column sums which are $1$.

Given $A$, Sinkhorn's algorithm finds these matrices $D_1,D_2$ by alternately rescaling all rows of $A$ (i.e., changing all the $f(u)$'s), so that the critical gauge equations are solved at all white vertices, and then rescaling all columns of $A$ (i.e., changing all the $g(v)$'s), breaking the critical gauge equations at white vertices but solving them at the black ones. Viewing these two updates as one step, each step produces a matrix that is closer to being bistochastic, see \cite{KenyonPohoata}.

See Figures~\ref{fig:sinkhorn_application} and \ref{fig:aztec_octahedron} for applications of Sinkhorn's algorithm. Empirically, critical gauge functions are almost always algebraic of high degree. We do however know of a small number of examples where they take rational values, in particular the Aztec diamond and Aztec cuboid. See Section~\ref{sec:limit_shapes}. It would be interesting to understand more generally which graphs have simple descriptions of their critical gauge functions or critical edge weights, see Question~\ref{q:closed_form_critical_gauge}.

We show that given a sequence of graphs $R_n$ as in the large deviation principle, the scaling limit of the critical gauge functions $g_n$ (resp.\ $f_n$) is the limiting gauge function $H$ (resp.\ $-H$), i.e., a solution to \eqref{eq:gauge_PDE_intro}.

\begin{thm}[See Theorem~\ref{thm:critical_gauge_scaling_limit}]
     Fix a $d$-dimensional lattice $\Lambda$ and $R\subset \m R^d$ as before. Suppose that $H$ is $C^1$ and solves the gauge PDE: $\text{div}(\nabla F(\nabla H))=0$, and that $g_n$ is the critical gauge function on black vertices of $R_n$ with limiting multiplicity $1$ everywhere on $R_n$. Further suppose that
    \begin{align*}
        \sup_{u\in B\cap\partial R_n} |(1/n) \log g_n(u) - H(u)| = O(1/n).
    \end{align*}
    We equate $(1/n) \log g_n$ defined on $B\cap R_n$ with its linear interpolation to a continuous function. Then
    \begin{align*}
        \lim_{n\to \infty} (1/n) \log g_n 
    \end{align*}
    exists in the topology induced by the supremum norm on continuous functions and is equal to $H$. Similarly, $\lim_{n\to \infty} (1/n) \log f_n = -H.$
\end{thm}

This theorem establishes rigorously the intuitive result that the scaling limit of the critical edge weights (which capture the large $N$ limiting behavior on each fixed graph) should be the limit shape divergence-free flow. We apply this in Section~\ref{sec:limit_shapes} to compute limit shapes in some examples. This also means that we can use Sinkhorn's algorithm to numerically approximate the limit shape divergence-free flow. See Figures~\ref{fig:sinkhorn_application} and \ref{fig:aztec_octahedron} for discrete, numerical approximations of the divergence free flow limit shape, generated with Sinkhorn's algorithm for a fixed $n$.

The proof of this theorem uses a collection of new ideas, in particular a version of the maximum principle for solutions to the critical gauge equations (which are discrete difference equations), and an expansion which relates the critical gauge equations to the gauge PDE. See Section~\ref{sec:critical_gauge_scaling}.

Since $F$ is analytic for this model, the condition that $H$ is $C^1$ is equivalent to requiring that the limit shape divergence-free flow $\omega$ (that is, $\nabla h$) is continuous, and this holds for the multinomial dimer model in two dimensions by Corollary~\ref{cor:smooth_intro}. While we believe the $C^1$ condition holds in general dimension for the large $N$ limit, this is an open question in dimensions $d>2$; see Question~\ref{q:always_smooth}.

\begin{figure}
    \centering
    \includegraphics[width=0.5\linewidth]{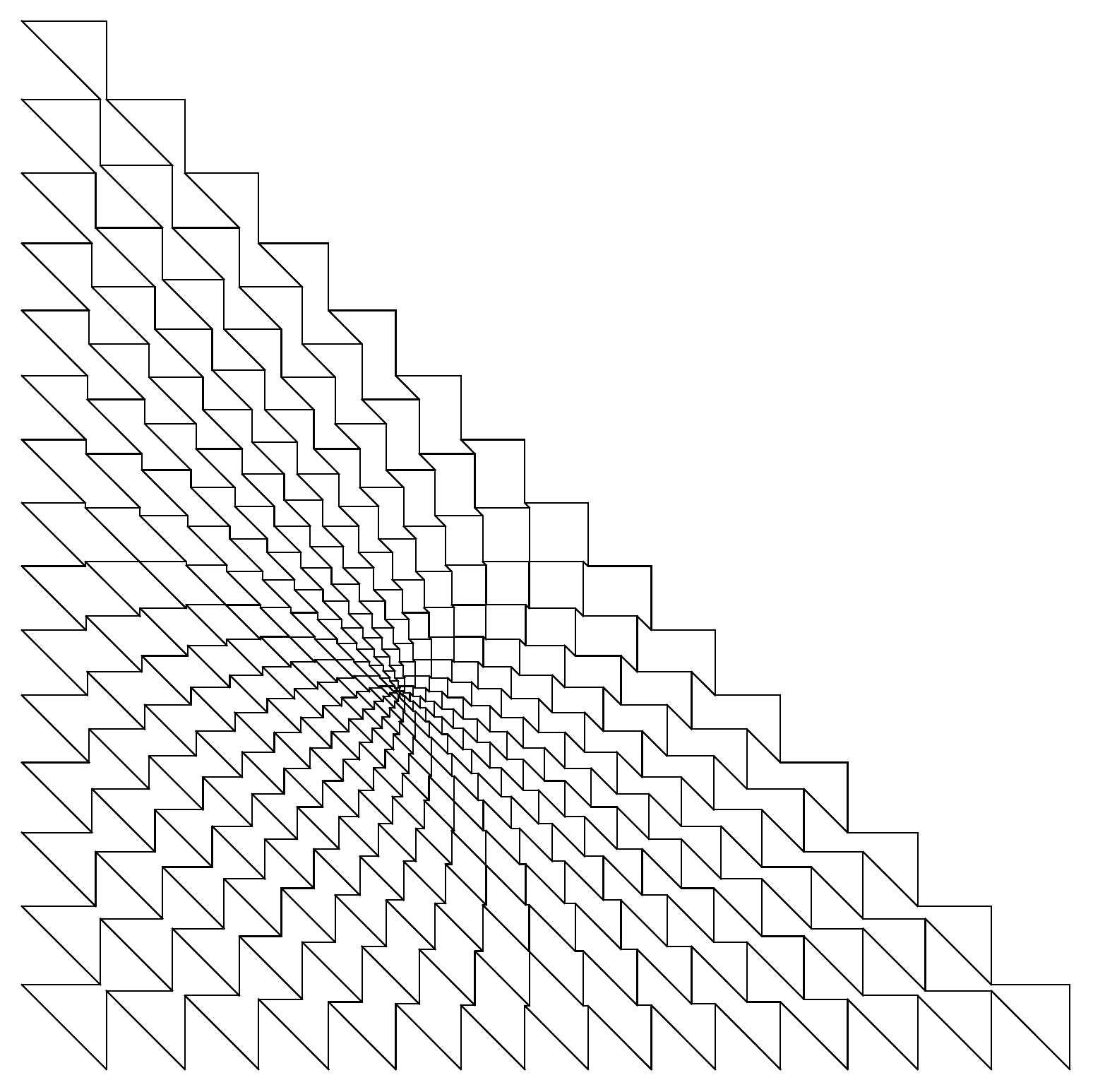}
    \caption{An application of Sinkhorn's algorithm gives a ``discrete approximation'' of the limit shape divergence-free flow for the hexagon graph $G_{n}$ in the honeycomb lattice, $n=30$. Namely, we embed the $G_{30}$ in the Newton polygon $\mc N$ for the lattice, where each vertex $v$ of the graph is embedded at the corresponding average slope $s(v)\in \mc N$, which is, up to a multiplicative constant, $\sum_{e\ni v} c(e) e$. For the hexagon graph, this embedding is a double cover.}
    \label{fig:sinkhorn_application}
\end{figure}

\begin{figure}
    \centering
    \includegraphics[width=0.5\linewidth]{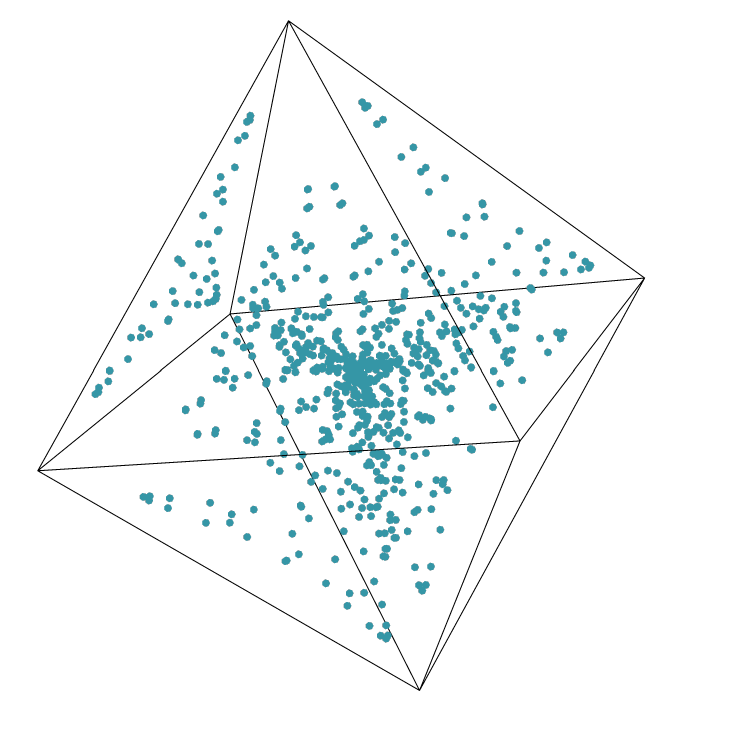}\includegraphics[width=0.5\linewidth]{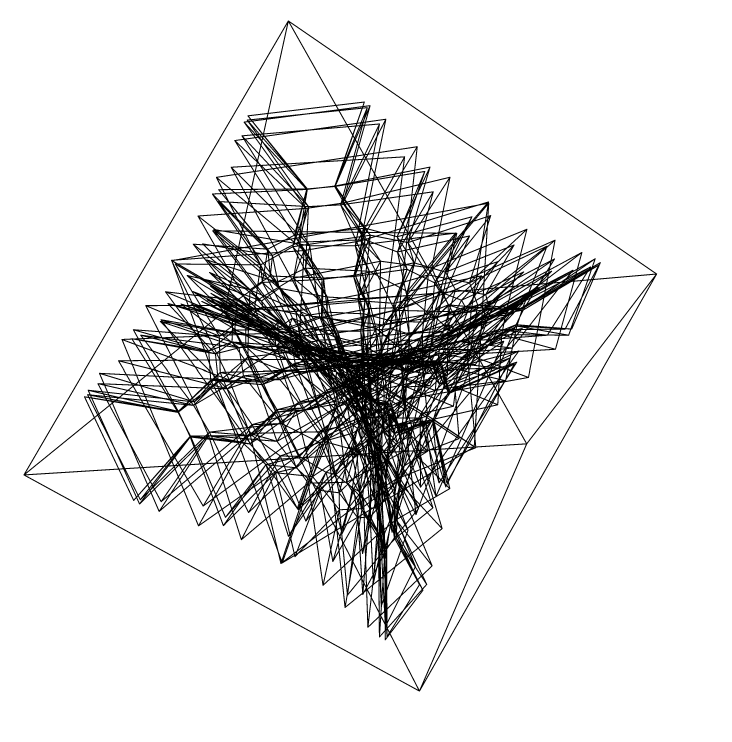}
\caption{The numerically computed exact critical gauge for the octahedron graph in $\m Z^3$, with side length $n=10$,
is used to define an embedding of the graph into the Newton polytope for $\m Z^3$ (which is also the octahedron). The left is the image of the vertices and the right is the image of the edges of the graph in this embedding.
This gives an approximation to the $n=\infty$ limit shape flow.}
    \label{fig:aztec_octahedron}
\end{figure}
\bigskip

\noindent \textbf{Outline of paper}: In Section~\ref{sec:background}, we define the model precisely and summarize some relevant background. In Section~\ref{sec:vector_fields}, we explain the topology used to compare dimer covers of different graphs via discrete flows. In Section~\ref{sec:torus_asymptotics}, we compute the free energy per dimer on the torus, and show that the surface tension is the Legendre dual of the free energy. In Section~\ref{sec:ldp}, we formulate and prove the large deviation principle in the iterated limit as $N$ and then the size of the graph go to infinity, with rate function given by the integral of the surface tension (up to an additive constant), and show that samples from multinomial measures concentrate on a unique limit shape satisfying the boundary conditions. Before that in Section~\ref{sec:patching}, we prove a ``patching theorem'' which is the key tool in the proof of the LDP. In Section~\ref{sec:euler_lagrange}, we formulate the Euler-Lagrange equations for the limit shape, show that multinomial limit shapes do not have facets, and formulate the dual Euler-Lagrange equations solved by the limiting gauge function. In Section~\ref{sec:critical_gauge_scaling}, we show that the scaling limit of the discrete critical gauge functions is the limiting gauge function. In Section~\ref{sec:limit_shapes} we explicitly compute multinomial limit shapes in some examples using a few methods, and in Section~\ref{sec:questions} we state some open questions and further directions. 

\bigskip

\noindent\textbf{Acknowledgments:} We thank Nishant Chandgotia, Scott Sheffield, Istvan Prause, Charlie Smart, and Matthew Nicoletti for helpful conversations and Felix Otto for a comment after a talk at the Princeton probability seminar, which motivated further understanding of the dual Euler-Lagrange equations. R.K. is supported by the Simons Foundation grant 327929. C.W. is supported by NSF grant DMS-2401750.

\section{Background and preliminaries}\label{sec:background}

The multinomial dimer model is a generalization of the dimer model and a special case of the multinomial tiling model introduced in \cite{KenyonPohoata}. 

\subsection{Definition of the model}\label{sec:deifnition_of_model}

\subsubsection{The dimer model.} Let $G = (V,E)$ be a finite bipartite graph. A \emph{dimer cover} or \emph{perfect matching} $\tau$ of $G$ is a collection of edges $\tau \subset E$ so that every vertex $v\in V$ is contained in exactly one edge $e\in \tau$. We let $\Omega_1(G)$
denote the set of dimer covers of $G$.

Given a collection of positive edge weights $(w_e)_{e\in E}$, the \emph{Boltzmann measure} is the 
associated probability measure on $\Omega_1(G)$ where the probability of a 
dimer cover $\tau$ is proportional to the product of 
its edge weights:
\begin{align*}
    \m P(\tau) =\frac1Z \prod_{e\in \tau} w_e.
\end{align*}
The normalizing constant $Z = \sum_{\tau} \prod_{e\in \tau} w_e$ is called the \emph{partition function}. 

\subsubsection{Blow-up graph.} 
Let $\NN = (N_v)_{v\in V}$ be a collection of positive integer \emph{vertex multiplicities} on the vertices of $G$. 

The \emph{$\NN$-fold blow-up} of $G$, denoted $\GN$, is the graph where each vertex $v\in V$ is replaced by $N_v$ vertices, called lifts of $v$, and each edge $e=(u,v)\in E$ is replaced by the complete bipartite graph on lifts of $u,v$. 
We denote the vertices and edges of $\GN$ by $(V_{\NN},E_{\NN})$. 
If $\NN\equiv 1$, then the blow-up graph is just the original graph $G$. 
Let $\pi: \GN\to G$ be the projection map.

\subsubsection{The multinomial dimer model.} The multinomial tiling model, introduced in \cite{KenyonPohoata}, studies the behavior of tilings of blow-up graphs $\GN$. The \emph{multinomial dimer model} is the special case where tiles are dimers, that is, the multinomial dimer model is the dimer model on $\GN$.

An \emph{$\NN$-dimer cover} of $G$ is a function $M:E\to\{0,1,2,\dots\}$ summing to $N_v$ at each vertex $v$,
that is, such that $\sum_{u:u\sim v}M_{(u,v)}=N_v$. We let $\Omega_{\NN}$ denote the set of $\NN$-dimer covers of $G$.
Note that when $\NN\equiv 1$, then $\Omega_{\NN}(G)=\Omega_1(G)$ is the set of dimer covers of $G$. 
If $\tau$ is a dimer cover of $\GN$, let $M := \pi(\tau)\in \Omega_{\NN}$ be its projection: for each edge $e\in E$, the multiplicity $M_e$ is the number of lifts of $e$ in $\tau$. We fix the notational convention throughout that $\tau$ is a dimer cover of a blow-up graph $G_{\NN}$, and 
$M\in\Omega_{\NN}$ is its projection $M=\pi(\tau)$. Note that while $\tau$ determines $M$, $M$ does not determine $\tau$. 

Given positive edge weights $(w_e)_{e\in E}$, we endow $G_{\NN}$ with edge weights such that all lifts of $e$ have weight $w_e$. This defines the Boltzmann measure on dimer covers of $G_{\NN}$ where the probability of a tiling $\tau$ is given by
\begin{align}\label{eq:lifted_measures}
   \m P(\tau) =\frac1Z \prod_{e\in E} w_e^{M_e}
\end{align}
where $Z$ is the partition function.

\begin{definition}\label{def:multinomial_measures}
Let $G = (V,E)$ be a bipartite graph with 
 vertex multiplicities $\NN = (N_v)_{v\in V}$. Given edge weights $(w_e)_{e\in E}$, the corresponding \textit{multinomial dimer measure} is the probability measure on $\Omega_{\NN}$ where
\begin{align*}
    \m P(M) = \frac{\left|\pi^{-1}(M)\right|}{Z} \prod_{e\in M} w_e^{M_e}
\end{align*}
where $Z=Z_{G,\NN}$ is the same partition function, given below. In other words, the multinomial dimer measure is the measure on $\Omega_{\NN}$ which is the projection of the Boltzmann measure on dimer covers of $G_{\NN}$ defined in \eqref{eq:lifted_measures}. We also refer to the measure in \eqref{eq:lifted_measures} as the \textit{lifted} version of the multinomial dimer measure. 
\end{definition}
The partition function of either measure is
\be
 \label{Zinit}   Z_{G,\NN}({w}) = \sum_{M\in \Omega_{\NN}(G)} |\pi^{-1}(M)|\prod_{e\in E} w_e^{M_e}.
\ee
If $\NN \equiv 1$, this measures is just the Boltzmann measure on $\Omega_1(G)$, the dimer covers $G$. 
For any $\NN$, if $w_e\equiv 1$, then the lifted measure is the uniform measure on dimer covers of $\GN$. However it is important to note that in this case the projection is \emph{not} the
uniform measure on $\NN$-dimer covers of $G$. The probability of an $\NN$-dimer cover $M$ is proportional to
$|\pi^{-1}(M)|$, which depends on the edge multiplicities $M_e$.  
In general, the number of lifts of $M$ is given by
\begin{align*}
    |\pi^{-1}(M)| = \frac{\prod_{v\in V}N_v!}{\prod_{e\in E} M_e!}.
\end{align*}
See Figure~\ref{fig:square_example} for an example.
(To see this, note that at each vertex $v$ with edges $e_1,\dots,e_d$ we must split the $N_v$ vertices over $v$ into subsets of sizes $M_{e_1},\dots, M_{e_d}$, and then for each edge $e$ pick a pairing of the $M_e$ vertices on each side.
This results in, for each vertex, a multinomial coefficient $\binom{N_v}{M_{e_1},\dots,M_{e_d}}$ then for each edge $e$
a set of $M_e!$ choices.)

  \begin{figure}
      \centering
     \includegraphics[width=0.3\linewidth]{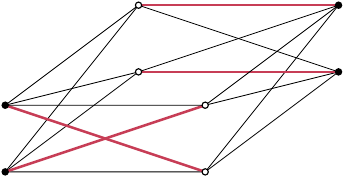}\hspace{1cm} \includegraphics[width=0.6\linewidth]{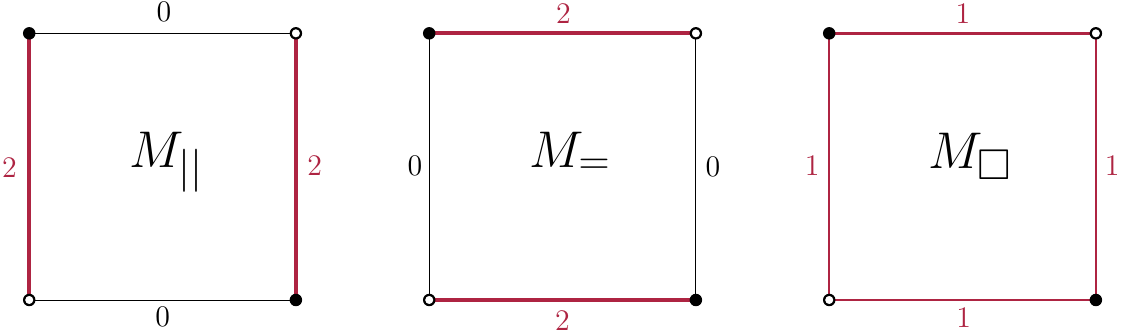}
      \caption{Here $G$ is a square and $\NN \equiv 2$. On the left is the blow-up graph $G_\mathbf{2}$ and a dimer cover of it. On the right are the three double dimer covers of $G$. There are $4,4,16$ ways to lift $M_{||},M_{=},M_{\square}$ respectively to dimer covers of $G_{\mathbf{2}}$. For the multinomial measure with uniform edge weights $w_e\equiv 1$, they have respective probabilities $\frac1{24}\{4,4,16\}=\{\frac16,\frac16,\frac23\}$ of occurring.}
      \label{fig:square_example}
  \end{figure}

\subsection{Lattices and multinomial dimers on subgraphs}\label{sec:subgraphs_boundaryconditions}

Let $\Lambda=(V=B\cup W,E)$ be a $D$-regular bipartite graph embedded in $\R^d$, with $\Z^d$ acting by translation isomorphisms, transitively on $W$ and on $B$. In other words we assume $W=\Z^d$ and $B=\Z^d+v_0$ where $v_0$ is some vector in $[0,1]^d$ minus its vertices. Edges connect $w$ to $w+e_1,\dots,w+e_D$ if $w\in W$, and therefore connect $b$ to $b-e_1,\dots,b-e_D$. We call such a graph $\Lambda$ a \emph{lattice} in $\R^d$. An edge $e$ connecting $w$ to $w+e_i$ is said to be \emph{of type $e_i$}. 

We assume that $\Lambda$ is connected (which implies that $\{e_1,\dots,e_D\}$ span $\R^d$) and \emph{harmonically embedded}, that is, $\sum_{i=1}^D e_i = 0$. We call $W$ the ``white vertices'' and $B$ the ``black vertices'' of the lattice. We also consider lattices $\Lambda$ which are a linear image of ones of the form described above. As a technical point, for simplicity, we assume that the edges of $\Lambda$ are embedded so that they are monotone and {don't contain other vertices of $\Lambda$ in their interior} (for the examples we consider in Section~\ref{sec:limit_shapes} this simply holds when all the edges are embedded as straight lines.) We also assume that $\Lambda$ is symmetric under reflection. This condition is not necessary, but simplifies one part of our argument, see Remark~\ref{rem:reflection}. 

Our results hold in any dimension $d$. They should also hold more generally for periodic graphs with larger fundamental domains, but we restrict to the family of lattices above to simplify the exposition. Examples of lattices satisfying the above conditions are $\Z^d$ for any $d$ and the body-centered cubic lattice ($\bcc$) in three dimensions. (See Section~\ref{sec:aztec_cuboid} for an example in the $\bcc$ lattice.)

If $G\subset \Lambda$ is a finite subgraph, we define the \emph{boundary} $\partial G$ of $G$ in $\Lambda$ to be its vertices that have neighbors in $\Lambda\setminus G$:
\begin{align*}
    \partial G := \{v\in G : \exists\, u\in \Lambda\setminus G, (u,v) \in E(\Lambda)\}. 
\end{align*}
We call $G^\circ = G\setminus \partial G$ the \textit{interior} of $G$. In this paper we study multinomial dimers on subgraphs $G\subset \Lambda$ where the multiplicity function $\NN$ is constant in $G^\circ$ but allowed to vary on the boundary:
\begin{itemize}
    \item for all $v\in G^\circ$, $N_v = N$;
    \item for $v\in \partial G$, $N_v\in \{1,\dots, N\}$;
    \item the collection of multiplicities $\NN = (N_v)_{v\in V}$ is \textit{feasible}, i.e. (i) $G$ has an $\NN$-dimer cover and (ii) for every each $e\in E$, there exists an $\NN$-dimer cover $M$ of $G$ with $M_e >0$. 
\end{itemize}
Note that the feasible condition depends only on $(G,\NN)$, and is independent of the edge weights. See Figure~\ref{fig:not_feasible} for an example of $(G,\NN)$ which is not feasible. 

\begin{figure}
    \centering
    \includegraphics[width=0.6\linewidth]{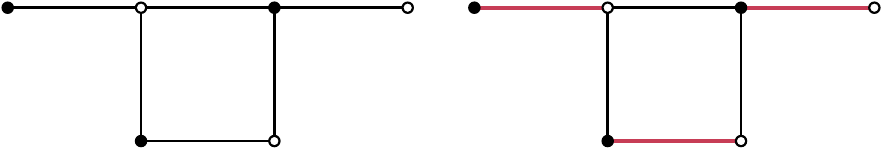}
    \caption{A subgraph $G\subset \m Z^2$ where for any $N$ there exists an $\NN\equiv N$ dimer cover, but $(G,\NN)$ is not feasible because not all edges can occur.}
    \label{fig:not_feasible}
\end{figure}

\subsection{Asymptotics as $\NN\to \infty$ on a fixed graph}\label{sec:KP_results}

\subsubsection{Limiting boundary conditions}

Fixing $G\subset \Lambda$ and letting $\NN\to\infty$, 
the \textit{limiting covering multiplicity} $\beta_v$ of a vertex $v\in G$ is defined to be $$\beta_v := \lim_{N\to \infty} N_v/N$$
which we assume exists (we only take limits $\NN\to\infty$ where this limit exists).
This multiplicity is $1$ for all $v\in G^\circ$ and has $\beta_v \in [0,1]$ for $v\in \partial G$. We refer to $(\beta_v)_{v\in \partial G}$ as the \textit{limiting boundary condition}. We say that $(\beta_v)_{v\in G}$ is feasible if it can be approximated by $(N_v)_{v\in V}$ feasible.

\subsubsection{Asymptotics of multinomial tilings}

Here we summarize some key definitions and results from \cite{KenyonPohoata} about the multinomial dimer model as $\NN\to \infty$ on a fixed finite graph.

The \textit{tiling polynomial} $P$ with edge weights $w = (w_e)_{e\in E}$ is a function of vertex weights $\mathbf x = (x_v)_{v\in V}$ given by
\begin{align*}
    P(\mathbf x) = \sum_{(u,v)=e\in E} w_e x_u x_v.
\end{align*}
As a function of $\mathbf x$, we can use this to write a generating series for the multinomial partition functions.
\begin{thm}[{\cite[Theorem 2.1]{KenyonPohoata}}]\label{thm:KP_generating_function}
Let ${\mathbf x}^{\NN} = \prod_{v\in V} x_v^{N_v}$ and let $\NN! = \prod_{v\in V} N_v!$. Then 
\begin{align*}
    \exp(P(\mathbf x)) = \sum_{\NN\geq 0} Z_{G,\NN}(w) \frac{\mathbf x^{\NN}}{\NN!}.
\end{align*}
\end{thm}
Two weight functions $w,w'$ are \textit{gauge equivalent} if there exists a function $\mathbf x:V\to \m R_+$ such that for all edges $(u,v) = e\in E$, we have 
\begin{align*}
    w_e' = w_e x_u x_v.
\end{align*}
Gauge equivalence is a central topic of \cite{KenyonPohoata}, and will play an important role in this paper as well. If $G$ and $\NN$ are fixed and $ w, w'$ are gauge equivalent, then they induce the same probability measure on $\Omega_{\NN}(G)$.

To study asymptotics, we must fix how $\NN\to \infty$. Using the normalization in \cite{KenyonPohoata}, we let $K = \sum_{v\in V} N_v/2$ be the total number of dimers in an $\NN$-dimer cover. We take the limit $\NN\to\infty$ in such a way
that $\lim N_v/K$ exists for each $v$; we denote
$\alpha_v := \lim N_v/K\in [0,1]$. Here $\alpha_v$ is the asymptotic fraction of dimers covering $v$. We say that $(\alpha_v)_{v\in V}$ is \emph{feasible} if it is approximated by feasible $(N_v)_{v\in V}$. This differs from the limiting covering multiplicity $\beta_v$ defined in the previous section by a multiplicative constant; see Remark~\ref{rem:N_vs_K} below. 

\begin{thm}[See {\cite[Theorem 3.1]{KenyonPohoata}}]\label{thm:critical_weights_unique}
    Fix a finite graph $G = (V,E)$, a weight function $w: E \to \m R_+$, $\alpha_v \in [0,1]$ feasible and a sequence of multiplicites $\NN$ with $N_v/K\to \alpha_v$ as $\NN\to \infty$ for all $v\in V$. Then there exist a \textit{critical weight function} $\mathbf x:V\to \m R_+$ such that for each $v\in V$,
\be\label{eq:critical_weights}
\sum_{u:(u,v)=e\in E} w_e x_u x_v  = \alpha_v P(\mathbf x).
\ee
Further, the critical weights $\mathbf x$ are unique up to the operation of multiplying by a constant on all black vertices,
and the inverse of this constant on white vertices. In particular, the \textit{critical edge weights} $w_e x_u x_v$ for $(u,v)\in E$ are unique.
\end{thm}

\begin{rem}\label{rem:N_vs_K}
    When $G$ is a lattice subgraph, and all interior vertices have the same multiplicity $N$, then $\alpha_v = \lim N_v/K$ and $\beta_v = \lim N_v / N$ differ by a constant which is independent of $v$ (approximately $|V|$). As such, the critical gauge equations all change by a global multiplicative constant. This can be absorbed as changing the edge weights by a gauge, and therefore does not change the critical gauge. 
\end{rem}
Since $G=(V,E)$ is bipartite with $V=W\cup B$, we can view the critical gauge function $\mathbf x : V\to \m R_+$ as a pair of functions $f: W\to \m R_+$ and $g: B\to \m R_+$, where $f,g$ are determined up to multiplying $f$ by a constant $C>0$ and $g$ by $1/C$. These functions $f,g$ are the \textit{critical gauge functions} on white and black vertices respectively. These will be a tool in Section~\ref{sec:torus_asymptotics}, and we study their scaling limits on a growing sequence of graphs in Section~\ref{sec:critical_gauge_scaling}. See Section~\ref{sec:limit_shapes} for explicit examples.

By dividing $M_e$ by $K$, we can view multinomial dimer measures as probability measures on the compact space $[0,1]^E$ for any $\NN$. As a corollary of Theorem~\ref{thm:critical_weights_unique}, the function $(M_e/K)_{e\in E}$ satisfies a ``limit shape theorem'' as $\NN\to \infty$, where the limit shape is proportional to the critical edge weights.

\begin{thm}[{See \cite[Corollary 3.2]{KenyonPohoata}}]\label{thm:largeN_concentration}
    Fix a finite graph $G = (V,E)$ and edge weights $(w_e)_{e\in E}$. Suppose that $\NN = (N_v)_{v\in V}$ are feasible vertex multiplicities such that $N_v/K\to \alpha_v$ as $\NN\to \infty$, and let $(c_e)_{e\in E}$ denote the critical edge weights. The configuration $\{\frac{M_e}{K}\}_{e\in E}$ concentrates on its expectation as $\NN\to\infty$ and the limit is proportional to the fractional matching $\{c_e\}_{e\in E}\subset [0,1]^E$.
\end{thm}

\subsection{Discrete vector fields}\label{sec:discrete_flow}

In two dimensions, there is a correspondence between single dimer covers and Lipschitz {height functions}, up to an additive constant \cite{Thurston}. In higher dimensions, there is no longer a correspondence between single dimer covers and height functions, but there is a related correspondence with \textit{discrete flows} $\omega$ which works in any dimension, as used in \cite{3Ddimers}. There are analogous correspondences for $\NN$-dimer covers. We use the correspondence with discrete flows to compare dimer covers of different graphs; this is explained in Section~\ref{sec:vector_fields}.

A \emph{(discrete) flow} $\omega$ is a function on the oriented edges of a graph such that $\omega(-e) = -\omega(e)$, where $-e$ denotes $e$ with reversed orientation. The \textit{divergence} of $\omega$ at a vertex $v\in V$ is $\text{div}\,\omega (v)=\sum_{e\ni v} \omega(e)$, where the edges $e$ in the sum are all oriented out of $v$. 

 As in Section~\ref{sec:subgraphs_boundaryconditions},  let $G = (V=W\cup B,E)\subset \Lambda$ be a subgraph of the lattice $\Lambda$ and let $\NN$ be a feasible multiplicity function equal to $N$ for all $v\in G^\circ$ and $N_v\in \{1,\dots,N\}$ for $v\in \partial G$. 

Given an $\NN$-dimer cover $M$, we define the flow $\omega_M(e)$ as a function of the oriented edge $e$ by rescaling
$M$ by $N$:
\begin{align}\label{eq:tiling_flow}
    \omega_M(e) = \pm\frac{M_e}{N}
\end{align}
where the sign is positive if $e$ is oriented white-to-black and negative otherwise. The divergences of $\omega_M$ are: 
\begin{align}
 \text{div}\, \omega_M(v) = \sum_{e\ni v} \omega_M(e) = \begin{cases} \pm 1 \qquad &\forall \; v\in G^\circ\\
  \pm N_v/N \qquad &\forall \; v\in \partial G \end{cases}
\end{align}
where the sign is positive if $v\in W$ and negative if $v\in B$. In particular, after fixing the vertex multiplicities $\NN$, the divergences of $\omega_M$ depend only on the parity of the vertex. 

\begin{rem}
    A single dimer cover $\tau$ corresponds to a flow $\omega_\tau$ which takes values $\{0,1\}$ on all white-to-black edges, and has $\pm 1$ divergences. Taking $\NN\to \infty$ can be seen as ``weakening'' the constraints of the dimer model by allowing the discrete flow to take continuous values in $[0,1]$ on edges, but still requiring it to have fixed divergences. 
\end{rem}

For any two $\NN$-dimer covers $M_1,M_2$, the flow $\omega_{M_1} - \omega_{M_2}$ is divergence-free. Alternatively, subtracting just a constant, we notice the flow $\overline{\omega}_M(e) = \omega_M(e)-1/D$ is divergence-free. In Section~\ref{sec:vector_fields}, we will see that the \textit{scaling limits} as $n\to \infty$ of all of these discrete flows on $\frac{1}{n}\Lambda$ are measurable divergence-free vector fields (Theorem~\ref{thm:scaling_limits_asymptotic}).

\begin{rem}[Relationship to height functions]
    In two dimensions, we can derive the correspondence (up to an additive constant) between an $\NN$-dimer cover $M$ and a height function $h_M$ from the correspondence with flows. This construction can be thought of as using $\overline{\omega}_M$. Let $\Lambda$ be a planar bipartite lattice, and let $\Lambda^*$ denote the dual lattice. The height function $h_M:\Lambda^*\to \m R$ is determined up to an additive constant by its differences across dual edges: for a neighbor $w^*$ connected to $v^*$ by a dual edge $e^*$, the height difference is
\begin{equation}\label{eq:height_diffs}
    h_M(w^*)-h_M(v^*) = \pm \bigg[\frac{M_e}{N} - \frac{1}{D}\bigg],
\end{equation}
with positive sign if the dual edge $e^*=(v^*,w^*)$ oriented from $v^*$ to $w^*$ has a white vertex on the left and negative sign otherwise. We note that this defines a function since summing these differences around any loop in $\Lambda^*$ gives $0$. 

The height function exists only in two dimensions for the topological reason that in two dimensions a divergence-free vector field (i.e., $\overline{\omega}_M(e):= \omega_M(e) - 1/D$) is dual---by rotating by $90^\circ$---to a curl-free vector field, and any curl-free vector field on the plane is the gradient of a function, hence of the form $\nabla h$. 
\end{rem}

  \subsection{Slope and Newton polytope}\label{sec:slope}

The main setting of this paper will be an iterated limit, where we take both the multiplicity $\NN\to \infty$ and the lattice spacing ${1}/{n}\to 0$. Let $R\subset \m R^d$ be a compact connected domain with piecewise smooth boundary (domain means that $R$ is the closure of its interior, denoted $R^\circ$). Fix a $d$-dimensional lattice $\Lambda$, and consider any sequence of regions $R_n\subset \frac{1}{n}\Lambda$ approximating $R$ in Hausdorff distance. We define the boundary of $R_n$ to be 
\begin{align*}
    \partial R_n = \{v\in R_n : \exists\, u\in\tfrac1{n}\Lambda\setminus R_n \text{ such that }(u,v)\in E(\tfrac1{n}\Lambda)\}.
\end{align*}
We refer to $R_n$ as a \textit{lattice approximation} of $R$. We only work with regions $R_n$ that are feasible (i.e., tileable with the given multiplicities).

We compare dimer covers of the graphs $R_n$ via the corresponding flows. For each $n$, we set feasible boundary conditions $(\beta_v^n)_{v\in \partial R_n}$ as $\NN\to \infty$, and study the behavior of a sequence of multinomial measures $(R_n,\NN)$ which are feasible and converge to $\beta_v^n$ as $\NN\to \infty$. 

For a dimer cover $M$, we define the \textit{slope at white vertices} or just \textit{slope} $s_M:W\to \m R^d$ by
\begin{align}\label{eq:def_mean_current}
    s_M(u) = \sum_{e\ni u} \omega_M(e) e,
\end{align}
where $e$ is oriented white-to-black and interpreted as a vector. 
At all white vertices, the incident edge directions are $e_1,\dots,e_D$. As such, $s_M$ is valued in the set
\begin{align*}
    \mc N(\Lambda) = \text{convex hull}\{e_1,\dots,e_D\}\subset \m R^d.
\end{align*}
This is the \textit{Newton polytope} for the lattice $\Lambda$\footnote{This differs from the usual definition of the Newton polytope using the characteristic polynomial, but these definitions are equivalent for the lattices $\Lambda$ we consider.}. We call a point $s\in \mc N(\Lambda)$ a \textit{slope}. The \textit{slope at black vertices} can be defined similarly, but the vector directions incident to a black vertex are $-e_1,\dots, -e_D$. 

\begin{rem}
    It is more standard in the 2D dimer literature for slope to mean the gradient of the height function. This differs from our convention by a rotation. If $s=(s_1,s_2)$ is the slope at white vertices, the gradient of the height function is $(-s_2,s_1)$. Our definition of slope is called \textit{mean current} in \cite{3Ddimers}. 
\end{rem}

\section{Comparing dimer covers via vector fields}\label{sec:vector_fields}

To compare $\NN$- and $\NN'$-dimer covers $M,M'$ of the same finite graph $G = (V,E)$, we can use the sup norm on
the edges $E$ of $G$: 
    \begin{align}
        |M-M'|_{\infty,E} :=\sup_{e\in E} |\omega_M(e) - \omega_{M'}(e)| = \sup_{e\in E} \left|\frac{M_e}N-\frac{M'_e}{N'}\right|.
    \end{align}
    This is an extremely natural way to compare dimer covers and is what is used in \cite{KenyonPohoata}.

    However, in this paper we will also be interested in comparing $\NN$-dimer covers of different graphs, in particular dimers covers of a growing sequence of rescaled graphs $R_n\subset \frac{1}{n} \Lambda$ approximating a compact region in $R\subset \R^d$ as above. To do this we use the \textit{weak topology on flows}, which is a topology on the flows corresponding to dimer covers described in Section~\ref{sec:discrete_flow}. This is the same topology used to study dimer covers of $\Z^3$ described in \cite[Section 5]{3Ddimers}. In adapting this to lattices other than $\Z^3$, we also streamline its description and provide more intuition for why this topology is natural.
    
    \subsection{Scaling limits of dimer covers}

Let $R\subset \m R^d$ be a compact, connected domain with piecewise smooth boundary and let $\Lambda$ be a $d$-dimensional lattice as in Section~\ref{sec:subgraphs_boundaryconditions}. Let $R_n\subset \frac{1}{n}\Lambda$ be a sequence of lattice approximation of $R$ as above. Recall that by convention, we fix the initial scaling of $\Lambda$ so that the Voronoi tile of each vertex in $\Lambda$ has volume $1/2$. 

If $M$ is an $N$-dimer cover of $R_n$, it has an associated flow $\omega_M:E\to [0,1]$, where we view $E$ as the set of edges oriented white-to-black. We study the asymptotics of the corresponding discrete vector field
\begin{align*}
        \omega_M(e) e,
\end{align*}
where $e$ is viewed as the vector oriented white-to-black.
Further associated to $M$ is its slope per white vertex function 
$$s_M(u) = \sum_{e\ni u} \omega_M(e) e: R_n\cap W\to \mc N(\Lambda).$$    
We also consider the slope per black vertex function, denoted $\tilde{s}_M$, which is given by 
    \begin{align*}
        \tilde{s}_M(v) = \sum_{e\ni v} \omega_M(e) (-e)\qquad v\in R_n \cap B.
    \end{align*}
    Here we still view $e$ as an edge oriented white-to-black, hence the minus sign comes from the fact that the edge directions at a black vertex are $-e_1,\dots,-e_D$. 
    
The objects $\omega_M(e)e, s_M, \tilde s_M$ are \emph{discrete vector fields}, in the sense that they are vector-valued functions defined on lattice points or edges. 
However for convenience we henceforth refer to them as \emph{discrete flows}.

The \textit{Wasserstein metric on flows} that we discuss next metrizes the weak convergence of the above discrete flows in the $n\to \infty$ limit. This topology was introduced in \cite{3Ddimers} to study dimer covers of $\m Z^3$. In this topology, we will show (see Theorem~\ref{thm:scaling_limits_asymptotic}) that the $n\to \infty$ scaling limits of $\omega_M(e)e$, $s_M$ and $-\tilde{s}_M$ are the same and are \textit{asymptotic flows} for $\Lambda$.

\begin{definition}
An \textit{asymptotic flow} on $R$ for $\Lambda$ is a measurable vector field $\omega$ which is (i) supported in $R$, (ii) divergence-free as a distribution, and (iii) valued in the Newton polytope $\mc N(\Lambda)$. We denote the space of asymptotic flows on $R$ for $\Lambda$ by $\text{AF}(\Lambda,R)$.
\end{definition}

\subsection{Weak topology on flows}

The key aspect of the topology we define here is that it is a metrizable topology that captures weak convergence of the slope functions in the scaling limit. We can view the slope functions $s_n$ or $\tilde{s}_n$, discrete flows $\omega_n(e)e$, and asymptotic flows $\omega$ as elements of the same space by viewing them as vector-valued distributions. Through this, there is a natural correspondence between these and their dual measures. Let $\eta_1,\dots, \eta_d$ denote unit basis vectors for $\m R^d$. Let $\dd x = \dd x_1\dots \dd x_d$. The dual measures are as follows:
\begin{itemize}
\item If $\omega=(\omega_1,\dots,\omega_d)$ is an asymptotic flow, then the dual measures are $\mu_i(x) = \omega_i(x) \, \dd x$.
\item If $\omega_n(e)e$ is a discrete flow, on edges $e\in E\subset \frac{1}{n} \Lambda$, then its dual measures are 
$$\mu_i(x) = \sum_{e\in E} \omega_n(e) \frac{\dd\lambda^i_{e}(x)}{n^d},$$ 
where $\dd\lambda^i_{e}$ is supported on $e$ and projects to Lebesgue measure on the projection of $e$ to the $i$th coordinate.
\item If $s_n:V\cap W\to \mc N(\Lambda)$ is a slope function on white vertices, then its dual measures 
are 
$$\mu_i(x) = \sum_{v\in V\cap W} \langle s_n(v), \eta_i\rangle \frac{\mathds{1}_{v}(x) \, \dd x}{n^d},$$ 
where $\mathds{1}_{v}(x) \, \dd x$ denotes the unit point mass at $v$. The slope function on black vertices has an analogous dual which is a sum of point masses at black vertices.
    \end{itemize}

{ \begin{rem}
        Recall that as a technicality, we assumed that the edges of $\Lambda$ are {monotone and embedded so that they don't cross other edges, and {don't pass though other vertices} of the lattice outside of their endpoints}. This doesn't change the limiting slope, and the reason we make this convention is that it means the measures $\vec{\mu}$ corresponding to a discrete flow $\omega_n$ determine $\omega_n$. This means that our use of $e$ in some places to denote the corresponding vector (which is just determined by the endpoints of the edge) is a slight abuse of notation. 
    \end{rem}}
    We refer to the measures that can occur above (as a coordinate component of an asymptotic flow, discrete flow, or slope function) as \textit{component measures}. A sequence of vector-valued measures $\vec{\mu}^k = (\mu_1^k,\dots,\mu_d^k)$ converges weakly to $\vec{\mu} = (\mu_1,\dots,\mu_d)$ if $\mu_i^k$ converges weakly for each $i=1,\dots,d$. We will see below (Proposition~\ref{Wassequalsweak}) that the \textit{weak topology on flows} is the same as the metric topology coming from the 
\textit{Wasserstein metric on flows} $d_W$ given by
     \begin{align}
        d_W(\omega,\omega') := \sum_{i=1}^{d} \m W_1^{1,1}(\mu_i,\mu_i'),
    \end{align}
    where $\mu_i,\mu_i'$ are the $i^{th}$ component dual measures of $\omega,\omega'$ respectively, and where $\m W_1^{1,1}$ is the \textit{generalized Wasserstein distance} which we describe in the next subsection. In other words, $d_W$ metrizes weak convergence for component measures.

\subsection{Wasserstein metric}

 The standard $L^1$ Wasserstein distance, denoted $W_1$, is a metric on probability measures on a metric space $(X,d)$ given by: 
   \begin{align}
       W_1(\mu,\nu) = \inf_{\gamma\in \Gamma(\mu,\nu)} \int_{X\times X} d(x,y) \, \dd \gamma(x,y),
   \end{align}
    where $\Gamma(\mu,\nu)$ is the collection of all couplings of $\mu,\nu$. Intuitively $W_1$ measures the minimum ``cost'' of transforming $\mu$ into $\nu$, where cost is the amount of mass moved times the distance moved to transform $\mu$ into $\nu$. Wasserstein metrics have been studied extensively in analysis and optimal transport, see e.g.\ \cite{Villani2009} for a general introduction. The generalized Wasserstein distance $\m W_1^{1,1}$ extends this notion to (i) measures of different total mass by adding an $L^1$ cost for adding and deleting mass \cite{Piccoli2014-vy,Piccoli2016-fe}, and (ii) to signed measures \cite{Ambrosio,piccoli_signed}. 
    
    Let $\mc M(\m R^d)$ denote the set of positive Borel measures with finite total mass. If $\mu, \nu\in \mc M(\m R^d)$, then 
    \begin{align*}
        \m W_1^{1,1}(\mu,\nu) = \inf_{\tilde{\mu},\tilde{\nu}\in \mc M(\m R^d)} |\mu-\tilde{\mu}| + |\nu-\tilde{\nu}| + W_1(\tilde{\mu},\tilde{\nu}).
    \end{align*}
    If $\mu = \mu_+-\mu_-$ and $\nu = \nu_+-\nu_-$ are signed measures, then $\m W_1^{1,1}(\mu,\nu) = \m W_1^{1,1}(\mu_++\nu_-,\nu_++\mu_-)$.

    For positive measures, it is shown in \cite[Theorem 3]{Piccoli2014-vy} that $\m W_1^{1,1}$ metrizes weak convergence for tight sequences. While the component measures we use are signed, they have bounded mass and we infer the following.
    \begin{prop}\label{Wassequalsweak}
        The Wasserstein metric on flows $d_W$ metrizes weak convergence of the component measures.
    \end{prop}

We further have that: 
\begin{thm}[{\cite[Theorem 5.3.4]{3Ddimers}}]\label{thm:AF_compact}
    The metric space $(\text{AF}(\Lambda,R),d_W)$ is compact.
\end{thm}
\begin{rem}
The only adaptation of this result is that $\text{AF}(R)=\text{AF}(\m Z^3,R)$ is replaced with $\text{AF}(\Lambda,R)$. This just changes the values asymptotic flows can take from $\mc N(\m Z^3)$ to the general $\mc N(\Lambda)$.
\end{rem}

The next three results are basic lemmas about $\m W_1^{1,1}$ which are useful for deriving bounds and proving convergence results in $d_W$. These are proved in \cite[Section 5]{3Ddimers}.

\begin{lemma}[{\cite[Lemma 5.2.7]{3Ddimers}}]\label{lem:wasserstein_bound_from_local}
	Suppose that $\mu$ and $\nu$ are component measures and are supported on a common compact set $R$ of dimension $d$. Suppose there is a partition of $R$ into sets $\mathcal{B} = \{B_1,...,B_K\}$ of diameter at most $\varepsilon$ such that $\bigg|\mu(B)-\nu(B) \bigg| < \delta$ for all $B\in \mathcal{B}$. If one of the measures corresponds to a discrete flow on $\frac{1}{n} \Lambda$, we {additionally} require that $1/n<\varepsilon$. Then 
	\begin{align*}
		\m W^{1,1}_1(\mu,\nu) \leq K(10\varepsilon^{d+1} +\delta).
	\end{align*}
\end{lemma}

   \begin{lemma}[{\cite[Lemma 5.2.8]{3Ddimers}}]\label{lem:wasserstein_bound_from_global}
       Suppose that $B\subset R$ is a connected region with piecewise smooth boundary. If $\mu,\nu$ are component measures supported in $R$ and $\m W_1^{1,1}(\mu,\nu)<\delta$, then there is a constant $C(B)$ such that 
       \begin{align}
           \m W_1^{1,1}(\mu\mid_{B},\nu\mid_{B}) < \delta + (C(B)+1)\delta^{1/2}.
       \end{align}
       Further, the constant $C(B)$ is bounded by the relation that $C(B)\delta^{1/2}\leq 2 \text{Vol}(\partial B, \delta^{1/2})$, the volume of the width $\delta^{1/2}$ annulus with inner boundary $\partial B$.
   \end{lemma}

\begin{lemma}[{\cite[Lemma 5.2.10]{3Ddimers}}]\label{lem:boundedlimits}
	Let $\nu_n$ be a sequence of signed measures supported in $R\subset \m R^d$ which converges in $\m W_1^{1,1}$ to another measure $\nu$. Further suppose the $\nu_n$ are absolutely continuous with respect to $d$-dimensional Lebesgue measure, and their densities $q_n(x)$ take values in $[-m,M]$. Then $\nu$ is also absolutely continuous with respect to $d$-dimensional Lebesgue measure. 
\end{lemma}

\subsection{Convergence theorems}\label{sec:convergence_in_weak_topology}

As above, fix $R\subset \m R^d$ to be a compact, connected domain with piecewise smooth boundary and a $d$-dimensional lattice $\Lambda$. Suppose that $R_n\subset \frac{1}{n} \Lambda$ is a sequence of lattice regions approximating $R$ as discussed above. Let $\omega_n: E(R_n)\to [0,1]$ be any sequence of discrete flows that could correspond to dimer covers, i.e., which have the specified divergences at all vertices.

\begin{thm}\label{thm:scaling_limits_asymptotic}
    Any subsequential limit $\omega_*$ of $\{\omega_n(e)e\}_{e\in E}$ as $n\to \infty$ in the weak topology on flows is an asymptotic flow, i.e., $\omega_* \in \text{AF}(\Lambda,R)$. 
\end{thm}

To prove the theorem, we first note that it is equivalent to work with the associated slope functions. 
\begin{lemma}\label{lem:replace_with_slope}
       Let $\omega_n$ be a sequence of discrete flows as above and $s_n,\tilde{s}_n$ their slope functions on white and black vertices respectively. If $\lim_{n\to \infty} \{\omega_n(e)e\}_{e\in E}=\omega_*$ in the weak topology on flows, then $\lim_{n\to\infty}  s_n=\omega_*$ and $\lim_{n\to \infty} \tilde{s}_n = -\omega_*$ .
\end{lemma}
\begin{proof}
          Let $\mu^n_i,\nu^n_i,\mu_i$ denote the $i^{th}$ component measures corresponding to $\omega_n,s_n,\omega_*$ respectively. Let $A_1,\cdots, A_K$ be the Voronoi cells of the white vertices in $R_n\cap W$. Note that $K = C n^d$ for some constant $C>0$. These sets have diameter $C'/n$ for some other constant $C'$, and $\mu_i^n(A_k) = \nu_i^n(A_k)$ for all $i,k$. Hence by Lemma~\ref{lem:wasserstein_bound_from_local}, $d_W(\omega_n,s_n) =O(1/n)$. Hence $\lim_{n\to\infty} s_n = \omega_*$.

For $\tilde{s}_n$, instead sum over the Voronoi cells of black vertices. On these, $-\tilde{s}_n$ and $\omega_n$ have the same total flow and we repeat the argument above.
\end{proof}

\begin{proof}[Proof of Theorem~\ref{thm:scaling_limits_asymptotic}]

Take a convergent subsequence $\{\omega_{n_k}\}_k$ and let $\omega_*=\lim_{k\to \infty} \omega_{n_k}$. By Lemma~\ref{lem:replace_with_slope}, 
\begin{align*}
    \lim_{k\to \infty} d_W(s_{n_k}, \omega_*)=0.
\end{align*}
Again let $A_1,\dots,A_K$ denote the Voronoi cells of the white vertices $u_1,\dots,u_K$ in $W\cap R_n$. Note that these all have volume $1/n^d$. Define the piecewise-constant flow $\gamma_n:R\to \m R^d$, which is constant on each Voronoi cell $A_m$, with value (for $x\in A_m$)
\begin{align*}
    \gamma_n(x) :=s_n(u_m)\in \mc N(\Lambda).
\end{align*}
By Lemma~\ref{lem:wasserstein_bound_from_local}, $d_W(\gamma_n,s_n) = O(1/n)$, and hence $$\lim_{k\to \infty} d_W(\gamma_{n_k},\omega_*) = 0.$$

Since $\gamma_{n_k}$ is absolutely continuous with respect to Lebesgue measure for all $k$, by Lemma~\ref{lem:boundedlimits}, $\omega_*$ is absolutely continuous with respect to Lebesgue measure. Further since $\gamma_{n_k}$ is valued in $\mc N(\Lambda)$ for all $k$, since convergence in $d_W$ implies weak convergence, $\omega_*$ is also valued in $\mc N(\Lambda)$ as a distribution. 

To see that $\omega_*$ is divergence-free, we look at another approximating sequence. Recall that $D$ is the degree of a vertex in $\Lambda$ and define $\overline{\omega}_n(e) = \omega_n(e)-1/D$ for all edges $e$ oriented white-to-black. Then we also have 
\begin{align}\label{eq:div_free_limit}
   \lim_{k\to \infty} d_W(\overline{\omega}_{n_k},\omega_*)=0.
\end{align}
The measures $\overline{\mu}_n^i$ corresponding to $\overline{\omega}_{n}$ are divergence-free as distributions on $R^\circ$ for any $n$, i.e., for any smooth function $\phi$ compactly supported in $R^\circ$, $\sum_{i=1}^d \int_R \pd{\phi}{x_i}\, \dd \overline{\mu}_i^n =0.$ (See \cite[Proposition 5.2.6]{3Ddimers}; this can be seen by applying the fundamental theorem of calculus along each edge and then using that $\sum_{e\ni u} \overline{\omega}_n(e) = 0$ for all $u\in R_n\cap R^\circ$.) Since convergence in $d_W$ implies weak convergence of the component measures, \eqref{eq:div_free_limit} implies that $\omega_*$ is divergence free in the interior of $R$. 

Therefore for any subsequential limit $\omega_*$ we have  $\omega_*\in \text{AF}(\Lambda,R)$. 

\end{proof}

Boundary conditions play an important role in the large deviation principle. If $\omega_n\to\omega$ in the weak topology on flows, then the corresponding boundary values must converge to the boundary value of $\omega$. We use the notion of 
\textit{boundary value operator} $T$ from \cite[Section 5]{3Ddimers} to define boundary values, but streamline the description.

The boundary value $T(\omega)$ of $\omega\in\text{AF}(\Lambda,R)$ on $\partial R$ is the flux of $\omega$ through $\partial R$. Let $\text{AF}^\infty(\Lambda,R)\subset \text{AF}(\Lambda,R)$ denote smooth asymptotic flows and let $\mc M^s(\m R^d)$ denote signed Borel measures on $\m R^d$.

\begin{definition}[See {\cite[Definition 5.4.1]{3Ddimers}}]
    Let $R\subset \m R^d$ be a compact region with piecewise smooth boundary $\partial R$. We define the \textit{boundary value operator} $T:(\text{AF}^\infty(\Lambda,R),d_W)\to (\mc M^s(\m R^d),\m W_1^{1,1})$ by 
\begin{align*}
    T(\omega) := \langle \omega(x), \xi(x)\rangle \mathds{1}_{\partial R}(x)\, \dd x =\langle \omega(x), \xi(x)\rangle \dd \sigma_{\partial R}(x), 
\end{align*}
where $\xi$ denotes the $L^2$ unit normal vector to $\partial R$ and $\dd \sigma_{\partial R}$ is the surface area measure on $\partial R$. 
\end{definition} 
By \cite[Proposition 5.4.6]{3Ddimers}, $T(\cdot)$ extends to a uniformly continuous map on $\text{AF}(\Lambda,R)$, and further for any $\omega\in \text{AF}(\Lambda,R)$, $T(\omega)$ is a signed measure absolutely continuous with the surface area measure $\dd \sigma_{\partial R}$ with density taking bounded values determined by $\mc N(\Lambda)$. In fact \cite[Proposition 5.4.6]{3Ddimers} shows that for any piecewise-smooth $(d-1)$-dimensional surface $S\subset R$, the analogously defined operator $T(\omega, S):= \langle \omega(x),\xi(x) \rangle \mathds{1}_{S}(x) \m \dd x$ extends to a continuous map on $\text{AF}(\Lambda,R)$.

We say that $b$ is a \textit{boundary asymptotic flow} for $\Lambda,R$ if there exists $\omega\in \text{AF}(\Lambda,R)$ such that $T(\omega) = b$. Uniform continuity of $T$ and the compactness of $\text{AF}(\Lambda,R)$ then give the following corollary.
\begin{cor}\label{cor:AFboundarycompact}
    If $b$ is a boundary asymptotic flow for $\Lambda,R$, then we define $\text{AF}(\Lambda,R,b) := T^{-1}(b)\subset \text{AF}(\Lambda,R)$. The metric space $(\text{AF}(\Lambda,R,b),d_W)$ is compact. 
\end{cor}

The natural notion of boundary condition that we consider on a graph $R_n$ consists of specifying the vertex multiplicity $\beta_n(u)=\sum_{e\ni u} \omega_n(u)\in (0,1]$ at each $u\in \partial R_n$. As a measure, after rescaling this is:
\begin{align}\label{eq:boundary_discrete_flow}
    b_n=\sum_{u\in \partial R_n}  \text{sign}(u) \beta_n(u) \frac{\mathds{1}_u(x) \,\dd x}{2n^{d-1}},
\end{align}
where $\text{sign}(u) = 1$ if $u\in W$ and $\text{sign}(u) = -1$ if $u\in B$. The factor of $1/2$ comes from the fact that we are summing over white and black vertices. We show that the notion of boundary value operator for asymptotic flows above can be extended by continuity to this case. The main idea is to bound the $\m W_1^{1,1}$ distance between the discrete boundary value in \eqref{eq:boundary_discrete_flow} and $T(\omega)$ using the $d_W$ distance between $\omega_n$ and $\omega$ in a thin neighborhood of $\partial R$. To simplify the proof (and since this assumption is also needed in Section~\ref{sec:ldp}) we assume that the continuum boundary condition $b$ is \textit{extendable outside}, i.e. that there exists $\lambda>0$ such that $b$ extends to a divergence-free flow in a $\lambda$ neighborhood of $R$. 

\begin{thm}\label{thm:discrete_continuum_boundary_convergence}
Let $R\subset \m R^d$ be compact with piecewise smooth boundary, and let $b$ be a boundary asymptotic flow which is extendable outside. Let $\omega$ be an asymptotic flow with $T(\omega) = b$ and suppose that $R_n\subset \frac{1}{n}\Lambda$ approximates $R$ as above, and $\omega_n$ is a discrete flow on $R_n$. If $\lim_{n\to \infty} \omega_n=\omega\in \text{AF}(\Lambda,R,b)$ in the weak topology on flows, then 
    \begin{align}\label{eq:boundary_converges}
        \lim_{n\to \infty} \m W_1^{1,1}(b_n , b) = 0
    \end{align}
    with $b_n$ from (\ref{eq:boundary_discrete_flow}) and $\beta_n(u) = \sum_{e\ni u} \omega_n(e)$ as above.
\end{thm}
\begin{proof}
Assume that $b$ is extendable outside to a $\lambda$ neighborhood of $R.$ 

By a mollification argument, we can without loss of generality assume that $\omega$ is continuous. To see this, first let $\omega$ also denote a divergence-free extension of itself to the $\lambda$ neighborhood of $R$. Then for any $\lambda'<\lambda$ we can define 
\begin{align*}
    \omega_{\lambda'}(x) = \frac{1}{|B_{\lambda'}(x)|}\int_{B_{\lambda'}(x)} \omega(y) \, \dd y,
\end{align*}
for $x\in R$. This is divergence-free and continuous for any $0<\lambda'<\lambda$, and converges to $\omega$ as $\lambda'$ goes to $0$. Since $T$ is uniformly continuous on $\text{AF}(\Lambda,R)$, $\lim_{\lambda'\to 0} \m W_1^{1,1}(T(\omega_{\lambda'}),T(\omega))=0.$ For the remainder of the proof we assume that $\omega$ is continuous.

Since $\partial R$ is compact and piecewise smooth, for any $\varepsilon>0$ we can partition it into $K = O(\varepsilon^{d-1})$ smooth patches with diameter of order $\varepsilon$, and perimeter of order $\varepsilon^{d-2}$. We denote these by $\pi_1,\dots,\pi_K$. Let $\xi(x)$ denote the outward-pointing normal vector to $\partial R$ at $x$. Since $\partial R$ is piecewise smooth and compact, by Taylor expansion it follows that there is a constant $c_0>0$ independent of $i$ such that $|\xi(x)-\xi(y)| < {c_0\varepsilon^2}$ for all $x,y\in \pi_i$. For $\varepsilon$ small enough, we can therefore find hyperplane patches $P_1,\dots,P_K$ such that (i) $P_i$ approximates $\pi_i$ in Hausdorff distance and (ii) $P_i$ is contained in the $\lambda$ neighborhood of $R$. Without loss of generality, we let $\omega$ denote a divergence-free extension of $\omega$ to the $\lambda$ neighborhood of $R$. We define 
\begin{align}
    b^i(x) = \langle \omega(x), \xi_i\rangle \mathds{1}_{P_i}(x)\, \dd x,
\end{align}
where $\mathds{1}_{P_i}(x)\, \dd x$ is the surface area measure on $P_i$ and $\xi_i$ denotes the normal vector to $P_i$ (with the same orientation as $\pi_i\subset \partial R$, which is oriented outward). Since the perimeter of $\pi_i$ is of order $\varepsilon^{d-2}$, and since the distance between $P_i$ and $\pi_i$ is at most a constant times $\varepsilon^2$, the divergence theorem implies that 
\begin{align}\label{eq:plane_approximation}
    \bigg| \int_{P_i} b^i - \int_{\pi_i} b \bigg| < (c_0 \varepsilon^2 )(c_1 \varepsilon^{d-2})={c\varepsilon^d}
\end{align}
for all $i\in \{1,\dots,K\}$ for some constant $c>0$. The order $\varepsilon^d$ comes from the surface area of the ``band'' needed to make $P_i\cup \pi_i$ into a closed surface, which has volume proportional to the perimeter of $\pi_i$ (order $\varepsilon^{d-2}$) times the distance between $\pi_i$ and $P_i$ (order $\varepsilon^2$).

We further let $\pi_i^n\subset \partial R_n$ be discrete patches of $\partial R_n$ which approximate $\pi_i$ (and hence also $P_i$) in Hausdorff distance. By Lemma~\ref{lem:wasserstein_bound_from_local}, if there is $\delta>0$ such that for all $i\in \{1,\dots,K\}$ we have 
    \begin{align}\label{eq:bound_per_patch}
         \bigg| \int_{\pi_i^n} b_n - \int_{\pi_i} b \bigg| < \delta,
    \end{align}
    then
    \begin{align*}
        \m W_1^{1,1}(b_n,b)< K(10\varepsilon^{d}+\delta).
    \end{align*}
    Using \eqref{eq:plane_approximation}, to show \eqref{eq:bound_per_patch}, we can replace the surface $\pi_i$ with the plane $P_i$, and it suffices to show that for all $i\in \{1,\dots,K\}$ and some $\delta>0$ (smaller than $1/K=O(\varepsilon^{d-1})$), we have
    \begin{align}\label{eq:patch_to_plane}
        \bigg| \int_{\pi_i^n} b_n -\int_{P_i} b^i\bigg| < \delta.
    \end{align}
    We now fix $i$, and achieve \eqref{eq:patch_to_plane} by relating it to the distance between $\omega_n$ and $\omega$. Fix a small constant $t_0>0$. For any $t$, we define $P_i(t) = P_i - t \xi_i$ to be this plane translated distance $t$ inward by its normal vector. We can choose $t_0$ small enough so that all these planes are still contained in the $\lambda$ neighborhood of $R$. We then define $U_i = \cup_{t=0}^{t_0} P_i(t)$, and $U_i(t) = \cup_{s=0}^t P_i(s)$. Since $\omega$ is divergence free, for all $0\leq t\leq t_0$ the divergence theorem implies that 
\begin{align}\label{eq:divergence_bound_continuous}
    \bigg | \int_{P_i} b^i - \int_{P_i(t)} \langle \omega(x), \xi_i\rangle \mathds{1}_{P_i(t)}(x) \, \dd x\bigg|=\bigg | \int_{P_i} \langle \omega(x), \xi_i\rangle \mathds{1}_{P_i}(x) \, \dd x - \int_{P_i(t)} \langle \omega(x), \xi_i\rangle \mathds{1}_{P_i(t)}(x) \, \dd x\bigg| < C t_0.
\end{align}
Here we use that $\xi_i$ on $P_i$, and $-\xi_i$ on $P_i(t)$, are outward unit normal vectors to $\partial U_i(t)$, and $C$ is a constant such that $|\partial U_i \setminus (P_i \cup P_i(t))| < Ct_0$ for all $0\leq t\leq t_0$. 
Averaging (\ref{eq:divergence_bound_continuous}) over $0\leq t\leq t_0$ gives that 
\begin{align}\label{eq:continuum_total}
    \bigg| \int_{P_i} b^i - \frac{1}{t_0}\int_{0}^{t_0} \int_{P_i(t)} \langle \omega, \xi_i\rangle \mathds{1}_{P_i(t)}(x) \, \dd x\bigg| < C t_0.
\end{align}

Next we derive a similar bound for the discrete flows. Let $U_i^n\subset R_n$ be lattice sets approximating $U_i$ in Hausdorff distance, where $U_i^n\cap \partial R_n = \pi_i^n$. We can then partition $U_i^n$ into slices, where each slice is a discrete approximation of one of the plans $P_i(t)$ for some $t$. We take values $t_1,\dots,t_L=t_0$, for $L$ approximately equal to $t_0 n$, and define $\pi_i^n(t_k)$ to be a discrete approximation of the plane $P_i(t_k)$,  such that all vertices in $U_i^n$ are contained in $\pi_i^n(t_k)$ for exactly one value of $k$. We further define $U_i^n(t_k) = \cup_{k'=1}^{k}\pi_i^n(t_{k'})$. The boundary value of $\omega_n$ restricted to $\pi_i^n(t_k)$ is 
\begin{align*}
    b_n(t) = \sum_{u\in \pi_i^n(t_k)} \frac{1}{2n^{d-1}}\text{sign}(u) \beta_n^{t_k}(u) \mathds{1}_u(x)\dd x,
\end{align*}
where (i) $\text{sign}(u)$ is $+1$ if $u\in W$ and $-1$ if $u\in B$ and (ii)  $\beta_n^{t_k}(u)$ is the sum of $\omega_n(e)$ over edges $e=(u,v)$ incident to $u\in \pi_i^n(t_k)$ with $v \not \in U_i^n(t_k)\setminus \pi_i^n(t_k)$.

The total divergence of $\omega_n$ on $U_i^n$ is $O(1/n)$. We note that since all vertices have the same degree $D$, and same collection of incident edge directions $e_1,\dots,e_D$, there is an upper bound on the Euclidean distance between vertices connected by a single edge. As such, for any $k$, the number of edges which cross laterally (i.e., out the sides of $\partial U_i^n(t_k)$ which are not one of the slices) is at most of order $t_k n^{d-1}$ (to be precise these are the edges through $\partial U_i^n(t_k)\setminus (\pi_i^n \cup \pi_i^n(t_k))$). Therefore taking $C$ possibly larger, the divergence theorem implies that for all $k$, $0\leq t_k\leq t_0$,
\begin{align}\label{eq:divergence_bound_discrete}
    \bigg| \int_{\pi_i^n} b_n - \int_{\pi_i^n(t_k)} b_n(t_k) \bigg| < C t_k n^{d-1}/n^{d-1} < C t_k < C t_0.
\end{align}
Therefore 
\begin{align}\label{eq:discrete_total}
    \bigg |\int_{\pi_i^n} b_n - \frac{1}{L}\sum_{k=1}^{L} \int_{\pi_i^n(t_k)} b_n(t_k) \bigg| < C t_0 + O(1/n). 
\end{align}
Combining \eqref{eq:continuum_total} and \eqref{eq:discrete_total} gives 
\begin{align}\label{eq:preliminary_bound}
    \bigg|\int_{P_i} b^i - \int_{\pi_i^n} b_n \bigg| < 2 C t_0 + t_0^{-1}\bigg|\int_{0}^{t_0} \int_{P_i(t)} \langle \omega, \xi_i\rangle \mathds{1}_{P_i(t)}(x) \, \dd x-\frac{1}{n}\sum_{k=1}^{K} \int_{\pi_i^n(t_k)} b_n(t_k) \bigg| + O(1/n).
\end{align}
Let $\mu=(\mu^1,\dots,\mu^d)$ and $\mu_n=(\mu_n^1,\dots,\mu_n^d)$ be the measures corresponding to $\omega, \omega_n$ respectively (recall that $\mu^m = \omega_m(x)\, \dd x$). Then splitting the integral into components, we have
\begin{align}\label{eq:plane_continuum}
    \int_{0}^{t_0} \int_{P_i(t)} \langle \omega, \xi_i\rangle \mathds{1}_{P_i(t)}(x) \, \dd x = \sum_{m=1}^d \xi_i^m \int_{U_i} \mu^m.
\end{align}
Recall that $\pi_i^n(t_k)$ is a discrete approximation of a hyperplane $P_i(t_k)$. The divergence theorem implies that any two approximations have the same flux, up to an error of order $O(\varepsilon^d)+O(1/n)$ (analogous to the bound in \eqref{eq:plane_approximation}). Therefore 
\begin{align}\label{eq:plane_discrete}
    \frac{1}{n}\sum_{k=1}^{L} \int_{\pi_i^n(t_k)} b_n(t_k) =  \sum_{m=1}^d \xi_i^m \int_{U_i} \mu_n^m+O(\varepsilon^d)+O(1/n).
\end{align}
Combining \eqref{eq:plane_continuum} and \eqref{eq:plane_discrete} with \eqref{eq:preliminary_bound} and applying the fact that $d_W(\omega_n,\omega)\to 0$ as $n\to \infty$ gives the bound \eqref{eq:patch_to_plane} with $\delta = O(\varepsilon^d)$, and completes the proof.
\end{proof}

\subsection{Comparison with sup norm topology}

Here we prove a straightforward comparison result that for discrete vector fields corresponding to $N$-dimer covers of the same finite graph $G=(V,E)$, the weak topology on flows is equivalent to the sup norm topology on edges. 

\begin{prop}\label{prop:Linfty_vs_weak}
  Suppose that $G=(V,E)\subset \Lambda$ is a finite subgraph. Then a sequence of discrete flows $(\omega_k(e))_{e\in E}$ converges to a discrete flow $(\omega(e))_{e\in E}$ as $k\to \infty$ in the sup norm topology on $E$ if and only if it converges to $(\omega(e))_{e\in E}$ in the weak topology on flows. 
\end{prop}

\begin{proof}
   First suppose that $\omega_k\to \omega$ in the sup norm topology on edges. Then for any $\varepsilon>0$ there exists $k$ large enough such that 
   \begin{align*}
       |\omega_k(e) - \omega(e)| < \varepsilon \qquad \forall e\in E. 
   \end{align*}
   To bound the Wasserstein distance $d_W$, it suffices to show that there is a way to redistribute, add, or delete flow to transform $\omega_k$ into $\omega$. In particular one way to redistribute would be to just add or delete the total flow difference on every edge $e\in E$. This gives the bound: 
   \begin{align*}
       d_W(\omega_k,\omega) < d|E|\varepsilon, 
   \end{align*}
   where $d$ is the dimension of $\Lambda$. Since $|E|$ is finite this completes the first direction.

   Next we assume that $d_W(\omega_k,\omega)\to 0$ as $k\to \infty$. Let $(\mu_1^k,\dots,\mu_d^k)$ be the measures corresponding to $\omega_k$ and let $(\mu_1,\dots,\mu_d)$ be the measures corresponding to $\omega$. For any $\varepsilon>0$ there is $k$ large enough such that $\m W_1^{1,1}(\mu_i^k,\mu_i)<\varepsilon/d=\delta$ for all $i=1,\dots,d$. 
       
    For all $i$ the support of $\mu_i^k, \mu_i$ is contained in the edges $E$. Since $E$ are embedded so that they intersect only at the vertices of $\Lambda$, we can partition the subset of $\m R^d$ covered by $G$ into cells $B_e$, each of which contains one edge $e\in E$. Each cell contains one edge $e$, on which $\mu_i^k,\mu_i$ both take constant values. Therefore {$\m W_1^{1,1}(\mu_i^k\mid_{B_e},\mu_i\mid_{B_e})  \leq1/\ell_i(e)|\omega_k(e)-\omega(e)|$}. Summing over $i$ and applying Lemma~\ref{lem:wasserstein_bound_from_global} to add together the contributions from each cell completes the proof. 
\end{proof}

%\begin{prop}\label{prop:height_vs_weak}
 %   {\color{red}TO DO:} relationship between: 1) weak topology on $f_M$ (aka $\nabla h_T)$ and 2) sup norm topology on $h_T$.  
%\end{prop}
%\begin{proof}
%   {\color{red}I think these topologies are probably be equivalent when they are both well-defined, but don't know how to prove this. }
%\end{proof}

\section{Free energy and surface tension}\label{sec:torus_asymptotics}

\subsection{Measures on the torus}

Let $\T=\T(n,\Lambda)=(V_n,E_n)$ denote the cubic torus in the $d$-dimensional lattice $\Lambda$: $\T=\Lambda/n\Lambda$. 
It contains $2n^d$ vertices, $n^d$ white and $n^d$ black. Recall that $\Omega_1(\T_N)$ is the set of single 
dimer covers of the $N$-fold blow-up $\T_N$ of $\T$. (And recall that this is not the same as $\Omega_N(\T)$, the set of $N$-dimer covers of $\T$.) 

The graph $\T$ has degree $D$ at all vertices, with edges connecting $v$ to $v+e_i$ for $i=1\dots,D$ for each white vertex $v$. We fix edge weights $w_1,\dots,w_D>0$ for edge types $e_1,\dots,e_D$ of $\T$. 

We let $\mu_{n,N,w}:=\mu_{n,N}(w_1,\dots, w_D)$ denote the associated Boltzmann measure on $\Omega_1(\T_N)$. 
Its partition function is
\begin{equation}\label{eq:partition_not_normalized}
    Z_{n,N}(w_1,\dots,w_D) = \sum_{\tau \in \Omega_1(\T_N)} \prod_{e\in \pi(\tau)=M} w_e^{M_e}= \sum_{M\in\Omega_N(\T)} |\pi^{-1}(M)|\prod_{e\in\T} w_e^{M_e},
\end{equation}
In order to study the asymptotics as $n,N\to \infty$, we need to renormalize the partition function. This does not change the corresponding probability measures $\mu_{n,N,w}$. In fact we study the asymptotics of 
\begin{equation}\label{eq:partition_normalized}
    \frac{Z_{n,N}(w_1,\dots, w_D)}{(N!)^{n^d}}.
\end{equation}
This normalization is analogous to ones used for random permutations, see e.g.\ \cite{KenyonWinkler_permuations}, \cite{Borga_LDP_permuations}. 

\begin{rem} To see that \eqref{eq:partition_normalized} is the correct normalization, consider the weight function where $w_1=1$ and $w_i =0$ for $i\neq 1$. Projected to a measure on $\Omega_N(\m T)$, $\mu_{n,N}(1,0,\dots,0)$ samples only one configuration with probability $1$, namely the $N$-dimer cover where every edge in the $w_1$ direction is in the collection with multiplicity $N$, and all other edges are not occupied. On the other hand, there are $(N!)^{n^d}$ ways to lift this configuration to a dimer cover of the blow-up graph, corresponding to choosing a permutation of $N$ for each edge $e$ in the collection. 
\end{rem}

The critical edge weights for these measures are very simple.
\begin{lemma}\label{lem:critical_weights_torus}
    Fix $n$ and $\Lambda$. As $N\to \infty$, the critical edges weights on $\m T(n,\Lambda)$ with edge weights $w_1,\dots,w_D$ are precisely $w_1,\dots, w_D$.
\end{lemma}
\begin{proof}
    Recall that the tiling polynomial with weights $w_1,\dots,w_D$ is $P(\mathbf x) = \sum_{e=(u,v)\in E_n} w_e x_u x_v$. The critical gauge function $(x_v)_{v\in V_n}$ solves the family of equations: 
    \begin{align*}
        \sum_{u: e=(u,v)\in E_n} w_e x_u x_v = \frac{1}{n^d} P(\mathbf x) \qquad \forall v\in V. 
    \end{align*}
   Since the graph is transitive (on white vertices or black vertices), $x_v\equiv 1$ for all $v\in V$ solves the critical gauge equations.
\end{proof}

It follows that the asymptotic behavior of samples from $\mu_{n,N,w}$ is determined by $w_1,\dots,w_D$.

\begin{prop}\label{prop:limit_shape_torus}
    Let $p_i := \frac{w_i}{w_1+\dots w_D}$. As $N\to\infty$ with $n$ fixed, samples from $\mu_{n,N,w}$ concentrate on the discrete flow where $\omega_n(e)=p_i$ for all edges $e$ of type $e_i$. 
\end{prop}
\begin{proof}
    By Theorem~\ref{thm:largeN_concentration}, samples from $\mu_{n,N,w}$ concentrate as $N\to \infty$ on a discrete vector field $\omega_n$ proportional to the critical edge weights, where the constant is determined by the constraint that the divergences are $\pm 1$ (if white, black respectively) at all vertices. Hence by Lemma~\ref{lem:critical_weights_torus}, $\omega_n(e) = p_i$ for all edges $e$ of type $e_i$. 
\end{proof}
\begin{rem}
    By Proposition~\ref{prop:Linfty_vs_weak}, this is equivalent to samples from $\mu_{n,N,w}$ converging to $\omega_n(e)$ in the weak topology on flows as $N\to \infty$, so if we further take the scaling limit $\lim_{n\to\infty} \omega_n$, we see that samples concentrate on the constant flow $s := \sum_{i=1}^D p_i e_i \in \mc N(\Lambda)$. 
\end{rem}

\subsection{Free energy}

Here we explicitly compute the free energy per dimer of the torus measures $\mu_{n,N,w}$ in the iterated limit as $N$ and then $n$ go to infinity. This computation is remarkably simple and depends minimally on the lattice $\Lambda$. Note that $\T=\T(n,\Lambda)$ has $2n^d$ vertices, so an $N$-dimer cover of $\T$ contains $N n^d$ edges, counted with multiplicity. 

\begin{thm}\label{thm:free_energy_torus}
    The {free energy per dimer} of $\mu_{n,N,w}=\mu_{n,N}(w_1,\dots,w_D)$ is
\begin{equation*}
    F(w_1,...,w_D) := \lim_{n\to \infty} \lim_{N\to \infty} \frac{1}{N n^d} \log \frac{ Z_{n,N}(w_1,\dots,w_D)}{(N!)^{n^d}}=
    \log (w_1+\dots+w_D).
\end{equation*}
\end{thm}

\begin{proof} This proof is an application of the results of \cite{KenyonPohoata} stated in Section~\ref{sec:KP_results}. Let $P_n(\mathbf x)$ denote the tiling polynomial for $\m T(n,\Lambda)$ with periodic weights $w_1,\dots,w_D$, namely
    \begin{equation*}
        P_n(\mathbf x) = \sum_{e=(u,v)\in E_n} w_e x_u x_v.
    \end{equation*} 
    By Theorem~\ref{thm:KP_generating_function}, 
    \begin{equation}\label{eq:partition_as_coefficient}
        \frac{Z_{n,N}(w_1,\dots,w_D)}{(N!)^{n^d}} = \frac{[\mathbf x^N] P_n(\mathbf x)^{Nn^d}}{(N n^d)!}
    \end{equation}
where the notation $[\mathbf x^N]Q$ denotes the coefficient of $\mathbf x^N = \prod_{v\in V} x_v^N$ in $Q$. We can extract the $\mathbf x^N$ coefficient in $P_n(\mathbf x)^{N n^d}$ using a contour integral:
    \begin{equation*}
        [\mathbf x^N]P_n(\mathbf x)^{Nn^d} = \bigg(\frac{1}{2\pi i}\bigg)^{n^d} \int_{(S^1)^{n^d}} \frac{P_n(\mathbf x)^{Nn^d}}{ \prod_{v\in V_n} x_v^N} \, \prod_{v\in V_n} \frac{\dd x_v}{x_v}.
    \end{equation*}
As $N\to \infty$ with $n$ fixed, the method of steepest descent says that this integral is approximated by the value at the extreme points of the integrand. In other words, by its value at $(x_v)_{v\in V_n}$ satisfying: 
    \begin{equation}\label{eq:first_weight_equation}
        \pd{}{x_v}\bigg( N n^d \log P_n(\mathbf x) - \sum_{v\in V_n} N \log x_v\bigg) = 0 \qquad \forall \; v\in V_n.
    \end{equation}
    Note that for any $v\in V_n$,
    \begin{equation*}
        \pd{}{x_v} P_n(\mathbf x) = \sum_{u : e = (u,v)\in E_n} w_e x_u.
    \end{equation*}
    Since $\m T(n,\Lambda)$ has homogeneous vertex degree $D$, this is a sum over $D$ terms for any $n$ and $v\in V_n$. Hence Equation \eqref{eq:first_weight_equation} can be rewritten:
    \begin{equation}\label{eq:second_weight_equation}
        \sum_{u:e = (u,v)\in E_n} w_e x_u x_v = \frac{1}{n^d} P_n(\mathbf x) \qquad \forall\; v\in V_n.
    \end{equation}
    These are exactly the critical gauge equations \eqref{eq:critical_weights}. By Lemma~\ref{lem:critical_weights_torus}, for any $n$, $\mathbf x \equiv 1$ solves the family of equations in \eqref{eq:second_weight_equation}, and determines the unique critical edge weight solution $w_e x_u x_v = w_e$ for all $e\in E_n$.  Therefore   \begin{equation}\label{eq:coefficient_approx}
       [\mathbf x^N] P_n(\mathbf x)^{N n^d} \approx P_n(\text{\textbf{1}})^{N n^d} \qquad \text{ as } N\to \infty.
    \end{equation}
    Further, since the graph is transitive,
    \begin{equation}\label{eq:poly_is_sum}
         P_n(\text{\textbf{1}}) = \sum_{e\in E_n} w_e = n^d (w_1+\dots+w_D).
    \end{equation}
    We can now compute the desired limit. Combining Equations \eqref{eq:coefficient_approx} and \eqref{eq:poly_is_sum} with Equation \eqref{eq:partition_as_coefficient}, we have
    \begin{align*}
        F(w_1,\dots, w_d) &= \lim_{n\to\infty}\lim_{N\to \infty} \frac{1}{N n^d} \log \bigg[ (N!)^{n^d} \frac{P_n(\text{\textbf{1}})^{N n^d}}{(N n^d)!} \bigg]\\
        &=\lim_{n\to\infty}\lim_{N\to \infty}  \frac{1}{N n^d} \bigg[n^d N \log N - n^d N - N n^d \log N n^d + N n^d + N n^d \log P_n(\text{\textbf{1}}) \bigg]\\
        &= \lim_{n\to\infty}\lim_{N\to \infty}  \bigg[-\log n^d + \log \bigg( n^d (w_1+\dots+w_D)\bigg)\bigg]\\
        &= \log(w_1+\dots+w_D).
    \end{align*}
\end{proof}

\subsection{Surface tension}\label{sec:surface_tension}

Let $\T^d=\R^d/\Z^d$ denote the $d$-dimensional unit cubic torus. If $\Lambda$ is a $d$-dimensional lattice, 
then the
graphs $\frac{1}{n}\T(n,\Lambda)$ are lattice approximations of $\m T^d$. For $s\in\mc N(\Lambda)$, the surface tension $\sigma(s)$ is defined (below) to be the asymptotic growth rate of the normalized partition function of $N$-dimer covers of $\frac{1}{n}\m T(n,\Lambda)$ which approximate flow $s$ in the $n,N\to \infty$ limit. 

To isolate the set of $N$-dimer covers $M$ of $\frac{1}{n}\m T(n,\Lambda)$ such that $\omega_M$ is close to $s$ as $n,N\to\infty$, 
we use the fact that the the torus has non-trivial (co)homology $H_1(\m T^d, \m Z) \cong H^{d-1}(\m T^d,\m Z) \cong \m Z^d$. A divergence-free discrete flow $\omega$ corresponds to a closed $(d-1)$-form, so it has a homology class $[\omega]\in \m Z^d$. 

The discrete flow $\omega_M$ has divergence $\pm 1$ at all white, black vertices respectively. The modified flow, where we subtract a constant from all edges, $\overline{\omega}_M(e) = \omega_M(e) - 1/D$, is divergence free and therefore has a homology class. On the other hand, since $\Lambda$ is harmonically embedded, this subtraction does not change the slope per white vertex function: 
\begin{align}\label{eq:slope_with_or_without}
    s_M(u) =\sum_{e\ni u} \omega_M(e) e = \sum_{e\ni u} \overline{\omega}_M(e) e.
\end{align}
\begin{definition}
    Fix a slope $s\in \mc N(\Lambda)$. Recall that $\pi$ is the projection map from dimer covers of the $N$-fold blow-up of a graph $G$ to $N$-dimer covers of $G$. We define the set of dimer covers $\tau$ such that the projection $\pi(\tau)$ (i.e., an $N$-dimer cover of $G$) has a particular homology class approximating the slope $s$:
\begin{equation*}
\Omega_{N,s}=\{ \tau\in \Omega_N : [N \overline{\omega}_{\pi(\tau)}]=[\lfloor n^{d-1} N s_1 \rfloor,\dots,\lfloor n^{d-1} N s_d \rfloor]\}.
\end{equation*}
Equivalently, by \eqref{eq:slope_with_or_without}, $s$ is the slope per white vertex over all of $\m T(n,\Lambda)$.
\end{definition}

From this, we can define probability measures $\mu_{n,N,w}^s$ restricted to a homology class $s$. The (unnormalized) partition function of this measure is
\begin{equation*}
    Z_{n,N,w}^s = \sum_{\tau \in \Omega_{N,s}} \prod_{e\in \pi(\tau) =M} w_e^{M_e}.
\end{equation*}
When $w_e\equiv 1$, we denote this by just $Z_{n,N}^s$. The surface tension is the asymptotic growth rate of these homology-class-restricted partition functions as $n,N\to \infty$, suitably normalized like the free energy.
\begin{definition}\label{def:surface_tension}
The \textit{surface tension} $\sigma: \mc N(\Lambda)\to \m R_+$ is given by
    \begin{equation}
    \sigma(s) := -\lim_{n\to \infty} \lim_{N\to \infty} \frac{1}{N n^d} \log \frac{Z_{n,N}^s}{(N!)^{n^d}}.
\end{equation}
It is straightforward to check that $\log (Z_{n,N}^s/(N!)^{n^d})$ forms as subadditive sequence in $n,N$, so $\sigma(s)$ exists for any fixed $s\in \mc N(\Lambda)$. 
\end{definition}
\begin{rem} Note that the edge weights in the surface tension definition are all equal to one. The surface tension can also be defined analogously with non-constant weights. 
\end{rem}

\subsection{Entropy}

To relate the surface tension to the free energy, we first relate the free energy to entropy. This will allow us to infer some key properties of $\sigma$ and provide a method to compute it explicitly. The arguments here mirror the ones for standard two-dimensional dimers in \cite{cohn2001variational}, but the variance bound on the number of tiles of each type is replaced by Lemma~\ref{lem:mult_concentration}.

For finite $n,N$, $\mu_{n,N,w}$ is a probability measure on the finite set $\Omega_1(\T_N)$ of dimer covers of the $N$-fold blow-up of $\m T(n,\Lambda)$. The Shannon entropy of ${\mu}_{n,N,w}$ is
\begin{equation*}
    H(\mu_{n,N,w}) = -\sum_{\tau\in \Omega(n,N)} {\mu}_{n,N,w}(\tau) \log {\mu}_{n,N,w}(\tau). 
\end{equation*} 
The key lemma is the following strengthening of Proposition~\ref{prop:limit_shape_torus}.

\begin{lemma}\label{lem:mult_concentration}
    Let $k_i$ denote the number of tiles on edges of type $e_i$, and let $\overline{k}_i = \m E_{\mu_{n,N,w}}[k_i]$. There exists a constant $c=c(w)>0$ such that for any $\varepsilon>0$ and $N$ large enough given $n,\varepsilon$, 
    \begin{equation}\label{eq:multiplicities_stong_concentrate}
        \m P_{\mu_{n,N,w}}(|k_i-\overline{k}_i|>\varepsilon \overline{k}_i)\leq \exp(-\varepsilon^2 N/2cn^d).
    \end{equation}
\end{lemma} 
\begin{proof}
By Proposition~\ref{prop:limit_shape_torus}, for any fixed $n$, samples from $\mu_{n,N,w}$ concentrate on the discrete vector field where $\omega_n(e) = p_i = \frac{w_i}{w_1+\dots+w_D}$ if $e$ is of type $e_i$ and hence has weight $w_i$. 
    
Fix $n$ and $\varepsilon>0$. By \cite[Theorem 5.1]{KenyonPohoata}, for $n$ fixed and $M$ sampled from $\mu_{n,N,w}$, the fluctuations of $(M_e)_{e\in E}$ around the limiting configuration as $N\to\infty$ are jointly Gaussian with covariances that can be computed explicitly in terms of the edge weights and are proportional to $Nn^d$. Hence there is $c=c(w)>0$ such that for all edges $e$,
\begin{align*}
        \text{Var}(M_e) \leq cNn^d
\end{align*}
Thus, $\omega_M(e) = M_e/N$ is approximately a Gaussian with variances bounded by $cn^d/N$ for $N$ large. Hence if $e$ has edge weight $w_i$, then
\begin{align*}
        \m P_{\mu_{n,N,w}}\bigg(\bigg|\omega_M(e) - p_i\bigg| > \varepsilon\bigg) <  \exp\bigg(-\frac{\varepsilon^2 N}{2cn^d}\bigg). 
\end{align*}
Since 
\begin{align*}
        \m P_{\mu_{n,N,w}}(|k_i-\overline{k}_i|>\varepsilon \overline{k}_i) \leq \m P_{\mu_{n,N,w}}\bigg(\bigg|\omega_M(e) - p_i\bigg| > \varepsilon\bigg)
\end{align*}
the result follows.
\end{proof}

\begin{thm}\label{thm:entropy_free_energy}
    The limiting \textit{entropy per dimer} for $\mu_{n,N,w}$ exists as $n,N\to \infty$ and is given by 
    \begin{equation*}
   \lim_{n\to \infty} \lim_{N\to \infty} \frac{1}{Nn^d} \bigg[H(\mu_{n,N,w}) - n^d\log N!\bigg] = F(w) - \frac{1}{w_1+\dots +w_D}\sum_{i=1}^D w_i \log w_i.
    \end{equation*}
\end{thm}
\begin{rem}
    The additive correction term on the left hand side is the same normalization as in the free energy and surface tension, and is necessary for the limit to be finite. 
\end{rem}
\begin{proof}
Recall from Theorem \ref{thm:free_energy_torus} that $F(w) = \lim_{n,N\to \infty} \frac{1}{Nn^d} [\log Z_{n,N,w}- n^d \log N! ]$. Here $Z_{n,N,w}$ is the unnormalized partition function of $\mu_{n,N,w}$. We can rewrite 
\begin{equation*}
    Z_{n,N,w} = \sum_{\tau\in \Omega_1(\T_N)} \prod_{e\in \pi(\tau)=M} w_e^{M_e} = \sum_{\tau \in \Omega_1(\T_N)} \prod_{i=1}^D w_i^{k_i(\tau)},
\end{equation*}
where $k_i(\tau)$ is the number of tiles of type $e_i$ in $\tau$. Letting $C(k_1,\dots,k_D)$ be the number of dimer covers in 
$\Omega_1(\T_N)$ with $k_1,\dots, k_D$ tiles of types $e_1,\dots, e_D$ respectively, we have that 
\begin{equation*}
    Z_{n,N,w} = \sum_{\substack{k_1,\dots, k_D\\ \sum k_i = Nn^d}} C(k_1,\dots, k_D) \prod_{i=1}^N w_i^{k_i}.
\end{equation*}
By Lemma \ref{lem:mult_concentration}, for $N$ large enough given $\varepsilon,n$, the set
\begin{equation*}
    U_{\varepsilon} = \{k=(k_1,\dots,k_D) : |k_i - \overline{k}_i| \leq \varepsilon \overline{k}_i, \; i=1,\dots,D\}
\end{equation*}
has $\mu_{n,N,w}(U_\varepsilon) \geq 1-\exp(-\varepsilon^2 N/cn^d)$ for some constant $c>0$. We can take this to be arbitrarily close to $1$ by taking $N$ larger with $\varepsilon, n$ fixed. Therefore given $n$ and $\varepsilon'>0$, for $N$ large enough we have
\begin{equation*}
    \sum_{k\in U_\varepsilon} C(k_1,\dots, k_D) \prod_{i-1}^N w_i^{k_i} = (1-\varepsilon') Z_{n,N,w}.
\end{equation*}
Let $V_\varepsilon = \{\tau \in \Omega_1(\T_N) : (k_1(\tau),\dots,k_D(\tau))\in U_\varepsilon\}$. Then
\begin{align*}
    H(\mu_{n,N,w}) &= -\sum_{\tau \in V_\varepsilon} {\mu}_{n,N,w}(\tau) \log {\mu}_{n,N,w}(\tau) +\sum_{\tau \not \in V_\varepsilon} {\mu}_{n,N,w}(\tau) \log {\mu}_{n,N}(\tau) \\
    &= -\sum_{\tau \in V_\varepsilon} {\mu}_{n,N,w}(\tau) \log \bigg[\frac{\prod_{i=1}^D w_i^{k_i}}{Z_{n,N,w}}\bigg] + O(Nn^d{\varepsilon' \log \varepsilon'})\\
    &=\sum_{\tau\in V_\varepsilon} {\mu}_{n,N,w}(\tau)\bigg[\log Z_{n,N,w} - \log \bigg(\prod_{i=1}^D w_{i}^{\overline{k}_i} w_i^{k_i-\overline{k}_i}\bigg)\bigg] + O(Nn^d {\varepsilon' \log \varepsilon'})\\
    &=\sum_{\tau\in V_\varepsilon} {\mu}_{n,N,w}(\tau)\bigg[\log Z_{n,N,w} - \sum_{i=1}^D \overline{k}_i \log w_i + (k_i-\overline{k_i}) \log w_i\bigg] + O(Nn^d {\varepsilon' \log \varepsilon'})\\
    &=(1-\varepsilon') \log Z_{n,N,w} -\sum_{\tau\in V_\varepsilon} {\mu}_{n,N,w}(\tau)\bigg[ \sum_{i=1}^D {Nn^d} p_i(n,N) \log w_i + (k_i-\overline{k_i}) \log w_i\bigg] \\
    &\;\;\;+ O(Nn^d {\varepsilon' \log \varepsilon'}).
\end{align*}
Therefore
\begin{align*}
    &\lim_{n\to \infty} \lim_{N\to \infty} \frac{1}{Nn^d}\bigg[H(\mu_{n,N,w}) -n^d \log N! \bigg]\\
    = &\lim_{n\to \infty}\lim_{N\to \infty}  \bigg[ (1-\varepsilon') \frac{1}{Nn^d} \log \frac{Z_{n,N,w}}{(N!)^{n^d}} - \sum_{\tau\in V_\varepsilon} {\mu}_{n,N}(\tau) \bigg(\sum_{i=1}^D p_i(n,N)\log w_i + \varepsilon p_i(n,N) \log w_i\bigg)\\
    &+ O(\varepsilon' \log \varepsilon') + \varepsilon'\log N\bigg].
\end{align*}
Since $\varepsilon' \sim \exp(-\varepsilon^2 N/cn^d)$, $\varepsilon' \log N \to 0$ as $N\to \infty$ with $n$ fixed. Thus applying Theorem~\ref{thm:free_energy_torus} and Lemma \ref{lem:mult_concentration} gives
\begin{align*}
  \lim_{n\to \infty} \lim_{N\to \infty} \frac{1}{Nn^d}\bigg[H(\mu_{n,N,w}) -n^d \log N! \bigg] = F(w) - \frac{1}{w_1+\dots+w_D}\sum_{i=1}^D w_i \log w_i + O(\varepsilon).
\end{align*}
Taking $\varepsilon\to 0$ completes the proof. 
\end{proof}

\subsection{Gauge equivalence and Legendre duality}

To relate $F$ to the Legendre dual of $\sigma$, we have to write $F$ as a function of $d$ ``magnetic field'' variables $\alpha_1,\dots,\alpha_d$.
Recall that $e_j\in\R^d$ is the vector representing the edge of type $e_j$. For $\alpha=(\alpha_1,\dots,\alpha_d)$, we define the weight function
on edges of type $e_j$ to be:
\begin{align}\label{eq:edge_weights_d}
    w_j(\alpha) := \exp(e_j\cdot \alpha).
\end{align}
We slightly abuse notation by writing $F(\alpha) := F(w_1(\alpha),\dots,w_D(\alpha))$. 

Recall from the background section that two edge weight functions $w_e, w_e'$ on a graph $G = (V,E)$ are gauge equivalent if there exists a function $\mathbf x:V\to \m R$ such that
\begin{align*}
    w_e' = x_u x_v w_e \qquad \forall e=(u,v) \in E.
\end{align*}
The weights in \eqref{eq:edge_weights_d} are not gauge equivalent to the uniform weights on the whole torus. This is a natural choice, however, because they are gauge equivalent to the uniform edge weights if we condition on a choice of homology class. Both of these claims can be seen from the following characterization of edge weights gauge equivalent to the uniform weights.

\begin{lemma}\label{lem:gauge_change_loop}
    Let $G=(V,E)$ be a finite bipartite graph. The weight function $\mathbf w : E\to \m R$ is gauge equivalent to the uniform weight function if and only if for every loop $\gamma$ of $G$, we have 
    \begin{align}\label{eq:gauge_change_loop}
        w_{e_1}w_{e_3}\dots w_{e_{2k-1}} = w_{e_2} w_{e_4}\dots w_{e_{2k}}
    \end{align}
    where $e_1,\dots, e_{2k}$ are the edges around $\gamma$.
\end{lemma}
\begin{proof}
    Suppose $(w_e)_{e\in E}$ are gauge equivalent to the uniform weights, and let $\mathbf x:V \to \m R$ denote the change of gauge. Then both sides of \eqref{eq:gauge_change_loop} are equal to $\prod_{v\in V} x_v$. 

    Conversely suppose \eqref{eq:gauge_change_loop} holds. First assume that $G = \gamma$ is a single cycle, and label its vertices $v_1,\dots, v_{2k}$ in cyclic order, where $e_i =(v_{i},v_{i+1})$. We define a function $\mathbf x$ such that $x_{v_1} = 1$, and 
    \begin{align*}
        x_{v_{i+1}} = w_{e_i} x_{v_i}^{-1} 
    \end{align*}
    for $i=1,\dots,2k-1$. This changes the weights on $e_1,\dots,e_{2k-1}$ to $w_{e_1},\dots, w_{e_{2k-1}}$ and it suffices to check the corresponding weight on $e_{2k}$. This is 
    \begin{align*}
        x_{v_1} x_{v_{2k}} = 1 \cdot w_{e_{2k-1}} x_{v_{2k-1}}^{-1} = w_{e_{2k-1}} w_{e_{2k-2}}^{-1} x_{v_{2k-2}} = \frac{ w_{e_{2k-1}} \dots w_{e_1}}{w_{e_{2k-2}} \dots w_{e_2}}
    \end{align*}
    which is equal to $w_{e_{2k}}$ by \eqref{eq:gauge_change_loop}. If $G$ is not just a single loop, then we iterate this construction. 
\end{proof}

The weights in \eqref{eq:edge_weights_d} satisfy the condition in Lemma~\ref{lem:gauge_change_loop} on contractible loops, but not on loops with non-trivial homology class. As such, the probability measures $\mu_{n,N}$ and $\mu_{n,N,w}$ are different (as one could infer e.g.\ from Proposition~\ref{prop:limit_shape_torus}). Instead, let $\mu_{n,N}^s,\mu_{n,N,w}^s$ denote these measures restricted to dimer covers of homology class $s$. Then since any two dimer covers of the same homology class differ by only contractible loops, Lemma~\ref{lem:gauge_change_loop} implies that $\mu_{n,N}^s$ and $\mu_{n,N,w}^s$ are the same probability measure.

\begin{thm}\label{cor:almost_legendre}
    Fix $\Lambda$. Let $w_1,\dots,w_D$ be periodic edge weights satisfying \eqref{eq:gauge_change_loop} for contractible loops, and suppose that $s=\sum_{i=1}^D p_i e_i$, where $p_i = \frac{w_i}{w_1+\dots+w_D}$. Then the surface tension is minus the entropy of $\mu_{n,N,w}$, i.e.,
    \begin{align*}
        \sigma(s) = \frac{1}{w_1+\dots+w_D}\sum_{i=1}^D w_i \log w_i - F(w). 
    \end{align*}
\end{thm}

\begin{proof}
    Recall that $\sigma(s)$ is given by: 
    \begin{align*}
        \sigma(s) &= -\lim_{n\to \infty} \lim_{N\to \infty} \frac{1}{Nn^d} \bigg[\log {Z_{n,N}^s} - n^d \log N! \bigg].
    \end{align*}
  By Lemma~\ref{lem:gauge_change_loop}, the weights $w$ are gauge equivalent to the uniform weights if we restrict to a homology class. Therefore there exist constants $C_{n,N}(w)$ such that
  \begin{align*}
      Z_{n,N,w}^s = C_{n,N}(w) Z_{n,N}^s.
  \end{align*}
    Since $s$ satisfies $s=\sum_{i=1}^D p_i e_i$, by Proposition~\ref{prop:limit_shape_torus}, the constants $C_{n,N}(w)$ are asymptotic to $\prod_{i=1}^D w_i^{p_i Nn^d}$, and $Z_{n,N,w}^s/Z_{n,N,w}\to 1$ in the iterated limit as $N\to \infty$ and then $n\to \infty$. Therefore 
  \begin{align*}
      \sigma(s) &= -\lim_{n\to \infty} \lim_{N\to \infty} \frac{1}{Nn^d} \bigg[\log \frac{Z_{n,N,w}}{C_{n,N}(w)} - n^d \log N! \bigg]\\
      &=-\lim_{n\to \infty} \lim_{N\to \infty} \frac{1}{Nn^d} \bigg[\log {Z_{n,N,w}} - Nn^d\sum_{i=1}^D p_i\log w_i - n^d \log N! \bigg]\\
    &= - F(w) + \frac{1}{w_1+\dots+w_D}\sum_{i=1}^D w_i \log w_i.
    \end{align*}
\end{proof}

Using the magnetic field weights of \eqref{eq:edge_weights_d}, we show that the surface tension is the Legendre dual of the free energy. 

\begin{thm}\label{thm:legendre_duality}
    Fix $\Lambda$. Then $\sigma(s)$ is the Legendre dual of $F(\alpha)$, i.e. for all $s\in \mc N(\Lambda)$, 
    \begin{equation}\label{eq:legendre}
        \sigma(s) = \max_{\alpha=(\alpha_1,\dots,\alpha_d)\in \m R^d} - F(\alpha) + \sum_{i=1}^d s_i \alpha_i.
    \end{equation}
\end{thm} 

\begin{proof} By Theorem \ref{thm:free_energy_torus} and the prescription in \eqref{eq:edge_weights_d},
\begin{align*}
    F(\alpha) = \log\bigg( \sum_{j=1}^D \exp(e_j\cdot \alpha)\bigg).
\end{align*}
Since $F$ is strictly convex and differentiable, the maximum value in \eqref{eq:legendre} is realized at $\alpha$ such that $s=\nabla F(\alpha)$. We compute: 
\begin{align*}
    \pd{F}{\alpha_i}(\alpha) = \frac{\sum_{j=1}^D \langle e_j, \eta_i\rangle \exp(e_j\cdot \alpha)}{\sum_{k=1}^D \exp(e_k \cdot \alpha)} \implies \nabla F(\alpha) = \frac{\sum_{j=1}^D \exp(e_j\cdot \alpha) e_j}{\sum_{k=1}^D \exp(e_k \cdot \alpha)}.
\end{align*}
On the other hand plugging the weights $w_j=\exp(e_j\cdot \alpha)$ in to the expression in Corollary~\ref{cor:almost_legendre}, where $s= \sum_{j=1}^D p_j e_j$, $p_j=w_j/(w_1+\dots+w_d)$, gives
\begin{align*}
    \sigma(s) &= \sum_{j=1}^d \frac{\exp(e_j\cdot \alpha)}{\sum_{k=1}^D \exp(e_k \cdot \alpha)} (e_j\cdot \alpha) - F(\alpha) = \nabla F(\alpha) \cdot \alpha - F(\alpha),
\end{align*}
as desired.
\end{proof}

We summarize some properties of $F(\alpha)$ and $\sigma(s)$ in the following corollaries. We define the \textit{amoeba} $\mc A(\Lambda)\subset \m R^d$ to be the largest domain where $F(\alpha)$ is strictly convex. 

\begin{cor}\label{thm:free_energy_convex}
The free energy $F(\alpha):\m R^d\to \m R$ is strictly convex and real analytic on all of $\m R^d$, in particular, $\mc A(\Lambda)= \m R^d$. 
\end{cor}
\begin{cor}\label{thm:strict_convex}
The surface tension $\sigma: \mc N(\Lambda)\to \m R$ is strictly convex on the entire closed Newton polytope $\mc N(\Lambda)$. 
\end{cor}
\begin{rem}
 These results are somewhat different from their analogs for the standard dimer model. In that setting, the surface tension is convex everywhere but is piecewise linear on the one-dimensional edges of $\partial \mc N(\Lambda)$, hence strictly convex only on the complement of the edges of $\partial \mc N(\Lambda)$ in $\mc N(\Lambda)$. Similarly, for the standard dimer model the amoeba is a strict subset of the plane. We leverage these corollaries in Section~\ref{sec:no_facets}. 
\end{rem}

\section{Patching configurations for large $N$}\label{sec:patching}

Here we prove a ``local interchangeability'' property for $\NN$-dimer covers as $\NN,n\to \infty$. This result will apply to samples from multinomial dimer measures because of $\NN\to \infty$ concentration (Theorem~\ref{thm:largeN_concentration}). This can be seen as a simplification of the patching theorem for dimer covers of $\m Z^3$ in \cite{3Ddimers}, where the proof here is simpler because of the control we have as $\NN\to \infty$. The main tool will be Hall's matching theorem, described below in Section~\ref{sec:hall}.

\subsection{Statement and set up}

Take a compact domain $R\subset \R^d$ with piecewise smooth boundary and let $R_n\subset \frac{1}{n}\Lambda$ be a lattice approximation of $R$ as before. We prove that two $N$-dimer covers of $R_n$ can be ``patched'' together, one on each side of the boundaries of an annulus of lattice width $O(n)$, in the iterated limit as $N$ and then $n$ go to infinity, if their $N\to \infty$ limiting boundary conditions are within $O(1/n)$.

\begin{thm}\label{thm:patching_limit}
Fix a constant $\delta>0$. Let $R,R_n$ be as above. Let $\mu_{n,N}$ and $\nu_{n,N}$ be sequences of measures on $N$-dimer covers of $R_n$ with boundary vertex multiplicities converging to $(\alpha_v^n)_{v\in \partial R_n},(\beta_v^n)_{v\in \partial R_n}$ respectively as $N\to \infty$. Assume that 
    \begin{itemize}
        \item The $N\to \infty$ limiting boundary values satisfy $\sup_{v\in \partial R_n} |\alpha_v^n-\beta_v^n| = O({1}/{n})$.
        \item For any $n$ fixed, samples from $\nu_{n,N}$ concentrate on a limiting discrete flow $\omega_n$ as $N\to \infty$ which is (i) for all $v\in R_n^\circ$, $\sum_{e\ni v} \omega_n(e) = 1$ and for all $v\in \partial R_n$, $\sum_{e\ni v} \omega_n(e) = \beta_v^n$, and (ii) \textit{strictly non-degenerate}, i.e.\ there is $a>0$ independent of $n$ such that $\omega_n(e) >a$ for all edges $e\in E(R_n)$. 
    \end{itemize}
    Let $R_n^\delta$ be a lattice approximation of $R\setminus N_\delta(\partial R) = \{x\in R : d(x,\partial R)>\delta\}$. If $M,M'$ are sampled from $\mu_{N,n},\nu_{N,n}$ respectively, the joint probability that there exists a dimer cover $M''$ of $R_n$ such that $M''\mid_{R_n^\delta} = M'$ and $M''$ has the same vertex multiplicities as $M$ at all $v\in \partial R_n$ goes to $1$ in the iterated limit as $N\to \infty$ and then $n\to\infty$.  
\end{thm}

\begin{rem}
    The notation $M''\mid_{R_n^\delta}=M'$ means that $M''_e = M'_e$ for all edges $e$ with an endpoint in $(R_n^\delta)^\circ$. 
\end{rem}

\begin{rem}\label{rem:applies_to_multinomial}
If the limiting boundary conditions $\beta_v^n$ are feasible for each $n$ and $\nu_{n,N}$ are \emph{multinomial} dimer measures with these limiting multiplicities, then by Theorem~\ref{thm:largeN_concentration} samples from $\nu_{n,N}$ will concentrate on a limiting flow $\omega_n$ satisfying (i) and which is non-degenerate, i.e. $\omega_n(e)>0$ for all $e\in R_n$. The theorem applies to any sequence of multinomial measures $\nu_{n,N}$ as long as their limiting flows $\omega_n$ together satisfy the strictly non-degenerate condition.
\end{rem}

Theorem~\ref{thm:patching_limit} is equivalent to saying that if $n, N$ are large enough, and $M,M'$ are sampled as described, then the annular region $A_n=R_n\setminus R_n^\delta$ has an $\NN$-dimer cover with multiplicity $N$ on its interior and vertex multiplicities on $\partial A_n$ determined by $M,M'$ respectively. The key tool in the proof is a version of \textit{Hall's matching theorem} which we describe in the next section. 

\subsection{Hall's matching theorem and counterexamples}\label{sec:hall}

An $\NN$-dimer cover of a graph $G$ can be lifted to a standard dimer cover of the $\mathbf{N}$-fold blow-up graph of $G$, so it suffices to show that the associated region in the blow-up graph has a dimer cover. {Hall's matching theorem} \cite{hallmarriage} gives a necessary and sufficient condition for a bipartite graph to have a dimer cover. 
\begin{thm}[Hall's matching theorem]
    Suppose that $G$ is a finite bipartite graph with bipartition $(B,W)$ such that $|B| = |W|$. There exists a dimer cover of $G$ if and only if there is no set $C\subset W$ such that 
    $$N(C) = \{v\in B : \exists\, c\in C, (v,c)\in E\}$$
    satisfies $|N(C)| < |C|$. We call $U:=C \cup N(C)$ satisfying these conditions a \textit{counterexample} in $G$.
\end{thm}
It is straightforward to see that Hall's matching theorem applied in the blow-up graph of $G$ projects to the following equivalent condition in $G$. 
\begin{prop}[Hall's matching theorem projected from the blow-up graph]\label{prop:hall_projection}
    Let $G$ be a finite bipartite graph with bipartition $(B,W)$. Label every vertex $v\in G$ with a multiplicity $N_v$ such that 
    \begin{align*}
        \sum_{v\in W} N_v = \sum_{v\in B} N_v.
    \end{align*}
    There exists an $\NN= (N_v)_{v}$-dimer cover of $G$ if and only if there is no set $C\subset W$ such that 
    $$N(C) = \{v\in B : \exists\, c\in C, (c,v)\in E\}$$
    satisfies 
    \begin{align*}
        \sum_{v\in N(C)} N_v < \sum_{v\in C} N_v.
    \end{align*}
    We call $U:=C \cup N(C)$ satisfying these conditions an \emph{$\NN$-counterexample} in $G$.
\end{prop}
We say that the $\NN$-\textit{imbalance} of a set $U\subset G$ is 
\begin{equation}
    \text{imbalance}_{\NN}(U) := \sum_{v\in W\cap U} N_v - \sum_{v\in B \cap U} N_v.
\end{equation}
Recall that we work with graphs $G$ which are a subset of a lattice. Any counterexample set $U$ must satisfy the condition that its \textit{interior boundary} $\partial_0 U$ is contained in $B$: 
\begin{equation}\label{eq:black_interior_boundary}
    \partial_0 U :=\{u\in U: \exists\, v\in G\setminus U, (u,v)\in E\} \subset B.
\end{equation}
More generally, we call a set $U$ with the property in \eqref{eq:black_interior_boundary} a \textit{potential counterexample} (note that this definition does not depend on $\NN$). If $U$ is a potential counterexample, it is an $\NN$-counterexample if and only if $\text{imbalance}_{\NN}(U)>0$. 

In our setting, the multiplicities at $v\in G^\circ$ are $N$, and the multiplicities along $v\in \partial G$ are between $1$ and $N$. We relate the $\NN$-imbalance of a potential counterexample $U$ to the maximum multiplicities of 
edges incident to vertices $\partial U = \{u\in U : \exists\, v\in \Lambda\setminus U, (u,v)\in E\}$. Note that this is the boundary of $U$ {in $\Lambda$}, not in $G$. Correspondingly, we split these edges into two groups, ``interior" boundary edges $I_E(U)$ and ``exterior" boundary edges $X_E(U)$: 
\begin{align*}
    I_E(U) &= \{e = (u,v) : u\in \partial U, v\in G\setminus U\}\\
    X_E(U) &= \{e = (u,v) : u\in \partial U, v\in \Lambda \setminus G\}.
\end{align*}

Recall that $D$ denotes the degree of any vertex in $\Lambda$. The divergence theorem (applied in the blow up graph) gives the following.

\begin{prop}\label{prop:imbalance}
   If $U$ is a potential counterexample in $G\subset \Lambda$ with multiplicities $N_v=N$ at all $v\in G^\circ$ and multiplicities $N_v$ between $1$ and $N$ at $v\in \partial G$, then 
    \begin{equation*}
        ND \cdot \text{imbalance}_{\NN}(U) = -N |I_E(U)| + \sum_{u\in \partial U} |X_E(u)| N_u \,\text{sign}(u),
    \end{equation*}
    where $\text{sign}(u) = +1$ if $u\in W$ and $-1$ if $u\in B$ and $X_E(u)$ is the set of $e\in X_E(U)$ incident to the specified vertex $u\in \partial U$ (this can be empty).
\end{prop}

\begin{proof}
Let $\tilde U\subset \Lambda_{\NN}$ denote a lift of $U$ to the blow-up graph.
% where for each vertex $u\in U$ we include $N_u$ lifts of $u$ in $\tilde U$. 
Let $(\tilde W, \tilde B)$ denote the bipartition in $\Lambda_{\NN}$. On one hand it is clear that 
    \begin{align*}
        \text{imbalance}_{\NN}(U) = |\tilde W\cap \tilde U| - |\tilde B \cap \tilde U|.
    \end{align*}
    On the other hand, consider a discrete vector field $r$ on the edges of $\Lambda_{\NN}$ defined by $r(\tilde e) = 1/(ND)$ for all $\tilde e\in E(\Lambda_{\NN})$ oriented from $\tilde W$ to $\tilde B$. Note that this has divergence $+1$ at every $v\in \tilde W$ and $-1$ at every $v\in \tilde B$. By the divergence theorem, 
    \begin{align*}
        \text{imbalance}_{\NN}(U) = \text{div}_{\tilde U}(r) = \frac{1}{ND} \sum_{\tilde e=(x,y):x\in \tilde U, y\not\in \tilde U} {\text{sign}(x)}.
    \end{align*}
    Let $\pi$ denote the projection to $\Lambda_{\NN}\to\Lambda$. For each $\tilde e$ in the sum, $\pi(\tilde e)$ is incident to $\partial U$. If $\pi(\tilde e) \in I_E(U)$, then $\text{sign}(x) = -1$. Otherwise it can have either sign. Every $e\in I_E(U)$ has $N$ lifts in the sum, and every $e\in X_E(U)$ has $N_u$ lifts if $e=(u,v)$, $u\in U$. Putting this together gives the result.
\end{proof}

In the proof of Theorem~\ref{thm:patching_limit}, we apply Hall's matching theorem to an annular region of lattice width $O(n)$. We use the $O(n)$ width to get a lower bound on the number of interior boundary edges along a potential counterexample set $U$. 
\begin{lemma}\label{lem:surface_area}
Suppose that $A=R_n\setminus R_n^\delta$ is as in Theorem~\ref{thm:patching_limit}, and suppose that $U\subset A$ is a potential counterexample which intersects both boundary components of $A$. Then there exists a constant $\kappa=\kappa(\delta)$ such that 
    \begin{align*}
        |I_E(U)| \geq \kappa n^{d-1}. 
    \end{align*}
\end{lemma}
The lemma follows from a discrete version of the isoperimetric inequality. For $\Lambda=\m Z^3$ it is the same as \cite[Proposition 6.5.1]{3Ddimers} but restated in our terminology.

\subsection{Proof of the patching theorem}\label{sec:patching_limit}

\begin{proof}[Proof of Theorem~\ref{thm:patching_limit}] Let $A_n = R_n\setminus R_n^\delta$ and suppose for contradiction that $U_n$ is a counterexample with boundary conditions on $A_n$ coming from $M_1,M_2$ sampled from $\mu_{n,N},\nu_{n,N}$ respectively. Note that since $M,M'$ alone extend to tilings of all of $R_n$, $U_n$ must connect inner and outer boundaries of $A_n$. {It suffices to show that $A_n$ is tileable with the following multiplicities on its boundary.
\begin{itemize}
        \item For $a\in \partial R_n$, the multiplicity is $N_a(M,A_n)= \sum_{e=(a,v),v\in A_n} {M}_e=\sum_{e=(a,v)} M_e$.
        \item For $a\in \partial R_n^\delta$, the multiplicity is $N_a(M',A_n) = \sum_{e=(a,v), v\in A_n} M_e'$. Note that $M_e'$ summed over all edges $e$ in $R_n$ incident to $a$ would be $N$ since $a\in R_n^\circ$, but this is the multiplicity needed for a tiling of $A_n$ that agrees with $M'$ in the interior of $R_n^\delta$.
    \end{itemize}    }
By Proposition~\ref{prop:imbalance} we have:
 \begin{align*}
        &ND \cdot \text{imbalance}_{\NN}(U_n) = -N |I_E(U_n)| + \sum_{u\in \partial U_n} |X_E(u)| N_u\, \text{sign}(u)  \\
        = &-N |I_E(U_n)| + \sum_{u\in \partial U_n\cap \partial R_n} |X_E(u)| N_u(M,A_n)\, \text{sign}(u) + \sum_{u\in \partial U_n\cap \partial R_n^\delta} |X_E(u)|N_u(M',A_n)\, \text{sign}(u).\label{eq:imbalance_in_proof2}
    \end{align*}
Since $M'$ extends to an $N$-dimer cover of $A_n$ with its own boundary condition on $\partial R_n$, we have that 
\begin{align*}
     \sum_{u\in \partial U_n\cap \partial R_n^\delta} |X_E(u)|N_u(M',A_n)\, \text{sign}(u) < N\sum_{e\in I_E(U_n)} M_e' -   \sum_{u\in \partial U_n\cap \partial R_n} |X_E(u)|N_u(M',A_n)\, \text{sign}(u)
\end{align*}
and we use this to replace the third term in the sum above. Along the outer boundary $\partial R_n$, the limiting boundary multiplicities $\alpha_v^n,\beta_v^n$ differ by $O(1/{n})$. On the other hand the total area of the boundary of is order $n^{d-1}$. Thus for any $\varepsilon>0$, for $N$ large enough, we have with probability $1-\varepsilon$,
\begin{equation*}
        \bigg|\sum_{u\in \partial U_n\cap \partial R_n} |X_E(u)|N_u(M,A_n)\, \text{sign}(u)  -  \sum_{u\in \partial U_n\cap \partial R_n} |X_E(u)|N_u(M',A_n)\,\text{sign}(u)\bigg| < \varepsilon N n^{d-1}+ O(N n^{d-2}).
    \end{equation*}
Combining these gives the bound
   \begin{align}\label{eq:imbalance_bound_3}
      D \cdot  \text{imbalance}_{\NN}(U_n) &\leq -|I_E(U_n)| + \sum_{e\in I_E(U_n)} \frac{M_e'}{N} + \varepsilon n^{d-1} + O(n^{d-2}).
    \end{align}
For $n$ fixed, as $N\to \infty$ samples from $\nu_{n,N}$ concentrate on a limiting flow $\omega_n$ which has unit divergences in $R_n^\circ$. Further, this family of limiting flows is strictly non-degenerate in $n$. Combined with the fact that $U_n$ connects the inner and outer boundaries of $A_n$, Lemma~\ref{lem:surface_area} implies there is a constant $c=c(\delta)$ such that for $N$ large enough, with probability $1-\varepsilon$,
\begin{align*}
    N|I_E(U_n)| - \sum_{e\in I_{E}(U_n)} M_e' > c N n^{d-1}.
\end{align*}
Plugging this in to \eqref{eq:imbalance_bound_3} gives that for $N$ large enough, we have with probability $1-\varepsilon$:
 \begin{align*}
        D \cdot  \text{imbalance}_{\NN}(U_n) &\leq -c n^{d-1} + \varepsilon n^{d-1} + O(n^{d-2}).
    \end{align*}
Since $c>0$ is fixed and $\varepsilon$ is arbitrary, taking $\varepsilon<c$ and then $N,n$ large enough would make the imbalance non-positive. Therefore for $n,N$ large enough, $U_n$ is not an $\NN$ counterexample, which completes the proof. 
\end{proof}

\section{Large deviations}\label{sec:ldp}

\subsection{Set up, statement, and corollaries}

Fix $R\subset \m R^d$ which is a compact, connected domain with piecewise smooth boundary and a boundary asymptotic flow $b$ on $\partial R$ as in Section~\ref{sec:vector_fields}. Let $R_n\subset \frac{1}{n}\Lambda$ be a lattice approximation of $R$, i.e.\ a sequence of lattice regions approximating it in Hausdorff distance. Here we prove a large deviation principle in the iterated limit as $N\to\infty$ and then $n\to \infty$, for samples from multinomial measures on $R_n$ with discrete boundary conditions converging to $b$ as $n,N\to \infty$ in the way specified below. 

We assume that $b$ is \textit{extendable outside}, i.e.\ there exists $\varepsilon>0$ such that $b$ extends to an asymptotic flow on an $\varepsilon$ neighborhood of $R$ (versions of this condition are also present in the large deviation principles for the standard dimer model in 2D \cite{cohn2001variational} and 3D \cite{3Ddimers}). 

On the discrete side, for each $n$ fixed, we specify boundary conditions on $R_n$ as $N\to \infty$ by choosing the limiting vertex multiplicities $\beta_n(v) \in (0,1]$ for all $v\in \partial R_n$. In terms of a discrete flow $\omega_n$ corresponding to a dimer cover of $R_n$, this is fixing the divergences of $\omega_n$ along $\partial R_n$. The rescaled boundary conditions on $R_n$ viewed as a dual measure as in Section~\ref{sec:convergence_in_weak_topology} are thus
\begin{align*}
    b_n:=\sum_{v\in \partial R_n} \text{sign}(v) \beta_n(v) \frac{\mathds{1}_v(x)\, \dd x}{2n^{d-1}} ,
\end{align*}
where $\text{sign}(v) = 1$ if $v\in W$ and $-1$ if $v\in B$. We say that $\beta_n$ converges to $b$ as $n\to \infty$ if $\m W_1^{1,1}(b_n, b) \to 0$ as $n\to \infty$.

For each $n$, we consider only feasible limiting boundary conditions $\beta_n$ on $\partial R_n$. To avoid some technicalities, instead of fixing one limiting boundary value for each $n$, we fix \textit{thresholds} $\theta_{n}$ such that $\theta_n\to 0$ as $n\to \infty$, but sufficiently slowly (the existence of these thresholds follows from Remark~\ref{rem:thresholds}). We define $\rho_{n,N}$ to be the multinomial dimer measure with vertex multiplicities $N$ at all interior points of $R_n$, and any multiplicities $(N_v)_{v\in \partial R}$ between $1$ and $N$ which are feasible and such that the limiting boundary value $b_n$ as $N\to \infty$ with these choices has $\m W_{1}^{1,1}(b,b_n)<\theta_n$.

The main result of this section is: 

\begin{thm}[Large deviation principle]\label{thm:ldp}
    Fix $\Lambda$, $R\subset \m R^d$ compact, connected with piecewise smooth boundary, and a boundary asymptotic flow $b$ on $\partial R$ which is extendable outside. Let $R_n\subset \frac{1}{n}\Lambda$ a lattice approximation of $R$, fix thresholds $\theta_{n}$ going to zero sufficiently slowly, and probability measures $\rho_{n,N}$ as above. 
    
    The measures $(\rho_{n,N})_{n,N\geq 1}$ satisfy a large deviation principle in the iterated limit as $N$ and then $n$ go to infinity in the weak topology on flows with good rate function $I:\text{AF}(\Lambda,R,b)\to \m R$. That is, for any Borel set $B$ in the weak topology on flows, 
    \begin{align*}
        -\inf_{\omega\in {B}^{\circ}} I(\omega) \leq \liminf_{n\to \infty} \liminf_{N\to\infty} \frac{1}{Nn^d} \log \rho_{n,N}(B) \leq \limsup_{n\to \infty} \limsup_{N\to \infty} \frac{1}{Nn^d} \log \rho_{n,N}(B) \leq -\inf_{\omega\in \overline{B}} I(\omega)
    \end{align*}
    where $\overline{B}, B^{\circ}$ denote the closure and interior of $B$ respectively. Further, if $\omega$ is an asymptotic flow, the rate function $I$ is given explicitly by 
    \begin{align}
        I(\omega) = C_b +\int_R \sigma(\omega(x)) \, \dd x,
    \end{align}
    where $C_b\geq 0$ is an additive constant and $\sigma$ is the surface tension for $\Lambda$ from Theorem~\ref{thm:legendre_duality}, i.e., it is the Legendre dual of the free energy on the torus $F(\alpha) = \log \sum_{j=1}^D \exp(e_j\cdot \alpha)$. In particular $C_b =-\min_{\omega\in \text{AF}(\Lambda,R,b)} \int_R \sigma(\omega(x))\,\dd x$. If $\omega$ is not an asymptotic flow, then $I(\omega) = \infty.$
\end{thm}
\begin{rem}
    The integral $\int_R \sigma(\omega(x)) \, \dd x$ is nonpositive for all asymptotic flows. Shifting by the additive constant $C_b$ makes $0$ the minimum value of $I(\cdot)$ on $\text{AF}(\Lambda,R,b)$. To simplify some of the later statements, we define the \textit{total entropy} $\Ent$ by
    \begin{align}\label{eq:total_surface_tension}
        \Ent(\omega):= -\int_R \sigma(\omega(x))\,\dd x.
    \end{align}
    We call this total entropy because of Corollary~\ref{cor:almost_legendre}, where we saw that the surface tension and entropy are related by a minus sign. Note that the choice of sign here means that $\Ent$ is a non-negative, strictly \textit{concave} function. 
\end{rem}

\begin{rem}
    The thresholds $\theta_n$ only enter in to one step of the proof (Theorem~\ref{thm:discrete_approximation}), and could be avoided by a proving a stronger and more technical patching theorem (replacing Theorem~\ref{thm:patching_limit}). For $\m Z^3$, a stronger patching theorem is implied by the generalized patching theorem for single dimer covers of $\m Z^3$ proved in \cite{3Ddimers}, and hence it would follow that Theorem~\ref{thm:ldp_intro} holds without the thresholds under the mild regularity condition that $(R,b)$ is \textit{flexible}, see Definition~\ref{def:flexible}. However, we choose to give more direct, simplified proofs that make use of the large $N$ structure, instead of relying on the substantial tools developed specifically for $\m Z^3$.
\end{rem}

As a corollary of the large deviation principle and properties of the surface tension $\sigma$, we see that samples from $\rho_{n,N}$ concentrate on a unique \textit{limit shape}, which is the minimizer of $I$ subject to the boundary conditions. 

\begin{cor}(Limit shape.)\label{cor:limit_shape} 
Fix $\Lambda$ and $(R,b)$ as in Theorem~\ref{thm:ldp}. First, the rate function $I$ has a unique minimizer $\omega_{\min}\in \text{AF}(\Lambda,R,b)$. Second, given $\varepsilon>0$ define the event 
    \begin{align*}
        V_\varepsilon = \{ \omega: d_W(\omega_{\min},\omega)>\varepsilon\}. 
    \end{align*}
    For any sequence of multinomial measures $\rho_{n,N}$ as in Theorem \ref{thm:ldp}, dimer covers sampled from $\rho_{n,N}$ concentrate exponentially fast on $\omega_{\min}$ in the iterated limit as $N$ and then $n$ go to infinity. In other words there is a constant $C>0$ determined by $(R,b)$ such for $n,N$ sufficiently large,
    \begin{align}
        \rho_{n,N}(V_\varepsilon) < C^{-Nn^d}. 
    \end{align}
    \end{cor}
\begin{proof}
    By Corollary~\ref{cor:AFboundarycompact}, $\text{AF}(\Lambda,R,b)$ is compact. By Corollary~\ref{thm:strict_convex}, $\sigma$ is strictly convex on the closed Newton polytope $\mc N(\Lambda)$. Combined this shows that $I$ has a unique minimum $\omega_{\min}\in \text{AF}(\Lambda,R,b)$. 

    Concentration on the minimizer follows from Theorem~\ref{thm:ldp} by a standard argument. Fix $\varepsilon>0$ and cover $\text{AF}(\Lambda,R,b)$ by open neighborhoods $B_{\omega}$ such that: 
    \begin{itemize}
        \item $B_{\omega_{\min}} = \{\omega : d_W(\omega,\omega_{\min})<\varepsilon\}$.
        \item If $\omega\neq \omega_{\min}$, then $I(\omega') > I(\omega_{\min})$ for all $\omega'\in \overline{B}_{\omega}$.
    \end{itemize}
By compactness this has a finite subcover which we can denote $B_{\omega_{\min}} = B_1, B_2, \dots, B_k$. Let $\omega_i$ be a minimizer of $I$ over $\overline{B_i}$. By Theorem~\ref{thm:ldp}, 
    \begin{align}
        \rho_{n,N}(V_\varepsilon) \leq \sum_{i=2}^k \rho_{n,N}(B_i) 
     \leq \sum_{i=2}^k \exp(Nn^d (I(\omega_{\min}) - I(\omega_i))). 
    \end{align}
    Since $I(\omega_{\min}) - I(\omega_i) < 0$ for all $i=2,\dots,k$ this completes the proof.
\end{proof}

The rest of this section is dedicated to proving Theorem~\ref{thm:ldp}. We first prove here that $I$ is a good rate function.

\begin{lemma}\label{lem:semicontinuity}
Fix $(R,b)$ as above. The function $I:\text{AF}(\Lambda,R,b)\to [0,\infty)$ given by 
    \begin{align*}
        I(\omega) = C_b+\int_R \sigma(\omega(x))\, \dd x
    \end{align*}
is a good rate function, i.e. it is lower semicontinuous in the weak topology on flows and the level sets $\{\omega\in \text{AF}(\Lambda,R,b):I(\omega)\leq c\}$ are compact. 
\end{lemma} 
\begin{rem}
    Since $\Ent,I$ have opposite signs, it follows that $\Ent$ is upper semicontinuous. 
\end{rem}
\begin{proof} 

We extend $I$ (which is defined to be finite only on $\text{AF}(\Lambda,R)$) to $\tilde{I}(\omega) =C_b+\int_{R} \sigma(\omega(x))\,\dd x$, which is defined for any measurable flow valued in $\mc N(\Lambda)$ (i.e., we have removed the divergence-free requirement). 

Suppose that $(\omega_k)_{k\geq 1}\subset \text{AF}(\Lambda,R)$ is a sequence with $d_W(\omega_k, \omega)\to 0$ as $k\to \infty$. We replace $ \omega_k$ with $ \omega_{k,\varepsilon}$ defined by 
\begin{align*}
     \omega_{k,\varepsilon}(x) = \frac{1}{|B_\varepsilon(x)|} \int_{B_\varepsilon(x)}  \omega_k(y)\,\dd y,
\end{align*}
where we set $\omega_k(y) = 0$ if $y\not\in R$. We define $\omega_{\varepsilon}$ similarly by averaging $\omega$.  These mollified flows may not be divergence-free, but they are valued in $\mc N(\Lambda)$ and hence $\tilde{I}(\omega_{k,\varepsilon}),\tilde{I}(\omega_\varepsilon)$ are still defined and finite. By the Lebesgue differentiation theorem $\omega_{k,\varepsilon},\omega_{\varepsilon}$ converge a.e.\ to $\omega_k,\omega$ respectively. Since the surface tension $\sigma:\mc N(\Lambda)\to \m R$ is continuous, this implies that $I(\omega_k) =\tilde{I}(\omega_k) = \lim_{\varepsilon\to 0} \tilde{I}(\omega_{k,\varepsilon})$ and $I(\omega) =\tilde{I}(\omega) = \lim_{\varepsilon\to 0} \tilde{I}(\omega_{\varepsilon})$.  For $\varepsilon$ fixed these flows are continuous and $d_W(\omega_{k,\varepsilon},\omega_\varepsilon)\to 0$ as $k\to \infty$. Since convergence in $d_W$ implies weak convergence (Proposition~\ref{Wassequalsweak}), $\omega_{k,\varepsilon}$ converges pointwise to $\omega_\varepsilon$, and hence $\lim_{k\to\infty}\tilde{I}(\omega_{k,\varepsilon})=\tilde{I}(\omega_\varepsilon).$ On the other hand since $\sigma$ is convex, 
\begin{align*}
    \sigma(\omega_{k,\varepsilon}(x)) = \sigma\bigg(\frac{1}{|B_\varepsilon(x)|}\int_{B_\varepsilon(x)} \omega_k(y)\, \dd y\bigg) \leq \frac{1}{|B_\varepsilon(x)|} \int_{B_{\varepsilon}(x)} \sigma(\omega_k(y))\,\dd y.
\end{align*}
Therefore there exists a constant $C>0$ such that 
\begin{align*}
    \tilde{I}(\omega_{k,\varepsilon}) \leq I(\omega_k) + C\varepsilon.
\end{align*}
Therefore 
\begin{align*}
    \liminf_{k\to \infty} I(\omega_k) \geq \liminf_{k\to \infty} \tilde{I}(\omega_{k,\varepsilon}) - C \varepsilon = \tilde{I}(\omega_\varepsilon) - C\varepsilon.
\end{align*}
Taking $\varepsilon\to 0$ and using that $\lim_{\varepsilon\to0}\tilde{I}(\omega_\varepsilon) = I(\omega)$ shows that $I$ is lower semicontinuous. Since $\text{AF}(\Lambda,R,b)$ is compact (Corollary~\ref{cor:AFboundarycompact}) it follows that the level sets are compact.
\end{proof}

Standard arguments (see e.g.\ \cite{varadhan2016large}) imply that the large deviation principle of Theorem~\ref{thm:ldp} is equivalent to upper and lower bound statements for open balls in $AF(\Lambda,R,b)$, plus the \textit{exponential tightness} property, that is: for any $c<\infty$ there exists a compact set $F_c\in AF(\Lambda,R)$ such that for any closed set $K$ with $K\cap F_c =\emptyset$, 
\begin{align*}
    \limsup_{n\to \infty} \limsup_{N\to \infty} \frac{1}{Nn^d}\log \rho_{n,N}(K) \leq -c. 
\end{align*}
{Exponential tightness follows by taking $F_c =\{\omega\in \text{AF}(\Lambda,R,b):I(\omega)\leq c\}$, the level set of the rate function, which is compact by Lemma~\ref{lem:semicontinuity}.}

For the upper and lower bound statements, it suffices to consider neighborhoods in the topology. For $\omega\in \text{AF}(\Lambda,R)$ we define: 
\begin{align}
    U_r(\omega) = \{ \omega': d_W(\omega,\omega')<r\}.
\end{align}
We let $Z_{n,N}$ denote the unnormalized partition function of $\rho_{n,N}$, and further for any set $U$ of possible samples from $Z_{n,N}$, let $Z_{n,N}(U)$ be partition function restricted to $U$. Analogous to Section~\ref{sec:torus_asymptotics}, we normalize the partition functions by dividing by $(N!)^{|R|n^d}$, where $|R|Nn^d$ is approximately the number of dimers in the cover (approximate instead of exact because samples from $\rho_{n,N}$ do not all have the same boundary conditions).

To prove Theorem \ref{thm:ldp} it remains to prove the following two bounds for the asymptotics of the normalized partition functions. {Recall that $\Ent$ differs from $I$ by a sign and an additive constant, see \eqref{eq:total_surface_tension}.}
\begin{lemma}[Lower bound]\label{lem:lower}
    For any $\omega\in \text{AF}(\Lambda,R,b)$, 
    \begin{align*}
        \lim_{r\to 0}\liminf_{n\to \infty}\liminf_{N\to\infty} \frac{1}{N n^d}\log \bigg[\frac{Z_{n,N}(U_{r}(\omega))}{(N!)^{|R|n^d}}\bigg] \geq \Ent(\omega). 
    \end{align*}
\end{lemma}

\begin{lemma}[Upper bound]\label{lem:upper}
     For any $\omega\in \text{AF}(\Lambda,R,b)$, 
    \begin{align*}
        \lim_{r\to 0}\limsup_{n\to \infty}\limsup_{N\to\infty} \frac{1}{Nn^d}\log\bigg[\frac{Z_{n,N}(U_{r}(\omega))}{(N!)^{|R|n^d}}\bigg] \leq \Ent(\omega). 
    \end{align*}
\end{lemma}

In Section~\ref{sec:approximations}, we prove two approximation theorems which are needed for the proof of the lower bound. After that, in Section~\ref{sec:lower} we prove the lower bound and in Section~\ref{sec:upper} we prove the upper bound. The other key tool is the $n,N\to \infty$ patching theorem (Theorem~\ref{thm:patching_limit}), which we use to patch samples from the torus measures defined in Section~\ref{sec:torus_asymptotics} with dimer covers sampled from $\rho_{n,N}$.

\begin{rem}\label{rem:reflection}
    Recall from Section~\ref{sec:subgraphs_boundaryconditions} that the lattice $\Lambda$ is assumed to be reflection symmetric. The only place this is used is to simplify the proof of Lemma~\ref{lem:upper}. In particular, the reflection symmetry gives us an elementary way to embed a piece of a dimer cover of $R_n$ into the torus, and therefore to relate it to the results in Section~\ref{sec:torus_asymptotics}. The reflection symmetry assumption can be lifted, for example, if a more robust patching theorem (instead of Theorem~\ref{thm:patching_limit}) were established for general lattices. A sufficiently robust patching theorem has been established, for example, for dimer covers of $\m Z^3$  in \cite{3Ddimers}. (This step can also be achieved easily for the honeycomb lattice using height function extension results.)
\end{rem}

\subsection{Approximation theorems}\label{sec:approximations}

\subsubsection{Discrete approximation} 

The main result is of this section is that any asymptotic flow $\omega$ can be approximated by discrete flows corresponding to $N$-dimer covers. This is a converse to Theorem~\ref{thm:scaling_limits_asymptotic}. 

From the critical gauge equations, it is straightforward to show that any discrete flow with the right divergences can be approximated by flows corresponding to $N$-dimer covers.

\begin{lemma}\label{lem:anything_is_limit}
    Fix $G=(V,E)\subset \Lambda$ and $(\beta_v)_{v\in \partial G}$  a feasible limiting boundary condition. Suppose that $\omega$ is a discrete flow on $G$ with divergences $\pm 1$ at interior vertices and $\pm \beta_v$ at boundary vertices (in both cases the sign is $+$ at white vertices and $-$ at black vertices), and which has $\omega(e)\neq 0$ for all $e\in E$. Then there exists a sequence of $N$-dimer covers $M_N$ of $G$ such that $\sup_{e\in E} |\omega_{M_N}(e) - \omega(e)| \to 0$ as $N\to \infty$. 
\end{lemma}

\begin{proof}
   For each $v\in V$, choose $N_v$ so that $(N_v)_{v\in\partial G}$ is a feasible boundary condition and $N_v/N\to \beta_v$ as $N\to \infty$. Consider the multinomial dimer measures $\mu_N$ with these multiplicities and edge weights $\omega(e)$; this is a valid edge weight function since $\omega(e)\neq0$ for all $e\in E$.
   
    Then $\mathbf x\equiv 1$ solves the critical weight equations, and hence by Theorem~\ref{thm:largeN_concentration}, the limit shape for samples from $\mu_N$ as $N\to \infty$ is $(\omega(e))_{e\in E}$. The result follows by sampling $M_N$ from $\mu_N$. 
\end{proof}

\begin{thm}\label{thm:discrete_approximation}
    Fix $\delta>0$ and any $\omega\in \text{AF}(\Lambda, R,b)$ with $b$ which is extendable outside. Let $R_n\subset \frac{1}{n}\Lambda$ be a lattice approximation of $R$ as before. For all $n,N$ large enough, there exists an $N$-dimer cover $M$ of $R_n$, with boundary vertex multiplicities $(N_v)_{v\in \partial R_n}$ between $1$ and $N$, such that $d_W(\omega_M,\omega)<\delta.$ 
\end{thm}
\begin{rem}\label{rem:thresholds}
    In particular, since the boundary value operator is uniformly continuous, this theorem shows that any $\omega$ can be approximated by an $N$-dimer cover in the support of $\rho_{n,N}$, as long as the thresholds $\theta_n$ used to define $\rho_{n,N}$ go to zero sufficiently slowly. 
\end{rem}

We note any $\omega\in \text{AF}(\Lambda,R,b)$ can be approximated by a smooth flow, and then give a construction for the discrete approximation when $\omega$ is $C^2$.

\begin{lemma}\label{lem:smoothing_asymptotic_flows}
Suppose that $(R,b)\subset \m R^d$ and that $b$ is extendable outside and fix $\varepsilon>0$ small enough. Any asymptotic flow $\omega\in \mathrm{AF}(\Lambda,R,b)$ can be approximated in $d_W$ by a smooth flow $\omega_{\mathrm{sm}}\in \mathrm{AF}(\Lambda,R)$, where $\omega_{\mathrm{sm}}$ is constructed by mollifying $\omega$ with a bump function $\phi$ supported in $B_\varepsilon(0)$, and  has $d_W(\omega,\omega_{\mathrm{sm}}) < C \varepsilon^{1/2}$ for some constant $C>0$. Further we can assume that $\omega_{\mathrm{sm}}$ is valued strictly in the interior of $\mc N(\Lambda)$.
\end{lemma}
\begin{rem}
    Note that the boundary conditions of $\omega_{\mathrm{sm}}$ are allowed to be different from those of $\omega$. However, since the boundary value operator $T$ is uniformly continuous, the bound $d_W(\omega,\omega_{\mathrm{sm}})<C \varepsilon^{1/2}$ implies a bound on the difference in their boundary conditions.
\end{rem}
\begin{proof}
    Since $b$ is extendable outside there exists $\lambda>0$ such that $b$ can be extended to an asymptotic flow on an $\lambda$ neighborhood of $R$ (i.e., divergence-free and valued in $\mc N(\Lambda)$). Let $\phi$ be a smooth bump function supported on $B_{\varepsilon}(0)\subset\R^d$ which integrates to $1$, with $\varepsilon<\lambda$. We define a smooth approximation $\omega_{\text{sm}} := \omega\ast \phi$ restricted to $R$. Since $\varepsilon<\lambda$, $\omega_{\text{sm}}$ is divergence-free on $R^\circ$. Since $\mc N(\Lambda)$ is convex, $\omega_{\text{sm}}$ is valued in $\mc N(\Lambda)$.
    
    To bound $d_W(\omega,\omega_{\text{sm}})$, first note there is a constant $c_0$ such that the total flow of $\omega,\omega_{\text{sm}}$ may differ by at most $c_0 \varepsilon$. Next for $\delta>0$ fixed, there is a constant $c_1$ such that for any box $B$ of side length $\delta$ contained in $R$, $$d_W(\omega\mid_{B},\omega_{\text{sm}}\mid_B) < c_1 \varepsilon \delta^{d-1},$$ since the averaging procedure to construct $\omega_\text{sm}$ moves flow by distance at most $\varepsilon$. We can cover $R$ with $O(\delta^{-d})=c_2\delta^{-d}$ boxes of side length $\delta$. Hence by Lemma~\ref{lem:wasserstein_bound_from_local}, 
    \begin{align*}
        d_W(\omega,\omega_\text{sm}) < c_0 \varepsilon + c_2 d\delta^{-d}(10\delta^{d+1}+c_1 \varepsilon \delta^{d-1}).
    \end{align*}
    Since $\delta$ is arbitrary, we can take it to be on the order of $\varepsilon^{1/2}$, which gives $d_W(\omega,\omega_{\mathrm{sm}})<C \varepsilon^{1/2}$ for some constant $C>0$. Finally, if $\omega_{\mathrm{sm}}$ itself is not valued strictly in the interior of $\mc N(\Lambda)$, then $(1-\kappa)\omega_{\mathrm{sm}}$ is for $\kappa$ sufficiently small. If we take $\kappa \leq O(\varepsilon^{1/2})$, then the distance between $\omega$ and $(1-\kappa)\omega_{\mathrm{sm}}$ is still of order $\varepsilon^{1/2}$, and without loss of generality we can replace $\omega_{\mathrm{sm}}$ with $(1-\kappa)\omega_{\mathrm{sm}}$.  
\end{proof}

Suppose that $\omega\in \text{AF}(\Lambda,R,b)$ is $C^2$. Recall that  $e_1,\dots, e_D$ are the edge vectors of edges incident to a white vertex. Recall from Section~\ref{sec:torus_asymptotics} that given a slope $s\in \mc N(\Lambda)$, there are probabilities $p_1(s),\dots,p_D(s)$ corresponding to the $D$ edge types which as edge weights are gauge equivalent to the uniform weights, and have the property that
    \begin{align}\label{eq:slope_as_probabilities}
        s = \sum_{i=1}^D p_i(s) \, e_i.
    \end{align}
    We define a discrete flow $\omega_{\text{{white}},n}$ on $R_n$ such that for each edge $e=(u,v)$ in $R_n$ of type $e_i$ and with $v\in W\cap R_n$, we set
    \begin{align}\label{eq:def_disc_approx}
        \omega_{\mathrm{white},n}(e) = p_i(\omega(v)),
    \end{align}
    for $e$ oriented white-to-black. If $e=(u,v)$ is not contained in $R_n$, then $\omega_{\mathrm{white},n}(e) = 0$. We similarly define $\omega_{\mathrm{black},n}$ where the role of black and white vertices is reversed. These approximations have the following properties.
\begin{lemma}\label{lem:divergences_of_approx}
    Fix $\omega\in \text{AF}(\Lambda,R)$ which is $C^2$. Then 
    \begin{enumerate}[(a)]
        \item For all interior black vertices $u$, $\mathrm{div}\, \omega_{\mathrm{black},n}(u) = -1$, and for all interior white vertices $v$, $\mathrm{div}\, \omega_{\mathrm{white},n}(v) = 1$. 
        \item For all interior white vertices $v$, $\mathrm{div}\, \omega_{\mathrm{black},n}(v) = 1 + O(1/n^2)$, and for all interior black vertices $u$, $\mathrm{div}\, \omega_{\mathrm{white},n}(u) = -1+O(1/n^2)$.
        \item For all edges $e\in E(R_n)$, $|\omega_{\mathrm{black},n}(e)-\omega_{\mathrm{white},n}(e)|=O(1/n)$.
        \item The distances $d_W(\omega,\omega_{\mathrm{black},n}), d_W(\omega,\omega_{\mathrm{white},n})$ are $O(1/n)$.
    \end{enumerate}
    \end{lemma}
\begin{proof}
    Part (a) follows immediately from \eqref{eq:def_disc_approx} and the fact that $\sum_i p_i(s) = 1$ for any slope $s$. Part (c) follows since $\omega:R\to \m R^d$ is $C^2$ and $p_i(s)$ is a continuous function of $s$.
    
    For (b) consider just $\omega_{\mathrm{black},n}$; the proof for $\omega_{\mathrm{white},n}$ is analogous. For all $v\in R_n^\circ\cap W$, we have
    \begin{align}\label{eq:divergence_at_pt}
        \text{div}\, \omega_{\mathrm{black},n}(v) = \sum_{i=1}^D p_i(\omega(v+{\frac{1}{n}} e_i)).
    \end{align}
    Let $\eta_1,\dots, \eta_d$ be the basis vectors for $\m R^d$, and define $a_i^j$ so that $e_i = \sum_{j=1}^d a_i^j \eta_j$. We expand \eqref{eq:divergence_at_pt} in ${1}/{n}$: 
    \begin{align*}
        \sum_{i=1}^D p_i(\omega(v+\frac{1}{n} e_i)) = \sum_{i=1}^D p_i(\omega(v)) + \frac{1}{n} \sum_{i=1}^D\frac{\partial(p_i\circ \omega)}{\partial e_i}(v) + O(1/n^2) = 1 + \frac{1}{n} \sum_{i=1}^D \sum_{j=1}^d a_i^j \frac{\partial (p_i\circ \omega)}{\partial \eta_i}(v) + O(1/n^2). 
    \end{align*}
    On the other hand by \eqref{eq:slope_as_probabilities}, the $j^{th}$ component $\omega_j$ of $\omega$ is
    \begin{align*}
        \omega_j(x) = \sum_{i=1}^D a_i^j \, p_i(\omega(x)).
    \end{align*}
    Thus for any $v\in \partial R_n^\circ\cap W$,
    \begin{align*}
        \text{div}\, \omega_{\mathrm{black},n}(v) = \sum_{i=1}^D p_i(\omega(v+\frac{1}{n} e_i)) = 1 + \frac{1}{n}\, \text{div}\,\omega(v) + O(1/n^2) = 1 +O(1/n^2).
    \end{align*}
    The second equality uses that $\omega$ is divergence-free in $R^\circ$. 

    Finally part (d) follows from the continuity of $\omega$, a simple averaging argument and Lemma~\ref{lem:wasserstein_bound_from_local}.
    
    %let $A_i$ denote the Voronoi cells of black vertices in $R_n$ and apply Lemma~\ref{lem:wasserstein_bound_from_local} to $\omega_{\text{black},n}$ and $\omega$ on this partition.
\end{proof}

\begin{proof}[Proof of Theorem~\ref{thm:discrete_approximation}] By Lemma~\ref{lem:anything_is_limit}, it suffices to show that $\omega$ can be approximated by a discrete flow with appropriate divergences. By Lemma~\ref{lem:smoothing_asymptotic_flows}, for any $\varepsilon>0$ small enough, we can assume that $\omega$ is smooth and valued in $(1-\varepsilon) \mc N(\Lambda)$, at the cost of moving $\omega$ by $O(\varepsilon^{1/2})$ in $d_W$ distance. Assuming $\omega$ is smooth, the construction of $\omega_{\mathrm{black},n}$ (and $\omega_{\mathrm{white},n}$) is well-defined. 
    
We define the discrete approximation $\omega_n$ from $\omega_{\mathrm{black},n}$ as follows. First, partition the vertices of $R_n$ into paths from $\partial R_n$ to $\partial R_n$ of length $O(n)$. At each white vertex $v\in W\cap R_n^\circ$, by Lemma~\ref{lem:divergences_of_approx}(b) we can correct the divergence of $\omega_{\text{black},n}$ by subtracting $O(1/n^2)$ from its value on incident edges. If we subtract $c/n^2$ from $e=(u,v)$, then we add this amount to another edge $e'$ incident to the black vertex $u$, so that the divergence at all interior black vertices remains the same. We iterate this along each path, making all divergences at interior vertices $1$, in sum perturbing the divergences at the boundary by $O(1/n)$. Since $\omega$ is valued in $(1-\varepsilon)\mc N(\Lambda)$, $\omega_{\mathrm{black},n}(e)>c(\varepsilon)>0$ for all edges, so for $n$ large enough this process gives a discrete flow which is nonzero on all edges. 

{For the boundary vertices, we check that the divergences are in $(0,1]$ in magnitude. Since $\omega$ is extendable outside, $\omega_{\mathrm{black},n}$ and $\omega_{\mathrm{white},n}$ have this property at $B\cap \partial R_n$ and $W\cap \partial R_n$ respectively. In fact since $\omega$ is valued in $(1-\varepsilon)\mc N(\Lambda)$, $\omega_{\mathrm{black},n}$ and $\omega_{\mathrm{white},n}$ have divergences of magnitude in $[c_1,1-c_2]$ on $B\cap \partial R_n$ and $W\cap \partial R_n$ respectively (as both are $0$ on edges which are not contained in $R_n$), where $c_1,c_2$ are small constants determined by $\varepsilon$. By Lemma~\ref{lem:divergences_of_approx}(c), $\omega_n$ has divergences of magnitude between $c_1 - O(1/n)$ and $1-c_2+O(1/n)$ for all $v\in \partial R_n$, hence $\omega_n$ has the desired property for $n$ large enough. } 

It follows that $\omega_n$ satisfies the conditions of Lemma~\ref{lem:anything_is_limit}. We note that $\omega_n$ approximates $\omega$ in $d_W$ by Lemma~\ref{lem:divergences_of_approx}(d) and the fact that we changed the values on each edge by at most $O(1/n)$, and smoothing with Lemma~\ref{lem:smoothing_asymptotic_flows} moves $\omega$ by an $O(\varepsilon^{1/2})$ distance; for any fixed $\delta>0$, $\varepsilon>0$ and $n$ can be taken small and large enough respectively so that $d_W(\omega_n,\omega)<\delta$.
\end{proof}

As a corollary of the proof, we also note that the discrete flow $\omega_n$ we construct approximating $\omega$ is very close to $\omega_{\text{black},n}$ and $\omega_{\text{white},n}$ in the sup norm on $E(R_n)$. 
\begin{cor}\label{cor:other_approximation}
    The discrete approximations $\omega_n,\omega_{\mathrm{black},n}$, $\omega_{\mathrm{white},n}$ satisfy
    \begin{align*}
        \sup_{e\in E(R_n)} |\omega_n(e) - \omega_{\mathrm{black},n}(e)| = \sup_{e\in E(R_n)} |\omega_n(e) - \omega_{\mathrm{white},n}(e)| = O(1/n).
    \end{align*}
\end{cor}

\subsubsection{Piecewise-constant approximation}

Fix $R\subset \m R^d$ and a scale $\delta>0$. Let $\mc X$ be a collection of $d$-dimensional simplices (i.e.\ triangles or tetrahedra in two or three dimensions) with diameters {of order} $\delta$ and disjoint interiors, all of which intersect $R$, and such that $R\subset \cup_{X\in \mc X} X$. We assume that the simplices $X\in \mc X$ are all translations of a finite collection of shapes {(e.g., those with vertices in an appropriate lattice)}. 

We say that an asymptotic flow $\omega$ is \textit{piecewise constant} on $\mc X$ if for each $X\in \mc X$, $\omega\mid_{X}$ is constant. We show that any asymptotic flow can be approximated by a piecewise constant one (with slightly perturbed boundary conditions).

\begin{prop}\label{prop:pc_approx}
Fix $\varepsilon>0$ and $(R,b)\subset \m R^d$ with $b$ extendable outside. 
Let $\mc X$ be a mesh of simplices as above, with $\delta>0$ small enough. For any $\omega\in \mathrm{AF}(
\Lambda,R,b)$, there exists $\widetilde{\omega}$ which is piecewise constant on $\mc X$, has $d_W(\omega,\widetilde{\omega})<\varepsilon$, is valued strictly in the interior of $\mc N(\Lambda)$, and which has $\widetilde{\omega}\in \mathrm{AF}(\Lambda,R)$. 
\end{prop}

The proof of Proposition~\ref{prop:pc_approx} is analogous to that of \cite[Proposition 8.3.1]{3Ddimers}. 

\begin{proof}

By Lemma~\ref{lem:smoothing_asymptotic_flows}, it suffices to prove this result for $\omega$ smooth. Note that since $b$ is extendable outside, we can take $\delta>0$ small enough so that $\omega$ is smooth and divergence free on all $X\in \mc X$. For any fixed simplex $X\in \mc X$, $\partial X$ consists of $(d+1)$ faces. Let $F_1,\dots,F_{d+1}$ denote the faces of $\partial X$ and let $\xi_1,\dots,\xi_{d+1}$ denote their outward-pointing unit normal vectors. We define the vector $s_X$ by the condition that 
\begin{align*}
    s_X \cdot \xi_i = \frac{1}{|F_i|}\int_{F_i} \langle \omega, \xi_i\rangle \, \dd x, \qquad \text{ for } i\in \{1,\dots, d-1\}.
\end{align*}
{Since $\omega$ is divergence free on $X$, the divergence theorem gives
$$\sum_{i=1}^d\int_{F_i} \langle \omega, \xi_i\rangle \, \dd x=0.$$
Moreover since $\sum_{i=1}^d\xi_i|F_i|=0$,} it follows that 
\begin{align*}
    s_X \cdot \xi_d = \frac{1}{|F_d|}\int_{F_d} \langle \omega, \xi_d\rangle \, \dd x.
\end{align*}
In particular, $s_X$ does not depend on the order of $\xi_1,\dots,\xi_d$. We define $\widetilde{\omega}(x) := s_X$ for all $x\in X$. This is valued in $\mc N(\Lambda)$, {continuous and} divergence-free on $R$, hence its restriction to $R$ is in $\mathrm{AF}(\Lambda,R)$. 

It remains to bound $d_W(\widetilde{\omega},\omega)$. Let $m$ denote the modulus of continuity of $\omega$ on $\widetilde{R}\supset R$. Since we assume that all simplices $X\in \mc X$ are translations of a finite collection, there are finitely many normal vectors $\xi_i$ ranging over all faces of $\partial X$ for all $X\in \mc X$. Thus there is a constant $K$ such that for all $X\in \mc X$, 
\begin{align}\label{eq:pc_vs_avg}
    |s_X - \text{avg}_X(\omega)| \leq K d m \delta,
\end{align}
where $s_X= \widetilde{\omega}\mid_X$. Hence by Lemma~\ref{lem:wasserstein_bound_from_local}, for any $\delta$ small enough given $\omega,\varepsilon$ we have $d_W(\omega,\widetilde{\omega})<\varepsilon$. We can replace $\widetilde{\omega}$ with itself times $1$ minus a small constant (of order at most $\varepsilon$); this modification will satisfy the same properties as above and be valued strictly in the interior of $\mc N(\Lambda)$.

\end{proof}

The piecewise constant approximations of Proposition~\ref{prop:pc_approx} have the following useful property. 

\begin{prop}\label{prop:pc_ent_approximation}
    Fix $\omega\in \text{AF}(\Lambda,R)$. If $\tilde{\omega}_\delta$ is a piecewise constant approximation of $\omega$ on $\delta$ simplices as in Proposition~\ref{prop:pc_approx}, then $\Ent(\widetilde{\omega}_\delta) = \Ent(\omega) + o(1)$ as $\delta\to 0$.
\end{prop}
\begin{proof}
    By Lemma~\ref{lem:semicontinuity}, $\Ent$ is upper semicontinuous, so $\limsup_{\delta\to 0}\Ent(\widetilde{\omega}_\delta) \leq \Ent(\omega)$, and hence $\Ent(\widetilde{\omega}_\delta) \leq \Ent(\omega) + o(1).$

For the other direction, we use the convexity of $\sigma$. Let $\mc X^\delta$ denote the collection of simplices at scale $\delta$ on which $\widetilde{\omega}_\delta$ is constant (we assume that for all $\delta$, $\mc X^\delta$ consists of translations of rescaled versions of the same collection of simplices, so that the collection of normal vectors remains the same for all $\delta$).  Define another flow $\omega_{\mathrm{avg},\delta}$ by
    \begin{align*}
        \omega_{\mathrm{avg},\delta}(x) := \frac{1}{|X|}\int_{X} \omega(y)\,\dd y \qquad \text{ for all } x\in X, \text{ for all } X\in \mc X^\delta,
    \end{align*}
    where we set $\omega(y) = 0$ for $y\not\in R$. While this may not be an asymptotic flow (it may not be divergence-free), it is still valued in $\mc N(\Lambda)$. Since $\sigma$ is convex, 
    \begin{align*}
        \sigma\bigg( \frac{1}{|X|}\int_{X} \omega(y)\,\dd y \bigg) \leq \frac{1}{|X|}\int_X \sigma(\omega(y))\, \dd y.
    \end{align*}
     Since $\sigma$ is continuous, \eqref{eq:pc_vs_avg} implies that 
    \begin{align*}
        \Ent(\tilde{\omega}_\delta) +o(1) = -\int_{R}\sigma(\omega_{\text{avg},\delta}(x))\, \dd x. 
    \end{align*} 
    Hence
    \begin{align*}
        \Ent(\omega) = -\int_{R} \sigma(\omega(x)) \, \dd x \leq - \int_{R} \sigma(\omega_{\text{avg},\delta}(x)) \, \dd x = \Ent(\tilde{\omega}_\delta) +o(1),
    \end{align*}
    which completes the proof.
\end{proof}

\subsection{Proof of the lower bound}\label{sec:lower}

\begin{proof}[Proof of Lemma~\ref{lem:lower}]
    Fix $\omega\in \text{AF}(\Lambda,R,b)$. By Proposition~\ref{prop:pc_approx}, we can approximate $\omega$ by $\widetilde{\omega}$ which is piecewise-constant on a mesh $\mc X$ of simplices with diameters of order $\delta$ and valued in the interior of $\mc N(\Lambda)$. Let $s_X :=\widetilde{\omega}\mid_X\in \mc N(\Lambda)^\circ$ for each $X\in \mc X$. By Proposition~\ref{prop:pc_ent_approximation},
    \begin{align}
        \Ent(\omega)=\Ent(\widetilde{\omega}) + o(1) = -\sum_{X\in \mc X} |X|\sigma(s_X) + o(1) 
    \end{align}
    as $\delta\to 0$. Further, for $\delta$ small enough, we have $U_{r/2}(\widetilde{\omega})\subset U_{r}(\omega)$. It therefore suffices to prove the lower bound for $\widetilde{\omega}$. 
    
   For $n,N$ large enough there exists an $N$-dimer cover $M$ of $R_n$ (with boundary conditions allowed by $\rho_{n,N}$) as in Theorem~\ref{thm:discrete_approximation} such that $d_W(\omega_M,\widetilde{\omega})<r/2$. In fact we can find $M$ which has a particular form, which is a modification of $\omega_{\text{black},n}$ approximating $\widetilde{\omega}$ as in the previous section. 
    
    As a technical point, $\widetilde{\omega}$ is not $C^2$, and thus we cannot immediately construct $\omega_{\text{black},n}$ from it. However, we can mollify it by convolving it with a bump function $\phi$ supported in $B_{\delta'}(0)$ for $\delta'\ll\delta$. The flow $\widetilde{\omega}\ast \phi$ will be smooth (and still close to $\widetilde{\omega}$ by Lemma~\ref{lem:smoothing_asymptotic_flows}), and in fact it will be equal to $\widetilde{\omega}$ outside a $2\delta'$ neighborhood of the set of boundaries $\{\partial X:X\in \mc X\}$. We construction $\omega_M$ from $\omega_{\text{black},n}$ from $\widetilde{\omega}\ast \phi$.

    Recall that a slope $s_X$ determines a collection of edge probabilities, $\{p_i(s_X)\}_{i=1}^D$. Since $s_X\in \mc N(\Lambda)^\circ$, $p_i(s_X)\neq 0$ for all $i$. (Concretely, $p_i(s_X) = \exp(e_i\cdot \nabla \sigma(s_X))/\sum_{k=1}^D \exp(e_k\cdot \nabla \sigma(s_X))$, see Theorem~\ref{thm:legendre_duality}.)

Let $X_n = \frac{1}{n}\Lambda\cap X$ and let $N_{2\delta'}(X)$ be the width-$2\delta'$ tubular neighborhood of $\partial X$. Since $\widetilde{\omega}\ast \phi$ is equal to $\widetilde{\omega}$ on $X\setminus N_{2\delta'}(X)$, $\omega_{\text{black},n}(e) = p_i(s_X)$ for all edges $e$ of type $e_i$ in $X_n^{2\delta'} = \frac{1}{n}\Lambda \cap (X\setminus N_{2\delta'}(\partial X))$. By Corollary~\ref{cor:other_approximation}, $\sup_{e\in E_n}|\omega_M(e)-\omega_{\text{black},n}(e)|=O(1/n)$, so $\omega_M(e)=p_i(s_X)+O(1/n)$ for all edges $e$ of type $e_i$ in $X_n^{2\delta'}$.

On the other hand, by Proposition~\ref{prop:limit_shape_torus}, the discrete flow $p(e) = p_i(s_X)$ for all edges $e$ of type $e_i$ is the limit shape for the torus measure with periodic edge weights $w_i=p_i(s_X)$ on each edge of type $e_i$. Let $Z_{X_n,N,w}^{\text{torus}}$ denote the (unnormalized) partition function of the torus measure with these weights on $X_n$ and let $\mu_{n,N,w}$ denote the corresponding probability measure. 

Fix $\varepsilon>0$. By Proposition~\ref{prop:limit_shape_torus} and Corollary~\ref{cor:other_approximation} respectively, $\omega_M$, and samples from $\mu_{n,N,w}$ for $N$ large enough given $\varepsilon$, have boundary conditions on $\partial X_n^{2\delta'}$ within $O(1/n)$ of each other  with probability $1-\varepsilon$ (they are both close to the discrete flow $p(e)$.) Further, since $s_X \in \mc N(\Lambda)^\circ$, $p(e)\neq 0$ for all edges $e$ in $X_n^{2\delta'}$, so $\mu_{n,N,w}$ satisfy the strictly non-degenerate condition. In summary, on $X_n^{2\delta'}$, the dimer cover $\omega_M$ (outside) and a sample from $\mu_{n,N,w}$ (inside), satisfy the conditions of the patching theorem (Theorem~\ref{thm:patching_limit}), on an annulus of width $\delta'$.

As such, for $n,N$ large enough, the dimer cover $\omega_M$ outside $X_n^{2\delta'}$ can be patched with a sample from $\mu_{n,N,w}$ restricted to $X_n^{3\delta'}$ with probability $1-\varepsilon$. We apply this on all $X\in \mc X$. By Proposition~\ref{prop:limit_shape_torus}, for $n,N$ large enough we can assume that all but $\varepsilon$ proportion of these dimer covers constructed by patching are in $U_{r/2}(\widetilde{\omega})$. 

To compute a lower bound on the number of dimer covers in $U_{r/2}(\widetilde{\omega})$ from this patching construction, recall as in Theorem~\ref{cor:almost_legendre} that restricted to the homology class $s_X$, the weights $w$ are gauge equivalent to the uniform edge weights, and hence there are constants $C_{n,N}(w)$ such that $Z_{n,N,w}^{s_X} = C_{n,N}(w) Z_{n,N}^{s_X}$. Therefore for $r'$ determined by $r$, for $n,N$ sufficiently large, we have for all $X_n$,
\begin{align*}
    \log Z_{X_n,N}^{\text{torus}}(U_{r'}(s_X)) \geq \log Z_{X_n',N}^{s_X} + O(\varepsilon \log \varepsilon),
\end{align*}
where $X_n'=X_n^{3\delta'}$ (the region we patched into). Therefore for $n,N$ large enough,
\begin{align*}
    \log Z_{n,N}(U_{r}({\omega})) \geq \log Z_{n,N}(U_{r/2}(\widetilde{\omega})) \geq \sum_{X\in \mc X} \log Z_{X_n',N}^{s_X} + O(\varepsilon \log \varepsilon),
\end{align*}
On the other hand,
      \begin{align*}
          \sigma(s_X)= -\lim_{n\to\infty}\lim_{N\to\infty}\frac{1}{Nn^d|X|}\log \frac{Z_{X_n,N}^{s_X}}{(N!)^{n^d|X|}} 
      \end{align*} 
      Hence
      \begin{align*}
          \liminf_{n\to\infty} \liminf_{N\to\infty}\frac{1}{N n^d} \log \frac{Z_{n,N}(U_{r}({\omega}))}{(N!)^{n^d|R|}} &\geq -\sum_{X\in \mc X} \sigma(s_X) |X|(1-O(\delta')) + O(\varepsilon \log \varepsilon) \\
          &= \Ent({\omega})(1-O(\delta')) + o_\delta(1)+O(\varepsilon \log \varepsilon).
      \end{align*}
     Here $\varepsilon,\delta,\delta'$ are small constants which go to $0$ as $r\to 0$. This completes the proof.
\end{proof}

\subsection{Proof of the upper bound}\label{sec:upper}

\begin{proof}[Proof of Lemma~\ref{lem:upper}]

    Let $\widetilde{\omega}$ be a piecewise constant approximation of $\omega$ as in Proposition~\ref{prop:pc_approx}. Let $\mc X$ be the mesh of simplices on which $\widetilde \omega$ takes constant values, $s_X:=\widetilde{\omega}\mid_{X} \,\in \mc N(\Lambda)^\circ$. By Proposition~\ref{prop:pc_ent_approximation}, $\Ent(\widetilde{\omega}) = \Ent(\omega) + O(\delta)$, where $\delta$ is the scale of the simplices in $\mc X$. It suffices to prove the upper bound for $\widetilde{\omega}$.    
    
    For each $X\in \mc X$, let $\rho_{X_n,N}^{\text{free}}$ be the multinomial measure on $N$-dimer covers of $X_n =X\cap \frac{1}{n}\Lambda$ with \textbf{free} boundary conditions. Let $Z_{X_n,N}^{\text{free}}$ be the corresponding unnormalized partition function. Recall that $U_r(\cdot)$ denotes a ball of radius $r$ in the Wasserstein metric on flows. We have the bound: 
   \begin{align}\label{eq:reduce_to_simplex}
       \frac{1}{Nn^d} \log \frac{Z_{n,N}(U_r(\widetilde{\omega}))}{(N!)^{n^d|R|}} \leq \frac{1}{Nn^d} \sum_{X\in \mc X} \log \frac{Z_{X_n,N}^{\text{free}}(U_r(s_X))}{(N!)^{n^d|X|}}.
   \end{align}
    It therefore suffices to understand the asymptotics of $Z_{X_n,N}^{\text{free}}(U_r(s_X))$ for one $X\in \mc X$. We can further cover $X_n$ with boxes $B_m$ with $(2m)^d$ vertices (i.e., $B_m\cap W$ is a translation of $[0,m-1]^d$). We can take $m$ to be of the order $n^{1/2}$, so that only a lower-order proportion of these boxes intersect the boundary of $X_n$. Hence it suffices to understand the asymptotics of $Z_{B_m,N}^\text{free}(U_r(s_X))$, where $B_m$ is a box. Let $\rho_{B_m,N}^\text{free}$ be the corresponding probability measure. We split this by boundary condition, i.e. consider
    \begin{align*}
        Z_{B_m,N}^{\text{free}}(U_r(s_X)) = \sum_{\beta} \m P(\beta) Z_{B_m,N}^\beta(U_r(s_X)),
    \end{align*}
    where $\beta$ is a choice of multiplicities for vertices in $\partial B_m$, and $\m P(\beta)$ is the probability of $\beta$ under $\rho_{B_m,N}^\text{free}$.

    Since $\Lambda$ is reflection symmetric, we can split the torus $\m T_{2m}$ with $(4m)^d$ sites into $2^d$ copies of $B_m$. We put the boundary condition $\beta$ on one copy of $B_m$, and the reflected boundary condition $\beta'$ on all neighboring copies of $B_m$ in $\m T_{2m}$. We repeat this until all copies have a boundary condition, and sample $2^d$ independent dimer covers with boundary condition $\beta$ (or $\beta'$) in each copy. We then add multiplicity to the remaining edges of $\m T_{2m}$, which go between the reflected copies of $B_m$, so that all vertices have total multiplicity $N$. For any $\beta$ this gives us the bound
    \begin{align*}
        (Z_{B_m,N}^{\beta}(U_r(s_X)))^{2^d} \leq \sum_{s: |s-s_X|<r} Z_{2m,N}^s,
    \end{align*}
    where $Z_{2m,N}^s$ is the partition function of $N$-dimer covers of homology class $s$ of the torus $\m T_{2m}$ with uniform edge weights. Recall that for $m,N$ fixed, there are of order $N(4mr)^d$ terms in the sum above. We can relate the left hand side to the version with free boundary conditions by
    \begin{align*}
         Z_{B_m,N}^{\text{free}}(U_r(s_X)) = \sum_\beta \m P(\beta) Z_{B_m,N}^{\beta}(U_r(s_X)) \leq \bigg(\sum_{s: |s-s_X|<r} Z_{2m,N}^s\bigg)^{1/2^d}
    \end{align*}
    and rearranging this gives
    \begin{align*}
        (Z_{B_m,N}^{\text{free}}(U_r(s_X)))^{2^d} \leq \sum_{s: |s-s_X|<r} Z_{2m,N}^s\leq N(4mr)^d\max_{s:|s-s_X|<r} Z_{2m,N}^{s}. 
    \end{align*}
    Therefore taking $\log$ and applying the definition of $\sigma$ (Definition~\ref{def:surface_tension}), we have 
    \begin{align*}
        \limsup_{m\to \infty}\limsup_{N\to \infty}\frac{1}{Nm^d}\log\frac{Z_{B_m,N}^{\text{free}}(U_r(s_X))}{(N!)^{m^d}} &\leq  \lim_{m\to \infty}\lim_{N\to \infty}\frac{1}{N(4m)^d} \log\frac{N(4mr)^d\max_{s: |s-s_X|<r} Z_{2m,N}^s}{(N!)^{(2m)^d}} \\
        &\leq -\max_{s:|s-s_X|<r} \sigma(s)\\
        &\leq -\sigma(s_X) +o_r(1),
    \end{align*}
    where in the last inequality we use that $\sigma$ is continuous. We can now add together the contributions from all the boxes $B_m$ used to cover $X_n$. Since $m=n^{1/2}$, there are $|X|n^{d/2}=|X|m^d$ boxes used to cover $X$, and only $O(n^{(d-1)/2})=O(m^{d-1})$ of them intersect $\partial X$.  So for $n,N$ large enough,
    \begin{align*}
        \frac{1}{Nn^d}\log \frac{Z_{X_n,N}^{\text{free}}(U_r(s_X))}{(N!)^{n^d|X|}} \leq \frac{1}{|X|m^{d}}\bigg( \frac{1}{Nm^d}\log\frac{Z_{B_m,N}^{\text{free}}(U_r(s_X))}{(N!)^{m^d}}\bigg) \leq -|X|(1+O(n^{-1/2}))(\sigma(s_X) +o_r(1)).
    \end{align*}
    Combined with the inequality in \eqref{eq:reduce_to_simplex}, we have 
    \begin{align*}
        \limsup_{n\to \infty} \limsup_{N\to \infty} \frac{1}{Nn^d}\log \frac{Z_{n,N}(U_r(\omega))}{(N!)^{n^d|R|}} \leq \Ent(\widetilde{\omega}) + o_r(1)= \Ent(\omega) + o_r(1) + o_\delta(1).
    \end{align*}
    Since $\delta$ goes to zero with $r$, this completes the proof. 
\end{proof}

\section{Euler-Lagrange equations and limit shape regularity}\label{sec:euler_lagrange}

In this section we study the regularity of limit shapes for the multinomial dimer model. We saw that limit shapes 
exist and are unique in Corollary~\ref{cor:limit_shape}. To avoid degenerate cases, we impose a mild regularity condition on $(R,b)$.

\begin{definition}[Flexibility]\label{def:flexible}
    A region and boundary condition $(R,b)$ is \textit{flexible} (for $\Lambda$) if for every point $x\in R^\circ$, there exists a neighborhood $U\ni x$ in $R$ and an asymptotic flow $\omega\in \text{AF}(\Lambda,R,b)$ such that $\omega\!\mid_U\subset \mc N(\Lambda)^\circ$  almost everywhere.  
\end{definition}
The flexible condition rules out the possibility that there are subsets $A\subset R$ of positive measure where {every} extension $\omega$ of $b$ is valued in $\partial \mc N(\Lambda)$ on $A$. Intuitively, it can be viewed as the continuum version of the feasible condition. One would expect a condition of this form to be needed for a large deviation principle if we used fixed boundary conditions on the discrete side, instead of allowing them to lie in a shrinking interval, as is the case e.g.\ for the standard 3D dimer model \cite{3Ddimers}. It is straightforward to see by an averaging argument that the flexible condition is equivalent to the following. 
\begin{lemma}\label{lem:equivalent_flexible}
    The pair $(R,b)$ is flexible if there exists $\omega\in \text{AF}(\Lambda,R,b)$ such that $\omega\!\mid_{R^\circ}\subset \mc N(\Lambda)^\circ$ almost everywhere. 
\end{lemma}

\subsection{Euler-Lagrange equations for a divergence free flow} 

From the large deviation principle, we saw that the limit shape divergence-free flow exists and is the unique minimizer of 
\begin{align*}
    \int_R \sigma(\omega(x))\,\dd x
\end{align*}
over $\omega\in \text{AF}(\Lambda,R,b)$. This can be rephrased as saying that $\omega$ is the unique solution to a system of Euler-Lagrange equations. 

\begin{thm}[Euler-Lagrange equations for a divergence free flow]\label{thm:EL_equations}
    The limit shape flow $\omega$ for $(R,b)$ is the solution to the following system of differential equations:
    \begin{itemize}
        \item the divergence equation, 
        \begin{align}
            \text{div}(\omega(x)) = \sum_{i=1}^d \pd{\omega_i}{x_i}(x) = 0;
        \end{align}
        \item a collection of ${d\choose 2}$ equations:
        \begin{align}\label{ELflow}
            \pd{}{x_i} \pd{}{s_j}\sigma(\omega(x)) - \pd{}{x_j}\pd{}{s_i}\sigma(\omega(x)) =0 
        \end{align}
        for all $i\neq j$ between $1$ and $d$.
    \end{itemize}
\end{thm}
\begin{proof}
  Since $\omega$ is an asymptotic flow it is divergence-free and therefore must solve the first equation. To derive the others, we use that a divergence-free vector field is dual to a closed $(d-1)$ form. Explicitly $\omega$ is dual to the differential form
  \begin{align*}
      A = \sum_{k=1}^d (-1)^{k-1}\omega_k \,\widehat{\dd x_k},
  \end{align*}
  where $\widehat{\dd x_k} = \dd x_1 \wedge\dots\wedge \dd x_{k-1}\wedge \dd x_{k+1}\wedge\dots\wedge\dd x_d$. Let $\Upsilon^m$ denote the set of $m$ forms on $\m R^d$ and denote the differential by $\dd : \Upsilon^{d-1}\to \Upsilon^d$. Since $\omega$ is divergence-free, $\dd A = 0$. Thus there exists $B\in \Upsilon^{d-2}$ such that $\dd B = A$. We can write $B$ using ${d\choose 2}$ functions $B_{i,j}=B_{j,i}$, $i\neq j$: 
  \begin{align*}
      B = \sum_{i<j} B_{i,j} \widehat{\dd x_{i,j}}.
  \end{align*}
  where $\widehat{\dd x_{i,j}}$ means both $\dd x_i$ and $\dd x_j$ are omitted. Thus 
  \begin{align*}
      A = \sum_{k=1}^d \bigg[ \sum_{i<k} \pd{B_{i,k}}{x_i} (-1)^{i-1} + \sum_{j>k} \pd{B_{k,j}}{x_j} (-1)^j\bigg] \widehat{\dd x_k} 
  \end{align*}
  Hence
  \begin{align*}
      \omega_k = \sum_{i<k} (-1)^{i-k} \pd{B_{i,k}}{x_i} + \sum_{j>k} (-1)^{j-k+1}\pd{B_{k,j}}{x_j}
        \end{align*}
We now use calculus of variations to derive the ${d\choose 2}$ equations. Pick $k<\ell$ and consider the perturbation $B_{k,\ell}\mapsto B_{k,\ell}+\delta \eta(x)$, where $\eta$ is a a smooth function with $\eta\mid_{\partial R} = 0$. 
This perturbation appears only in $\omega_k$ and $\omega_\ell$. If $\omega$ is a minimum, then
  \begin{align*}
    0 &= \frac{d}{d\delta}\bigg\lvert_{\delta = 0} \int_R \sigma\bigg( \omega_1, \dots, \omega_k+ (-1)^{\ell-k+1} \delta \pd{\eta}{x_\ell},\dots, \omega_\ell + (-1)^{k-\ell} \delta \pd{\eta}{x_k}, \dots, \omega_d \bigg) \, \dd x \\
    &= \pm \int_R \pd{\sigma}{s_k} \pd{\eta}{x_\ell} - \pd{\sigma}{s_\ell} \pd{\eta}{x_k} \, \dd x\\
    &= \mp\int_R \bigg(\pd{}{x_\ell}\pd{\sigma}{s_k} - \pd{}{x_k} \pd{\sigma}{s_\ell} \bigg)\eta \, \dd x .
\end{align*}
In the last equality we use integration by parts, and note that the boundary term vanishes because $\eta$ is $0$ on $\partial R$. By the fundamental lemma of calculus of variations, since $\eta$ is arbitrary, for all pairs $k<\ell$ we have:
\begin{align*}
    \pd{}{x_\ell}\pd{\sigma}{s_k}(\omega(x)) - \pd{}{x_k} \pd{\sigma}{s_\ell}(\omega(x))  = 0. 
\end{align*}
\end{proof}

\begin{rem}
    When $d=3$ we can instead phrase this in terms of vector calculus. In that case, $\omega$ divergence free implies it has a vector potential $V$ (dual to the form $B$) such that $\text{curl}(V) = \omega$. The ${3\choose 2} = 3$ equations each come from perturbing each component of $V=(V_1,V_2,V_3)$.
\end{rem}
In dimension $d=2$, the equations above can be rewritten as a gradient variational problem using the height function $h$. The limit shape height function is related to the limit shape flow $\omega=(\omega_1,\omega_2)$ by $\ast \omega = (-\omega_2,\omega_1) = \nabla h$. 
\begin{cor}[Euler-Lagrange equations for the height function in 2D]
    Suppose that $h$ is the limit shape height function for $(R,b)$. Then $h$ is the unique solution to the PDE
    \begin{align*}
     \text{div} (\nabla \sigma (\nabla h)) =\pd{}{x_1} \pd{\sigma }{s_1}(\nabla h) + \pd{}{x_2} \pd{\sigma}{s_2}(\nabla h)= 0.
    \end{align*}
\end{cor}

\subsection{Multinomial limit shapes do not have facets}\label{sec:no_facets}

One of the well-known properties of limit shapes for the standard dimer model (e.g., on the Aztec diamond) is the existence of \textit{facets} in the limit shape $\omega$, i.e., regions where $\omega$ takes values in $\partial \mc N(\Lambda)$. Here we show that the limit shapes for the $N\to \infty$ multinomial dimer model do not have facets (as long as $(R,b)$ is flexible). 

Recall that the \textit{amoeba} $\mc A(\Lambda)$ is the closure of the set where the multinomial free energy $F(\alpha)$ is convex, and that since $F$ depends analytically on $\alpha$, $\mc A(\Lambda) = \m R^d$ (Corollary~\ref{thm:free_energy_convex}). The fact that $\mc A(\Lambda) = \m R^d$ implies certain blow-up behavior of $\nabla\sigma(s)$ as $s\to s_0\in \partial \mc N(\Lambda)$.

\begin{lemma}[surface tension blow-up]\label{lem:blowup}
Fix $s_0\in \partial \mc N(\Lambda)$ and choose any direction $v$ such that for $\varepsilon>0$ small enough, $s_0+\varepsilon v\in \mc N(\Lambda)^\circ$. Then 
\begin{align}
    \lim_{\varepsilon\to 0^+} \frac{\sigma(s_0+\varepsilon v) - \sigma(s_0)}{\varepsilon}= -\infty.
\end{align}
\end{lemma}
\begin{proof} Note that $\sigma$ is continuous and convex, achieving its minimum value at $0$, and is negative except at the vertices of $\partial \mc N(\Lambda)$ where it is zero. Hence $\sigma(s_0+\varepsilon v) - \sigma(s_0)<0$ for all $\varepsilon>0$. The gradient $\nabla \sigma$ is a one-to-one correspondence from $\mc N(\Lambda)^\circ$ to $\mc A(\Lambda)^\circ$. Since $\mc A(\Lambda) = \m R^d$ while $\mc N(\Lambda)$ is compact, the result follows. 
\end{proof}
\begin{rem}
    For the standard dimer model (in two or three dimensions, with all edge weights equal to one), the surface tension does not blow-up if $s_0$ is a vertex of $\partial \mc N(\Lambda)$, but it does blow-up if $s_0$ is along an edge of $\partial \mc N(\Lambda)$. Correspondingly, if $(R,b)$ is flexible, limit shapes for the standard dimer model can have regions where they are valued at one of the vertices of $\mc N(\Lambda)$, but not along an edge.
\end{rem}

As a consequence of the blow-up property for all points on $\partial \mc N(\Lambda)$, we now show that the $N\to \infty$ multinomial limit shapes have no facets.

\begin{thm}[No facets]\label{thm:no_facets}
   For any $\Lambda$, if $(R,b)$ is flexible, then its limit shape is valued in the interior of $\mc N(\Lambda)$ almost everywhere on the interior of $R$.
\end{thm}
\begin{proof}
    Let $\omega$ denote the limit shape for $\Lambda, (R,b)$, which is unique by Corollary~\ref{cor:limit_shape}. Suppose for contradiction that there is a positive measure set $A\subset R$ such that $\omega\mid_A \subset \partial \mc N(\Lambda)$ almost everywhere {(throughout this proof, we use this restriction notation to mean almost everywhere)}. Since $(R,b)$ is flexible, there exists a flow $\omega_0\in \text{AF}(\Lambda,R,b)$ such that $\omega_0\mid_{R^\circ}\subset \mc N(\Lambda)^\circ$. We then define:
    \begin{align*}
        \omega_\varepsilon = (1-\varepsilon) \omega + \varepsilon \omega_0. 
    \end{align*}
    For any $\varepsilon>0$, $\omega_\varepsilon\in\text{AF}(\Lambda,R,b)$ and is valued in $\mc N(\Lambda)^\circ$ on $R^\circ$. We split $I(\omega_\varepsilon)$ into an integral over $R\setminus A$ and an integral over $A$. 
    
    Let $M = -\min_{s\in \mc N(\Lambda)} \sigma(s)>0$, this is finite since $\sigma$ is continuous and $\mc N(\Lambda)$ is compact. Since $\sigma$ is convex (Theorem~\ref{thm:strict_convex}), we have
    \begin{align*}
        \int_{R\setminus A} \sigma((1-\varepsilon) \omega(x) + \varepsilon \omega_0(x)) - \sigma(\omega(x)) \, \dd x &\leq \varepsilon \int_{R\setminus A} \sigma(\omega_0(x)) - \sigma(\omega(x)) \, \dd x \leq \varepsilon M |R\setminus A|.
    \end{align*}
    On the other hand, by Lemma~\ref{lem:blowup}, there is a function $K(\varepsilon)>0$ such that $\varepsilon K(\varepsilon)\to\infty$ as $\varepsilon\to 0$, and for $\varepsilon$ small enough given $s_0,v$ as in the statement, 
    \begin{align*}
        \sigma(s_0+\varepsilon v) - \sigma(s_0) \leq -\varepsilon K(\varepsilon). 
    \end{align*}
    Hence there is a constant $c>0$ such that
    \begin{align*}
        \int_A \sigma((1-\varepsilon) \omega(x) + \varepsilon \omega_0(x)) - \sigma(\omega(x))\, \dd x \leq -c \varepsilon K(\varepsilon). 
    \end{align*}
    Therefore
    \begin{align*}
        I(\omega_\varepsilon) - I(\omega) \leq  \varepsilon M |R\setminus A| - \varepsilon c K(\varepsilon).
    \end{align*}
    For $\varepsilon$ small enough this would be negative, which contradicts the claim that $\omega$ minimizes $I(\cdot)$. Therefore $A$ has measure zero and hence the limit shape $\omega$ has no facets.
\end{proof}

Theorem~\ref{thm:no_facets} allows us to leverage classical elliptic regularity results to show that $N\to \infty$ multinomial limit shapes in two dimensions are smooth. See e.g.\ \cite[Chapter 3]{XavierEllipticPDE} for an overview of regularity theorems for solutions to gradient variational problems. Extending this to higher dimensions would require other ad hoc estimates; see Question~\ref{q:always_smooth}.

\begin{cor}[Smoothness of limit shapes in 2D]\label{cor:2Dsmooth}
    If $\Lambda$ is a two dimensional lattice and $(R,b)$ is flexible, its limit shape height function $h$ is a smooth function, and as a consequence the limit shape divergence-free flow $\omega$ is a smooth vector field. 
\end{cor}
\begin{proof}

    By Theorem~\ref{thm:no_facets}, for any $x\in R^\circ$ there exists a neighborhood $U\ni x$ with $\overline{U}\subset R^\circ$, and where the limit shape height function $h$ has $\nabla h$ taking values contained in a closed set $C\subset \mc N(\Lambda)^\circ$.     By Corollary~\ref{thm:strict_convex}, $\sigma$ is strictly convex on the closed Newton polytope $\mc N(\Lambda)$, so on $C$ it is smooth and uniformly convex. Hence by elliptic regularity, $h$ is smooth on $U$. Since we can apply this for any interior point $x\in R^\circ$, it follows that $h$ and hence $\omega = \ast \nabla h$ are smooth on $R^\circ$. 
\end{proof}

\subsection{Dual Euler-Lagrange equations for the gauge}

There is also a system of Euler-Lagrange equations for minimizers of the free energy $F$. Since $\sigma$ and $F$ are Legendre duals, the gradients $\nabla \sigma: \mc N\to \mc A$ and $\nabla F:\mc A\to \mc N$ are inverses of each other and any flow $\omega\in \text{AF}(\Lambda,R,b)$ has a dual \textit{gauge flow} $\alpha:R\to \mc A$ related to it by: 
\begin{align}\label{eq:omega_to_alpha}
    \omega = \nabla F(\alpha) \iff \alpha = \nabla \sigma(\omega).
\end{align}
Using \eqref{eq:omega_to_alpha} one can see that the Euler-Lagrange equations for limit shape flow $\omega$ in Theorem~\ref{thm:EL_equations} are equivalent to the following Euler-Lagrange equations for $\alpha$. 

\begin{thm}[Euler-Lagrange equations for the gauge flow]\label{thm:EL_equations_gauge}
    Fix $R\subset \m R^d$ and boundary value for the gauge flow $b'$. The limiting gauge flow $\alpha:R\to \mc A=\m R^d$ is the unique solution to the following system of differential equations with boundary value $b'$: 
    \begin{itemize}
        \item the equation \begin{align}
            \text{div} (\nabla F(\alpha)) = \sum_{i=1}^d \pd{}{x_i} \pd{}{a_i} F(\alpha) = 0;
        \end{align}
        \item a collection of ${d\choose 2}$ equations of the form 
        \begin{align}\label{eq:alpha_curl}
            \pd{ \alpha_j}{x_i} - \pd{\alpha_i }{x_j} = 0,
        \end{align}
        for all $i\neq j$ between $1$ and $d$.
    \end{itemize}
\end{thm}
\begin{rem}
    There are some technicalities in relating the boundary conditions $b'$ for the gauge to boundary conditions $b$ of the divergence free flow. If $b$ has an extension $\omega$ such that $\omega\mid_{\partial R}$ is a.e.\ valued in $\mc N(\Lambda)^\circ$, then we can define $b' = T(\nabla \sigma(\omega))$ (recall that $T$ is the boundary value operator). However, if $\omega$ takes values in $\mc N(\Lambda)$ on $\partial R$ (which is the case e.g.\ for the Aztec diamond), then $\nabla \sigma$ will not be defined. We do not address this rigorously, but note that, for example, one could approximate the solution by replacing $b$ with $(1-\delta) b$ and taking $\delta$ to zero.
\end{rem}
The equations \eqref{eq:alpha_curl} imply that we can rephrase the Euler-Lagrange equations for the gauge as a gradient variational problem. 
\begin{cor}\label{cor:EL_gauge_cor}
    Fix $R\subset \m R^d$ and a suitable boundary function $H_b$ on $\partial R$. If $H$ is a weak solution to the differential equation 
    \begin{align}\label{eq:gauge_PDE}
         \text{div} (\nabla F(\nabla H)) = 0
    \end{align}
    with $H\mid_{\partial R} = H_b$, the limiting gauge flow is $\alpha = \nabla H$.
\end{cor}
\begin{rem}
    Note that we do not derive an explicit description for the suitable boundary conditions $H_b$ of $H$ (where ``suitable'' means e.g.\ that there exists a unique solution to the PDE). We note however that $H_b$ can be well-defined and finite even when $b'$ (the boundary condition for $\alpha = \nabla H$) is infinite, as is the case e.g.\ for the Aztec diamond (see Section~\ref{sec:limit_shapes}). 
\end{rem}
This \textit{gauge function limit shape} $H$ exists (at least weakly) in any dimension. Unlike the height function, which is Lipschitz and therefore continuous and differentiable almost everywhere, in general dimension, $H$ does not a priori have any additional regularity.

\section{Scaling limit of critical gauge functions}\label{sec:critical_gauge_scaling}

Here we show that the limit shape gauge function $H$ arises as the scaling limit of critical gauge functions on graphs $R_n\subset \frac{1}{n} \Lambda$ approximating $R$. Let $g_n,f_n$ denote the critical gauge function on the black, white vertices of $R_n$ respectively. Here we assume that all limiting vertex multiplicities are $1$, even on $\partial R_n$. Then these satisfy the \textit{critical gauge equations}: 
\begin{align*}
    \sum_{u : u\sim v} g_n(u) f_n(v) &= 1 \\
     \sum_{v : u\sim v} g_n(u) f_n(v) &= 1 
\end{align*}
for all $u\in B\cap R_n$ and $v\in W\cap R_n$. By solving for $f_n(v)$ in the first equation, we can combine these as a single equation in terms of only $g_n$. Recall that $e_1,\dots, e_D$ are the edge vectors incident to a white vertex in $\Lambda$, and correspondingly the edge vectors incident to a black vertex are $-e_1,\dots,-e_D$. 
The equation is:
\begin{align*}
    \frac{1}{g_n(u)} = \sum_{i=1}^D \bigg(\sum_{j=1}^D g_n(u - \frac1ne_i + \frac1n e_j)\bigg)^{-1}      \qquad \forall \, u\in B\cap R_n.
\end{align*}

\begin{thm}\label{thm:critical_gauge_scaling_limit}
    Fix a $d$-dimensional lattice $\Lambda$ and $R\subset \m R^d$ as before. Suppose that $H$ is $C^1$ and solves the gauge PDE: $\text{div}(\nabla F(\nabla H))=0$, and that $g_n$ is the critical gauge function on black vertices of $R_n$ with limiting multiplicity $1$ everywhere on $R_n$. Further suppose that
    \begin{align*}
        \sup_{u\in B\cap\partial R_n} |(1/n) \log g_n(u) - H(u)| = O(1/n).
    \end{align*}
    We equate $(1/n) \log g_n$ defined on $B\cap R_n$ with its linear interpolation to a continuous function. Then
    \begin{align*}
        \lim_{n\to \infty} (1/n) \log g_n 
    \end{align*}
    exists in the topology induced by the supremum norm on continuous functions and is equal to $H$.
\end{thm}

\begin{rem}
    The limit $\lim_{n\to \infty}(1/n) \log f_n$ also exists, but differs by a sign (since all the edge directions per vertex have the opposite sign, etc). Analogous arguments show that $\lim_{n\to \infty}(1/n) \log f_n= -{H}$.
\end{rem}
\begin{rem}
    If $H$ is $C^1$, then standard perturbative estimates can be used to show that $H$ is real analytic. The technical issue here is that while $F$ is strictly convex, is not uniformly convex on all of $\m R^d$, and thus e.g.\ the classical elliptic regularity results of \cite{DeGiorgi,Nash1,Nash2} do not apply. In 2D, by Corollary~\ref{cor:2Dsmooth}, the limit shape height function $h$ is smooth, hence the flow $\omega$ and gauge flow $\alpha$ are smooth, and hence any weak solution $H$ to \eqref{eq:gauge_PDE} must be $C^1$. This argument is specific to two dimensions, since it uses that the dual equation to \eqref{eq:gauge_PDE} in terms of the height function is also a gradient variational problem. In higher dimensions it is not clear whether one can lift the $C^1$ assumption on $H$; see Question~\ref{q:always_smooth}.
\end{rem}
We use in the proof that the solution $H$ is real analytic. To simplify the next statement, we introduce the following ``operator'' notation  
\begin{align*}
    \mc D(q) := \text{div} \nabla F(\nabla q)
\end{align*}
and 
\begin{align*}
    \mc C_n(q(u)) :=   {q(u)} \sum_{i=1}^D \bigg(\sum_{j=1}^D q(u - (1/n) e_i + (1/n) e_j)\bigg)^{-1}
\end{align*}
for $u\in B\cap R_n$.

\begin{definition}
	We say that $q$ is a {subsolution} (resp. {supersolution}) to the equation $\mc D(\cdot)=0$ if $\mc D(q)<0$ (resp. $\mc D(q)>0$). Similarly we say that $q$ is a subsolution (resp. supersolution) to $\mc C_n(\cdot) = 1$ if $\mc C_n(q) < 1$ (resp. $\mc C_n(q)> 1$). 
\end{definition}

The proof of Theorem~\ref{thm:critical_gauge_scaling_limit} has three main steps: (1) we show that sub/supersolutions to critical gauge equations satisfy a version of the maximum principle (Proposition~\ref{prop:maximum_principle}), (2) we relate $\mc C_n$ to $\mc D$ by expanding the critical gauge equations as $1/n\to 0$ (Proposition~\ref{prop:expanding_critical_equations}), (3) we partition $R$ into small balls and construct perturbations of $H$ which are sub/supersolutions to the critical gauge equations on each ball, and iteratively apply the maximum principle to show that $(1/n) \log g_n$ being close to $H$ on $\partial R$ implies it is close on all of $R$. 

\begin{prop}[Maximum principle for the critical gauge equations]\label{prop:maximum_principle}
	Suppose that $q, q_+, q_-$ are a solution, supersolution, and subsolution respectively to the critical gauge equation $\mc C_n(\cdot)= 1$ on a region $U\subset B\cap \varepsilon\Lambda$. Then 
	\begin{enumerate}[(i)]
		\item the minimum value of $q/q_-$ over $U$ is realized on $\partial U$. 
		\item the maximum value of $q/q_+$ over $U$ is realized on $\partial U$;
	\end{enumerate}
\end{prop}
\begin{proof}
	We show (i), as the proof of (ii) is analogous. Since $q,q_-$ are a solution and subsolution respectively we have that for $u\in U$, 
	\begin{align*}
		\frac{1}{q(u)} &= \sum_{i=1}^D \bigg(\sum_{j=1}^D q(u - (1/n) e_i + (1/n) e_j)\bigg)^{-1}\\
		\frac{1}{q_-(u)} &> \sum_{i=1}^D \bigg(\sum_{j=1}^D q_-(u - (1/n) e_i + (1/n) e_j)\bigg)^{-1}.
	\end{align*}
Hence, dividing the second by the first,
\begin{align*}
	\frac{q(u)}{q_-(u)} >\frac{\sum_{i=1}^D \bigg(\sum_{j=1}^D q_-(u - (1/n) e_i + (1/n) e_j)\bigg)^{-1}}{ \sum_{i=1}^D \bigg(\sum_{j=1}^D q(u - (1/n) e_i + (1/n) e_j)\bigg)^{-1}}.
\end{align*}
On the other hand note that if $a,b,c,d,e,f>0$ then 
\begin{align*}
	\frac{a}{b} = \frac{c+d}{e+f} \quad \implies \quad \max\bigg\{ \frac{c}{e}, \frac{d}{f} \bigg\}\geq \frac{a}{b} \geq \min\bigg\{ \frac{c}{e}, \frac{d}{f} \bigg\}.
\end{align*}
We apply this twice to show that 
\begin{align*}
	\frac{q(u)}{q_-(u)} > \min_{i,j=1, \dots, D} \frac{q(u+(1/n) e_i - (1/n) e_j)}{q_-(u+(1/n) e_i - (1/n) e_j)},
\end{align*}
in other words, $q(u)/q_-(u)$ is bounded below by its value at at least one of its neighbors. Iterating this until we reach $\partial U$ proves (i). 
\end{proof}

We now relate $\mc C_n$ to $\mc D$ by expanding the critical gauge equations as $1/n\to 0$.
\begin{prop}\label{prop:expanding_critical_equations}
    Suppose $q$ is a $C^2$ function on $R$. For $u\in B\cap R_n$, we have the expansion
    \begin{align}
        \mc C_n(\exp(n  q(u)) = 1 - \mc D(q(u)) (1/n) + O(1/n^2)
    \end{align}
    as $1/n\to 0$.
\end{prop}
\begin{proof} 
We expand in $1/n$:
    \begin{align*}
        \mc C_n(\exp(n q(u)) &= \sum_{i=1}^D \exp(n q(u)) \bigg( \sum_{j=1}^D \exp(n q(u -(1/n) e_i + (1/n) e_j)\bigg)^{-1}\\
        &= \sum_{i=1}^D \bigg( \sum_{j=1}^D \exp(n (q(u -(1/n) e_i + (1/n) e_j)-q(u))\bigg)^{-1} \\
        &=\sum_{i=1}^D \bigg( \sum_{j=1}^D \exp((e_j-e_i)\cdot \nabla q(u) + \frac{1}{2n}(e_j-e_i)^t D^2 q(u) (e_j-e_i) + O(1/n^2) )\bigg)^{-1} 
    \end{align*}
    where $D^2 q$ is the Hessian matrix of $q$. Let $v_{ij} = e_j-e_i$ and $Q_{ij}(u)= (e_j-e_i)^t D^2 q(u) (e_j-e_i)$.
    \begin{align*}
        \mc C_n(\exp(n q(u)) &= \sum_{i=1}^D \bigg( \sum_{j=1}^D \exp(v_{ij}\cdot \nabla q(u) + \frac{1}{2n}Q_{ij}(u) + O(1/n^2) )\bigg)^{-1} \\
        &=\sum_{i=1}^D \bigg( \sum_{j=1}^D \exp(v_{ij}\cdot \nabla q(u)) (1  + \frac{1}{2n}Q_{ij}(u) + O(1/n^2) )\bigg)^{-1}\\
        &= \sum_{i=1}^D\bigg[ \frac{1}{\sum_{k=1}^D \exp(v_{ik}\cdot \nabla q(u))} - \frac{1}{2n} \frac{\sum_{j=1}^D \exp(v_{ij}\cdot \nabla q(u)) Q_{ij}}{\bigg(\sum_{k=1}^D  \exp(v_{ik}\cdot \nabla q(u))\bigg)^2} + O(1/n^2)\bigg]\\
        &= 1 -\frac{1}{2n} \sum_{i=1}^D \frac{\sum_{j=1}^D \exp(v_{ij}\cdot \nabla q(u)) Q_{ij}}{\bigg(\sum_{k=1}^D  \exp(v_{ik}\cdot \nabla q(u))\bigg)^2} + O(1/n^2)\\
        &= 1 -\frac{1}{2n} \frac{\sum_{i=1}^D \sum_{j=1}^D\exp((e_i+e_j)\cdot \nabla q(u)) Q_{ij}}{\bigg(\sum_{k=1}^D  \exp(e_k\cdot \nabla q(u))\bigg)^2} + O(1/n^2).
    \end{align*}
    To see that the order $1/n$ term above is $-\mc D(q)$ recall from Theorem~\ref{thm:free_energy_torus} and the prescription for weights gauge equivalent to uniform in \eqref{eq:edge_weights_d} that $F(\nabla q)$ is given by: 
    \begin{align}
        F(\nabla q) = \log \bigg( \sum_{j=1}^D \exp(e_j \cdot \nabla q)\bigg).
    \end{align}
    These are the weights for the free energy whose Legendre dual is $\sigma(s)$, where $s$ is the average flow per {white} vertex. Let $\eta_1,\dots,\eta_d$ denote unit vectors parallel to the coordinates $(x_1,\dots,x_d)$ for $\m R^d$, which are the coordinates we take the divergence in $\mc D(q)$ with respect to. Define $a_j^\ell$ to be the coefficients $e_j = \sum_{\ell=1}^d a_j^\ell \eta_\ell$. Then
    \begin{align*}
        \nabla F(\nabla q)_\ell = \frac{\sum_{j=1}^D a_j^\ell \exp(e_j \cdot \nabla q)}{\sum_{k=1}^D \exp(e_k \cdot \nabla q)}.
    \end{align*}
    We verify the final equality by taking the divergence. 
    \begin{align*}
       \text{div} \nabla F(\nabla q) &= \sum_{\ell=1}^d \pd{}{x_\ell} \frac{\sum_{j=1}^D a_j^\ell \exp(e_j \cdot \nabla q)}{\sum_{k=1}^D \exp(e_k\cdot \nabla q)} \\
       &= \sum_{\ell=1}^d \frac{1}{(\sum_{k=1}^D \exp(e_k\cdot \nabla q))^2}\\ 
       &\quad\times\bigg[ \sum_{j=1}^D a_j^\ell \sum_{n=1}^d a_j^n \frac{\partial^2 q}{\partial x_\ell \partial x_n} \exp(e_j\cdot \nabla q) \sum_{i=1}^d \exp(e_i\cdot \nabla q) \\
       &\quad\qquad - \sum_{j=1}^D a_j^\ell \exp(e_j\cdot \nabla q) \sum_{i=1}^D \sum_{m=1}^d a_i^m \frac{\partial^2 q}{\partial x_\ell \partial x_m} \exp(e_i\cdot \nabla q)\bigg]
    \end{align*}
    which we further simply to find
    \begin{align*}
       &= \sum_{i=1}^D \sum_{j=1}^D\frac{ \exp(e_j\cdot \nabla q) \exp(e_i \cdot \nabla q)}{(\sum_{k=1}^D \exp(e_k\cdot \nabla q))^2} \bigg[\frac{\partial^2 q}{\partial e_j^2} - \frac{\partial^2 q}{\partial e_j \partial e_i}\bigg]\\
       &= \frac{1}{2} \sum_{i=1}^D \sum_{j=1}^D\frac{ \exp(e_j\cdot \nabla q) \exp(e_i \cdot \nabla q)}{(\sum_{k=1}^D \exp(e_k\cdot \nabla q))^2} Q_{ij}.
    \end{align*}
\end{proof}

\begin{proof}[Proof of Theorem~\ref{thm:critical_gauge_scaling_limit}]
  Fix $\delta,\delta_2>0$ to be specified later. Let $U\subset R$ be an open ball of radius $\delta_2$ centered at $x_0$ and define $q_U(x) = \frac{1}{2}||x-x_0||^2$ {for $x\in U$}, and consider the perturbations 
  \begin{align*}
    H\pm \delta q_U.
  \end{align*}
  Note that $\pd{q_U}{x_i}(x) \leq \delta_2$ and that $\frac{\partial^2 q_U}{\partial x_i \partial x_j}$ is $1$ if $i=j$ and $0$ otherwise. Hence expanding $H\pm \delta q_U$ as $\delta\to 0$ and using that $\mc D(H)=0$, we find for $x\in U$: 
   \begin{align}\label{eq:delta_expansion}
      \mc D(H+\delta q_U) &=\delta \bigg[ {\sum_{i,j=1}^{ d}} \frac{\partial^2}{\partial \alpha_i \partial \alpha_j} F(\nabla H) +  {\sum_{i,j,k=1}^{ d}} \frac{\partial^3}{\partial \alpha_i \partial \alpha_j \partial \alpha_k} F(\nabla H) \frac{\partial^2 H}{\partial x_i\partial x_j} \pd{q_U}{x_k}\bigg] + O(\delta^2)\\
      &{ \geq} \delta \bigg[ {\sum_{i,j=1}^{ d}} \frac{\partial^2}{\partial \alpha_i \partial \alpha_j} F(\nabla H) - \delta_2 {\sum_{i,j,k=1}^{ d}} \bigg\lvert\frac{\partial^3}{\partial \alpha_i \partial \alpha_j \partial \alpha_k} F(\nabla H) \frac{\partial^2 H}{\partial x_i\partial x_j}\bigg\rvert\bigg] + O(\delta^2).\nonumber
  \end{align} 
  Let $\mc D_2^+(H,q_U)$ denote the order $\delta$ term in \eqref{eq:delta_expansion}. Note that the second line of the above equation depends only on $F,H$, and $\delta_2$. Since $F$ is strictly convex on all of $\m R^d$ and $H$ is real analytic, there exists $\delta_2$ small enough given $H,F$ such that $\mc D_2^+(H,q_U)>c>0$ for all $U$. Therefore $\mc D(H+\delta q_U)>0$ on $U$, and hence $H+\delta q_U$ is a supersolution to $\mc D(\cdot)=0$ on $U$.
  
  We analogously have the bound 
  \begin{align*}
    \mc D(H-\delta q_U) \leq -\delta \bigg[ {\sum_{i,j=1}^{ d}} \frac{\partial^2}{\partial \alpha_i \partial \alpha_j} F(\nabla H) + \delta_2 {\sum_{i,j,k=1}^{ d}} \bigg\lvert\frac{\partial^3}{\partial \alpha_i \partial \alpha_j \partial \alpha_k} F(\nabla H) \frac{\partial^2 H}{\partial x_i\partial x_j}\bigg\rvert\bigg] + O(\delta^2).
  \end{align*}
  Let $\mc D_2^-(H,q_U)$ to be the order $\delta$ term in the expansion above. Again by taking $\delta_2$ small enough, we can ensure that $\mc D_2^-(H,q_U)<c<0$, and hence $\mc D(H-\delta q_U)<0$ on $U$. Therefore $H-\delta q_U$ is a subsolution to $\mc D(\cdot) = 0$ on $U$.

    We now apply Proposition~\ref{prop:expanding_critical_equations} to see that 
    \begin{align*}
        \mc C_n(\exp(n(H\pm\delta q_U))) &= 1 - (1/n) \mc D(H+\delta q_U) + O(1/n^2)\\
        &= 1 - (1/n)[\mc D[H]\pm \delta \mc D_2^\pm (H,q_U)+O(\delta^2)+O(1/n^2)\\
        &= 1 \mp (\delta/n) \mc D_2^\pm(H,q_U) + O(\delta^2/n)+O(1/n^2).
    \end{align*}
    Take $\delta = n^{-1/2}$, and let $U_n = \frac{1}{n}\Lambda\cap U$. Then we have that 
    \begin{itemize}
        \item $\mc C_n(\exp(n(H+n^{-1/2} q_U))) < 1$ on $B\cap U_n$, hence $\exp(n(H+n^{-1/2} q_U))$ is a subsolution to $\mc C_n(\cdot) =1$ on $U_n$;
        \item $\mc C_n(\exp(n(H-n^{-1/2} q_U))) > 1$ on $B\cap U_n$,  hence $\exp(n(H-n^{-1/2} q_U))$ is a supersolution to $\mc C_n(\cdot) =1$ on $U_n.$
    \end{itemize}

  Let $g_n$ denote the solution to the critical gauge equation on all of $R_n$. Now fix $x\in R_n$ and a $\delta_2$ neighborhood $U$ containing $x$. By the maximum principle (Proposition~\ref{prop:maximum_principle}), and since $0\leq q_U(x) \leq \delta_2^2/2$, 
  \begin{align*}
      (1/n) \log g_n(x) - H(x) &\geq \min_{y\in \partial U_n} (1/n) \log g_n(y) - H(y) -n^{-1/2}q_U(y)\\
      &\geq  (1/n) \log g_n(x_{\min}^U) - H(x_{\min}^U) - n^{-1/2}\delta_2^2/2.
  \end{align*}
  where $x_{\min}^U$ achieves the minimum value on $\partial U_n$. Similarly, 
  \begin{align*}
      (1/n) \log g_n(x) - H(x) &\leq \max_{y\in \partial U_n} (1/n) \log g_n(y) - H(y) +n^{-1/2}q_U(y)\\
      &\leq  (1/n) \log g_n(x_{\max}^U) - H(x_{\max}^U) + n^{-1/2}\delta_2^2/2.
  \end{align*}
We now iterate this, applying the same estimate with new neighborhoods $U_1,U_2$ containing $x_{\min}^U,x_{\max}^U$. We repeat this process until both sequences of neighborhoods hit $\partial R_n$. The number of neighborhoods passed through in each sequence is at most $O(\delta_2^{-d})$. Thus there exists $u,u'\in \partial R_n$ and constants $C_1,C_2>0$ such that: 
\begin{align*}
   (1/n) \log g_n(u) - H(u) - C_1n^{-1/2}\delta_2^{2-d} &\leq  (1/n) \log g_n(x) - H(x) \\
   &\leq (1/n) \log g_n(u') - H(u') + C_2n^{-1/2}\delta_2^{2-d}.
\end{align*}
Recall that $\delta_2$ is a constant determined by $F$ and $H$ which is independent of $n$. Therefore since $(1/n) \log g_n$ and $H$ differ by $O(1/n)$ on $\partial R_n$, 
\begin{align*}
    |(1/n) \log g_n(x) - H(x) | = O(n^{-1/2})
\end{align*}
for any $x\in R_n$, which completes the proof. 

\end{proof}

\section{Computing limit shapes}\label{sec:limit_shapes}

From the results in previous sections, we have a few approaches for computing limit shapes:
\begin{itemize}
    \item solve the Euler-Lagrange equations from Section~\ref{sec:euler_lagrange}, either for the flow or gauge;
    \item compute the critical gauge functions on a sequence of graphs and take the scaling limit, applying Theorem~\ref{thm:critical_gauge_scaling_limit};
    \item in two dimensions, solve the Euler-Lagrange equations for the limiting height function.
\end{itemize}

Here we compute examples where the critical gauge functions turn out have nice closed form expressions in $n$, and use this to explicitly compute all of the equivalent data of the limit shape by applying Theorem~\ref{thm:critical_gauge_scaling_limit}. We compute: the limiting gauge function, the limit shape divergence-free flow, the limit shape gauge flow, and (in two dimensions) the limit shape height function. 

These examples are in fact the only examples we know of families of growing graphs where the critical gauge has a closed form expression in terms of $n$. Finding other examples, and understanding in general the properties a graph must have for its critical weights to have a closed form expression, is an interesting open question; see Question~\ref{q:closed_form_critical_gauge}.

In Section~\ref{sec:tangent_plane}, for two-dimensional lattices where there is a height function, we use the tangent plane method of \cite{KenyonPrause} to describe the limit shape height function as an envelope of $\kappa$-harmonically moving planes, where $\kappa$ is determined by the surface tension.

\subsection{Aztec diamond}

The first example we consider is the Aztec diamond, where the critical gauge functions for fixed $n$ were computed in \cite{KenyonPohoata}. The Aztec diamond limit shape height function is displayed in the Introduction, see Figure~\ref{fig:aztecs}. We embed $\m Z^2$ in the plane rotated by $45^\circ$, i.e., the vertices of the lattice are:
\begin{itemize}
    \item A set $B$ which is a copy of the standard $\m Z^2$ lattice
    \item A set $W$ which is $B$ translated by $(\frac{1}{2},-\frac{1}{2})$. 
\end{itemize}
We then add edges connecting nearest neighbors; this constructs a rotated copy of $\m Z^2$ with vertices $B\cup W$. The free energy for $\m Z^2$ with this embedding is 
\begin{align*}
    F(\alpha_1,\alpha_2) = \log ((\mathrm{e}^{\alpha_1/2}+\mathrm{e}^{-\alpha_1/2})(\mathrm{e}^{\alpha_2/2}+\mathrm{e}^{-\alpha_2/2})). 
\end{align*}
The Newton polytope $\mc N = [-1/2,1/2]^2$ and the surface tension is 
\begin{align*}
    \sigma(s_1,s_2) &= \frac{1+2s_1}{2}\log\bigg(\frac{1+2s_1}{2}\bigg) + \frac{1-2s_1}{2}\log\bigg(\frac{1-2s_1}{2}\bigg) \\
    &+ \frac{1+2s_2}{2}\log\bigg(\frac{1+2s_2}{2}\bigg) + \frac{1-2s_2}{2}\log\bigg(\frac{1-2s_2}{2}\bigg).
\end{align*}
See Figure~\ref{fig:surface_tensions}. As such, the Euler-Lagrange equation for the limit shape divergence free flow $\omega=(\omega_1,\omega_2)$ (as given in Theorem~\ref{thm:EL_equations}) is 
\begin{align}\label{eq:z2_flow_euler}
    \frac{4}{1-4\omega_1^2} \pd{\omega_1}{y} - \frac{4}{1-4 \omega_2^2} \pd{\omega_2}{x}= 0.
\end{align}
Finally, the Euler-Lagrange equation for the limiting gauge function $H(x,y)$ (as in Theorem~\ref{cor:EL_gauge_cor}) is 
\begin{align}\label{eq:z2_gauge_euler}
    \frac{\mathrm{e}^{H_{x}}H_{xx}}{(1+\mathrm{e}^{H_{x}})^2} +  \frac{\mathrm{e}^{H_{y}}H_{yy}}{(1+\mathrm{e}^{H_{y}})^2} =0.
\end{align}
We will see below (by taking the scaling limit of the critical gauge) that the limit shape divergence free flow and limiting gauge function for the Aztec diamond satisfy these PDEs.

To describe the critical gauge on the Aztec diamond, we introduce some coordinates. Label vertices so that $b\in B$ is denoted by its standard coordinate $(i,j)$, and $w\in W$ is denoted by $(i,j)$ if its actual position is $(i',j')=(i+1/2, j-1/2)$. See Figure~\ref{fig:Z2_labeled_neighbors}.
\begin{figure}[H]
    \centering
    \includegraphics[scale=0.7]{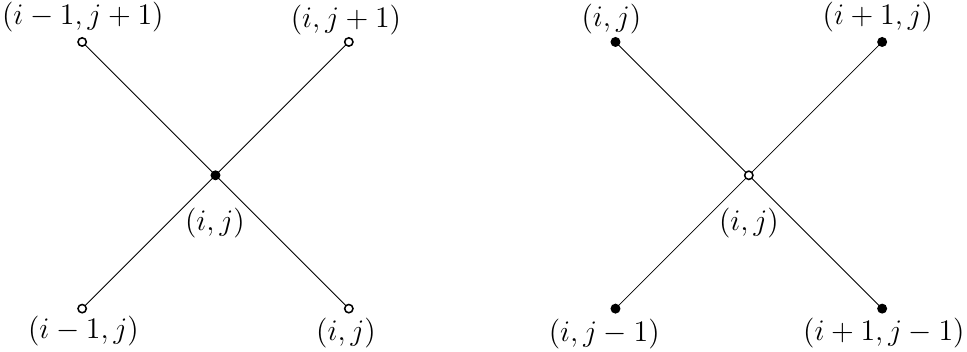}
    \caption{Neighbors of a black vertex (left) or white vertex (right) in these coordinates.}
    \label{fig:Z2_labeled_neighbors}
\end{figure}
The size $n$ Aztec diamond $\AD(n)$ has the vertices where $B_n\subset B$ is the $(n+1)\times n$ square $\{0,....,n\}\times \{0,....,n-1\}$ and $W_n\subset W$ is the $n\times (n+1)$ square $\{0,...,n-1\}\times \{0,...,n\}$. The critical gauge functions are
\begin{align}
    f_n(i,j) = \frac{n}{n+1} \frac{{n-1 \choose i}} {{n\choose j}}\qquad \text{and} \qquad g_n(i,j) = \frac{{n-1 \choose j}}{{n\choose i}}
\end{align}
at white and black vertices respectively. It is straightforward to check that these satisfy the critical gauge equations.

In this embedding of $\m Z^2$, the four edge directions are northeast, northwest, southeast, and southwest (NE, NW, SE, SW respectively) when viewed as a vector from white-to-black. The corresponding critical edge weights for the edge incident to the white vertex $(i,j)$ are:
\begin{align*}
   f_n(i,j) g_n(i+1,j) &= \frac{(i+1)(n-j)}{n(n+1)} \qquad \text{for NE edges}\\
    f_n(i,j)g_n(i,j) &= \frac{(n-i)(n-j)}{n(n+1)} \qquad \text{for NW edges}\\
    f_n(i,j) g_n(i+1,j-1) &= \frac{(i+1)j}{n(n+1)} \qquad \text{for SE edges}\\
    f_n(i,j) g_n(i,j-1) &= \frac{(n-i)j}{n(n+1)} \qquad \text{for SW edges}.
\end{align*}
To take the scaling limit, we multiply the points in $\AD(n)$ by ${1}/{n}$ so that they lie in the unit square $[0,1]^2$. If $(i_n,j_n)$ is a sequence of points such that $(i_n/n,j_n/n)\to (x,y)\in [0,1]$, then the edge weights at $(x,y)$ limit to
\begin{align*}
    (w_{\text{NE}}, w_{\text{NW}}, w_{\text{SE}}, w_{\text{SW}}) \to (x(1-y), (1-x)(1-y), xy, (1-x)y).
\end{align*}
Therefore the limiting divergence free flow $\omega=(\omega_1,\omega_2)$ is 
\begin{align}\label{eq:omega_aztec_diamond_direct}
    \omega &= 1/2\cdot (w_{\text{NE}}-w_{\text{NW}}+w_{\text{SE}}-w_{\text{SW}}, w_{\text{NE}}+w_{\text{NW}}-w_{\text{SE}}-w_{\text{SW}}) \\
    &= 1/2 \cdot (2x-1 , 1-2y). \nonumber
\end{align}
Since $\nabla h$ is the dual of $f$, $\nabla h = (-\omega_2,\omega_1) = 1/2\cdot (2y-1,2x-1)$. Hence the height function is $h(x,y) = \frac{1}{4} (2x-1) (2y-1)$ for $(x,y)\in [0,1]^2$. {Under a suitable the change of coordinates sending $[0,1]^2$ to $\{(u,v): |u|+|v|\leq 1\}$, $h(u,v) = u^2 - v^2$. See Figure~\ref{fig:aztecs}.}

We can also compute the scaling limit $H = \lim_{n\to \infty} \frac{1}{n} \log g_n$ and apply Theorem~\ref{thm:critical_gauge_scaling_limit}. We compute:
\begin{align*}
    H(x,y) = \lim_{n\to \infty} \frac{1}{n}\log g_n(nx, ny) = (1-x)\log(1-x) + x\log x -(1-y)\log(1-y) - y\log y,
\end{align*}
for $(x,y)\in[0,1]^2$. Note that this is $C^1$ on $(0,1)^2$. It is straightforward to check that this solves the gauge PDE in \eqref{eq:z2_gauge_euler}. From this, we find that the gauge flow is
\begin{align*}
    \alpha = \nabla H = (2\text{arctanh}(2x-1),2\text{arctanh}(1-2y)).
\end{align*}
Since $\nabla F(\alpha_1,\alpha_2) = (1/2\tanh(\alpha_1/2),1/2\tanh(\alpha_2/2))$ the limit shape divergence free flow is
\begin{align*}
    \omega(x,y) = \nabla F(\alpha(x,y)) = 1/2 \cdot (2x-1,1-2y),
\end{align*}
in agreement with \eqref{eq:omega_aztec_diamond_direct}. We also note that $\lim_{n\to \infty} \frac{1}{n} \log f_n = -H$.

\subsection{Weighted Aztec diamond}\label{weightedAD}
We can also consider the more general case of the Aztec diamond with nonconstant edge weights $a,b,c,d$ around each white vertex, so the free energy is 
\begin{align}\label{eq:free_energy_weighted}
    F(\alpha_1,\alpha_2) = \log(a \mathrm{e}^{\alpha_1/2+\alpha_2/2} + ba \mathrm{e}^{-\alpha_1/2+\alpha_2/2}+c \mathrm{e}^{-\alpha_1/2-\alpha_2/2}+d\mathrm{e}^{\alpha_1/2-\alpha_2/2})
\end{align}
While there is no simple expression for the critical gauge functions, the scaling limit $H$ should solve a gauge PDE $\text{div}(\nabla F(\nabla H))=0$, analogous to \eqref{eq:z2_gauge_euler}. From \eqref{eq:free_energy_weighted} we find that this PDE is
\begin{align}\label{eq:gauge_pde_weighted}
    (a\mathrm{e}^{H_y}+d \mathrm{e}^{-H_y})(b\mathrm{e}^{H_y}+c\mathrm{e}^{-H_y})H_{xx} + 2(ac-bd)H_{xy} + (a\mathrm{e}^{H_x}+b\mathrm{e}^{-H_x})(d\mathrm{e}^{H_x}+c\mathrm{e}^{-H_x})H_{yy}=0.
\end{align}
In particular this has a solution of the form $H(x,y) = f(x)+g(y)$ which happens to have the correct boundary values for the Aztec diamond. The limit shape height function $h(x,y)$ is related to $H$ solving \eqref{eq:gauge_pde_weighted} by $\nabla h(x,y) = \ast \nabla F(\nabla H)) = (-\omega_2,\omega_1)$, where $\omega= (\omega_1,\omega_2) = \nabla F(\nabla H)$. 

From this, for $\lambda := ac/bd >1$ we find the explicit formula for the height function on the unit square $(x,y)\in [\frac{\log(b/c)}{\log\lambda},\frac{\log(a/d)}{\log\lambda}]\times[\frac{\log(b/a)}{\log\lambda},\frac{\log(c/d)}{\log\lambda}]$ is (up to an additive constant)
\be\label{wAD}
 h(x,y)=-2\frac{\log\left(a\lambda^{-\frac{x}2+\frac{y}2}-b\lambda^{-\frac{x}2-\frac{y}2}+c\lambda^{\frac{x}2-\frac{y}2}-d\lambda^{\frac{x}2+\frac{y}2}\right)}{\log\lambda}.
\ee

\subsection{Aztec cuboid}\label{sec:discrete_3D}\label{sec:aztec_cuboid}

The \textit{Aztec cuboid} $\text{AC}(a,b,c)$ is a region in the three-dimensional $\bcc$ lattice (body-centered cubic lattice) which generalizes the Aztec diamond. See Figure~\ref{fig:cuboid} in the Introduction, which shows the flows lines limit shape divergence-free flow we calculate here. This construction can be generalized to body-centered cubic lattices in any dimension $d$. 

First we describe the $\bcc$ lattice. Take $B=\m Z^3$ and $W=\m Z^3 + (\frac{1}{2},\frac{1}{2},\frac{1}{2})$ and connect nearest neighbors (as we saw in the previous section, this construction applied to $\m Z^2$ produces a rotated copy of $\m Z^2)$. For this lattice, the free energy is 
\begin{align*}
    F(\alpha_1,\alpha_2,\alpha_3) = \log((\mathrm{e}^{\alpha_1/2}+\mathrm{e}^{-\alpha_1/2})(\mathrm{e}^{\alpha_2/2}+\mathrm{e}^{-\alpha_2/2})(\mathrm{e}^{\alpha_3/2}+\mathrm{e}^{-\alpha_3/2})).
\end{align*}
The Newton polytope is $\mc N = [-1/2,1/2]^3$ and the surface tension is 
\begin{align*}
    \sigma(s_1,s_2,s_3) = \sum_{i=1}^3 \frac{1+2s_i}{2}\log\bigg( \frac{1+2s_i}{2} \bigg) +  \frac{1-2s_i}{2}\log\bigg( \frac{1-2s_i}{2} \bigg).
\end{align*}
The three Euler-Lagrange equations for the divergence-free flow (as in Theorem~\ref{thm:EL_equations}) are
\begin{align}\label{eq:bcc_flow_euler}
    \frac{4}{1-4\omega_1^2} \pd{\omega_1}{y} - \frac{4}{1-4 \omega_2^2} \pd{\omega_2}{x}= 0, \quad \frac{4}{1-4\omega_1^2} \pd{\omega_1}{z} - \frac{4}{1-4 \omega_3^2} \pd{\omega_3}{x}= 0, \quad \frac{4}{1-4\omega_2^2} \pd{\omega_1}{z} - \frac{4}{1-4 \omega_3^2} \pd{\omega_3}{y}= 0.
\end{align}
The Euler-Lagrange equation for the gauge function $H(x,y,z)$ is 
\begin{align}\label{eq:bcc_gauge_euler}
    \frac{\mathrm{e}^{H_{x}}H_{xx}}{(1+\mathrm{e}^{H_{x}})^2} +  \frac{\mathrm{e}^{H_{y}}H_{yy}}{(1+\mathrm{e}^{H_{y}})^2} + \frac{\mathrm{e}^{H_{z}}H_{zz}}{(1+\mathrm{e}^{H_{z}})^2} =0.
\end{align}
To construct the Aztec cubiod, take positive integers $a,b,c$ satisfying\footnote{A general integer solution to \eqref{eq:abc} can be found by replacing $a,b,c$ with $\frac{x-1}{2}, \frac{x+y-3}{2}, \frac{x+z-3}{2}$ where $x$ is odd and $y,z$ are even. This reduces the equation to $yz = x^2-1$. Given an odd $x$, choose $y$ any divisor of $x^2-1$ such that $(x^2-1)/y$ is even.}
\begin{align}\label{eq:abc}
   (a+1)(b+1)(c+1 = a (b+2)(c+2).
\end{align}
We define the \textit{Aztec cuboid} $\text{AC}(a,b,c)\subset \bcc $ for integers satisfying \eqref{eq:abc} to be the region $B_{a,b,c}\cup W_{a,b,c}$, where
\begin{align*}
    B_{a,b,c} &= \{0,\dots, a+1\}\times \{0,\dots,b\} \times \{0,\dots, c\}\subset B\\
    W_{a,b,c} &= \{0',\dots, (a-1)'\} \times \{(-1)',0',\dots,b'\} \times \{(-1)',0',\dots,c'\}\subset W,
\end{align*}
where $(i',j',k') = (i+1/2,j+1/2,k+1/2)$. The critical gauge functions $g_{a,b,c},f_{a,b,c}$ for black and white vertices respectively are given by
\begin{align*}
    g_{a,b,c}(i,j,k) &= \frac{{a\choose i}}{{b\choose j}{c\choose k}}\\
    f_{a,b,c}(i,j,k) &= \frac{bc}{(b+1)(c+1)} \frac{{b-1\choose j}{c-1\choose k}}{{a+1\choose i+1}}
\end{align*}
Note that the black vertices adjacent to the white vertex $(i,j,k)$ are, in each direction $(\pm 1,\pm 1,\pm 1)$: 
\\

{\tiny\centerline{\begin{tabular}{c|c|c|c|c|c|c|c}
    $(-1,-1,-1)$ & $(-1,-1,1)$ & $(-1,1,-1)$ & $(1,-1,-1)$ & $(-1, 1, 1)$ & $(1,-1,1)$ & $(1,1,-1)$ & $(1,1,1)$\\
  \hline
   $(i,j,k)$ & $(i,j,k+1)$ & $(i,j+1,k)$ & $(i+1,j,k)$ & $(i,j+1,k+1)$ & $(i+1,j,k+1)$ & $(i+1,j+1,k)$ & $(i+1,j+1,k+1)$
\end{tabular}}
}
\vspace{0.25cm}

The corresponding critical edge weights for the eight directions incident to the white vertex $(i,j,k)$ are therefore: 
\begin{align*}
    f_{a,b,c}(i,j,k) g_{a,b,c}(i,j,k) &= \frac{(i+1)(b-j)(c-k)}{(a+1)(b+1)(c+1)} \quad \text{in the $(-1,-1,-1)$ direction}\\
    f_{a,b,c}(i,j,k) g_{a,b,c}(i,j,k+1) &= \frac{(i+1)(b-j)(k+1)}{(a+1)(b+1)(c+1)} \quad \text{in the $(-1,-1,1)$ direction}\\
    f_{a,b,c}(i,j,k) g_{a,b,c}(i,j+1,k) &= \frac{(i+1)(j+1)(c-k)}{(a+1)(b+1)(c+1)} \quad \text{in the $(-1,1,-1)$ direction}\\
    f_{a,b,c}(i,j,k) g_{a,b,c}(i+1,j,k) &= \frac{(a-i)(b-j)(c-k)}{(a+1)(b+1)(c+1)} \quad \text{in the $(1,-1,-1)$ direction}\\
    f_{a,b,c}(i,j,k) g_{a,b,c}(i,j+1,k+1) &= \frac{(i+1)(j+1)(k+1)}{(a+1)(b+1)(c+1)} \quad \text{in the $(-1,1,1)$ direction}\\
    f_{a,b,c}(i,j,k) g_{a,b,c}(i+1,j,k+1) &= \frac{(a-i)(b-j)(k+1)}{(a+1)(b+1)(c+1)} \quad \text{in the $(1,-1,1)$ direction}\\
    f_{a,b,c}(i,j,k) g_{a,b,c}(i+1,j+1,k) &= \frac{(a-i)(j+1)(c-k)}{(a+1)(b+1)(c+1)} \quad \text{in the $(1,1,-1)$ direction}\\
    f_{a,b,c}(i,j,k) g_{a,b,c}(i+1,j+1,k+1) &= \frac{(a-i)(j+1)(k+1)}{(a+1)(b+1)(c+1)} \quad \text{in the $(1,1,1)$ direction}
\end{align*}
To take the scaling limit, we choose a sequence $(a_n,b_n,c_n)$ satisfying \eqref{eq:abc} such that $(a_n/n,b_n/b,c_n/c)\to (A,B,C)$ as $n\to \infty$. By \eqref{eq:abc}, $A,B,C$ must satisfy ${1}/{A}={1}/{B}+{1}/{C}$. The limit shape for these regions $\omega=(\omega_1,\omega_2,\omega_3)$ is the flow where $\omega_1$ is the net flow in the $(1,0,0)$ direction, and so on. Taking the scaling limit of the edge weights above, we get that 
\begin{align}\label{eq:omega_bcc_direct}
    \omega = \left(1-\frac{2x}A, \frac{2y}B-1, \frac{2z}C-1\right).
\end{align}
This is divergence free because $1/A = 1/B+1/C$. See Figure~\ref{fig:cuboid} for an example. One can also check that $\omega$ satisfies the system of Euler-Lagrange equations for this lattice in \eqref{eq:bcc_flow_euler}.

We can also compute the scaling limit of the critical gauge functions $g_{a_n,b_n,c_n}$ and use this to compute $\omega$ via Theorem~\ref{thm:critical_gauge_scaling_limit}. If we take $(a_n,b_n,c_n)\approx (A n, B n, Cn)$, then 
\begin{align*}
    H(x,y,z) =& \lim_{n\to \infty}\frac{1}{n}g_{a_n,b_n,c_n}(xn,yn,zn) \\
    =\;& A \log A- x\log x - (A-x)\log(A-x) \\
    - &B\log B + y \log y + (B-y) \log (B-y) \\
    - &C\log C + z\log z + (C-z)\log(C-z).
\end{align*}
Note that this is $C^1$ on $(0,A)\times (0,B)\times (0,C)$. It is again straightforward to check that this solves the gauge PDE in \eqref{eq:bcc_gauge_euler}. From this we see that the gauge flow is 
\begin{align*}
    \alpha = \nabla H = (\log((A-x)/x), -\log((B-y)/y), -\log((C-z)/z)).
\end{align*}
Since $\nabla F(\alpha)=(1/2\tanh(\alpha_1/2),1/2\tanh(\alpha_2/2),1/2\tanh(\alpha_3/2))$, the limit shape divergence free flow is 
\begin{align*}
    \omega = \nabla F(\alpha) = \left(1-\frac{2x}A, \frac{2y}B -1, \frac{2z}C-1\right),
\end{align*}
in agreement with \eqref{eq:omega_bcc_direct}.

\subsection{Truncated orthants} 

For any dimension $d$, there is a $d-1$ dimensional lattice $\mc D_{d-1}$ which can be seen as a ``diagonal slice'' of $\m Z^d$. Namely, the vertices are 
\begin{align*}
    \mc D_{d-1}= \{(v_1,\dots,v_d) \in \m Z^d : \sum_i v_i = 0\text{ or 1}\}
\end{align*}
and the edges are those inherited from $\m Z^d$. In particular, $\mc D_2$ is the honeycomb lattice, and this family can be seen as generalization of the honeycomb lattice to arbitrary dimension. The three-dimensional case $\mc D_3$ is called the \textit{diamond cubic} lattice. 

While these are outside the scope of our main theorems (since we deal only with finite graphs), we give an example of a family of infinite graphs in $\mc D_d$ which have explicit critical gauge functions. We give these formulas explicitly for $d=2,3$.

\subsubsection{Example in the honeycomb lattice}\label{truncorth}

This is a generalization of an example from \cite{KenyonPohoata}. 
Here we take a linear image of the standard honeycomb lattice so that black vertices are at integer points, and the white vertices at $\Z^2+(\frac13,\frac13)$. Fix a positive integer $K$.

Consider the graph $G_K$ where the black vertices are $\{(i,j)~|~i,j\ge0, i+j\ge K\}$ and white vertices $\{(i+\frac13,j+\frac13)~|~i,j\ge0, i+j\ge K-1\}$, see Figure \ref{hexbinom}.
\begin{figure}[htbp]
\begin{center}
\includegraphics[width=3in]{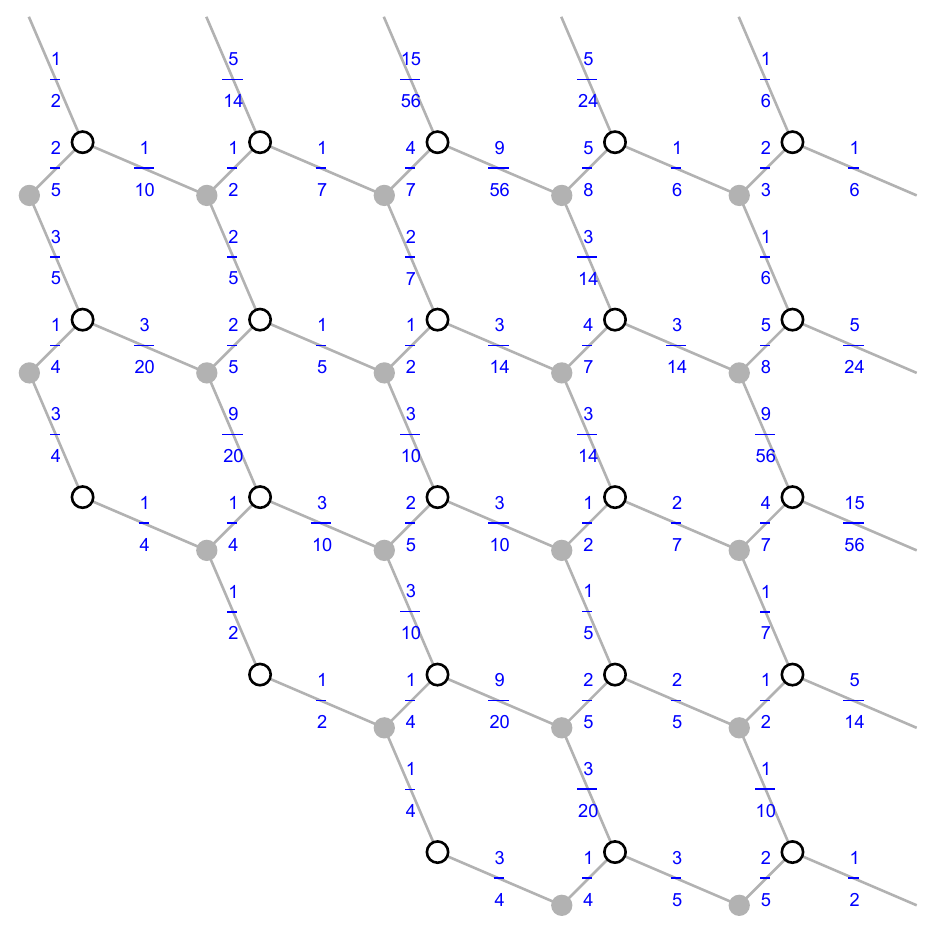}\hskip1cm\includegraphics[width=3in]{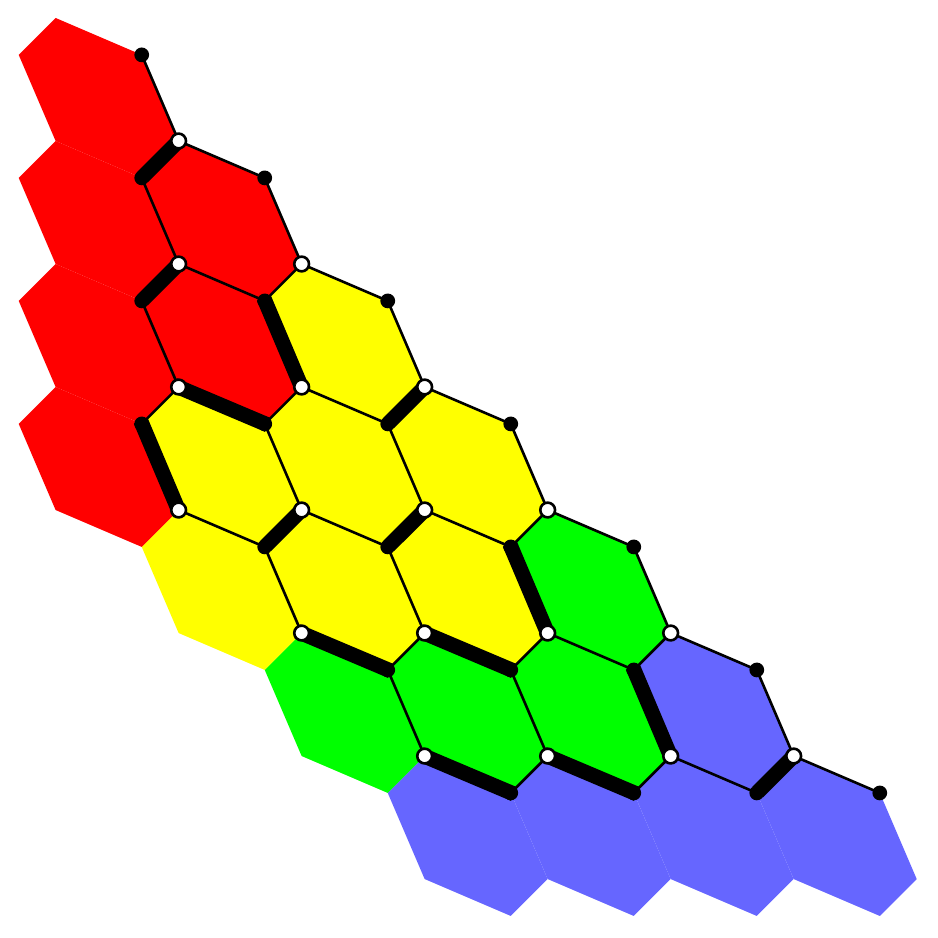}
\end{center}
\caption{\label{hexbinom} Truncated quadrant in the honeycomb graph, with $K=3$; the critical edge weights (left) and ``Polya corners process'' (right).}
\end{figure}

The critical gauge functions are
\begin{align*}b_{K}(i,j)&=\binom{i+j}{j}^{-1}\frac{K^{i+j}}{(i+j+1)(i+j-K)!}\\
w_K(i+\frac13,j+\frac13)&=\binom{i+j}{j}\frac{(i+j+1-K)!}{K^{i+j}}
\end{align*}
at black and white vertices respectively. 

This comes from a construction that relates dimer covers of these graphs to the classical \textit{Polya's urn}, as follows. Put $K+1$ balls of distinct colors in an urn; at each future time step, select one ball at random and replace it along with another ball of the same color.
Letting $n_i$ be the number of balls of color $i$ at time $t$, the distribution of $(n_1,\dots,n_{K+1})$ is always uniform over all possibilities $\{n_1+\dots+n_{K+1}=K+1+t\}$. 

Given an instance of this process, one can build a single dimer cover of $G_K$ as follows. First, we color the faces of the graph with colors $1,\dots,K+1$, monotonically increasing to the northwest, so that the $t$-th diagonal has $n_i$ faces of each color. Next, we place a dimer on the diagonal edges at $i+j=K+1+t$ between adjacent faces of the same color. (This rule determines all diagonal edges, and thus the entire dimer cover, see the figure.) The critical edge weights given above are exactly the edge probabilities for this dimer measure. Note that on a given diagonal row all diagonal edge weights are the same; this is a consequence of the uniformity of the urn process.

\subsubsection{Example in the diamond cubic lattice}

The Polya urn construction above can be generalized to any dimension; here we give the resulting formulas for $d=3$. We view the diamond cubic lattice $\mc D_3$ embedded in $\m R^3$ so that its vertices are the union of $B=\Z^3$ and $W=\Z^3+(\frac14,\frac14,\frac14)$, with edges connecting
$b_{i,j,k}$ to the four vertices
$w_{i,j,k}$, $w_{i-1,j,k}$, $w_{i,j-1,k}$, and $w_{i,j,k-1}$.

Let $\m N_K^3$ be the subgraph of $\Z^3$ consisting of black vertices $b_{i,j,k}=(i,j,k)\in\Z^3$ with $i,j,k\ge 0$ and $i+j+k\ge K$ and white vertices $w_{i,j,k}=(i+\frac14,j+\frac14,k+\frac14)$ with integer $i,j,k\ge 0$ and $i+j+k\ge K-1$ . This is the graph of the ``truncated orthant".
Let $L_K$ be the corresponding region in the diamond cubic lattice, namely, the black vertices in $\m N_{K}^3$ and white vertices connected to them by an edge of $\mc D_3$. 

The critical gauge functions $w_Kb_K$ on $L_K$ for white and black vertices respectively are
\begin{align*}
    w_K(w_{i,j,k}) &={i+j+k \choose{i,j}} u_{i+j+k}\\
        b_K(b_{i,j,k}) &= {i+j+k \choose{i,j}}^{-1} v_{i+j+k}
\end{align*}
where for $\ell\ge K$, the coefficients satisfy  
\begin{align*}(u_{\ell-1}+u_\ell)v_\ell&=1\\
u_\ell(v_\ell+\frac{\ell+3}{\ell+1}v_{\ell+1})&=1
\end{align*}
with $v_{K-1}=0$. This leads to 
$$u_\ell = \frac{(\ell-K+1)!(\ell+K+2)!}{K^{\ell}(K+1)^{\ell}(2K+1)!}$$
$$v_\ell = \frac{(2K+1)!K^{\ell}(K+1)^{\ell}}{(\ell+1)(\ell+2)(\ell+K+1)!(\ell-K)!}.$$

The scaling limit of these is a divergence free flow. To see this, let $i=xK, j=yK,k=zK$ so the domain becomes $\{(x,y,z)~|~ x+y+z\ge 1, x,y,z\ge 0\}$. As above, let $\ell = i+j+k$. Then in the scaling limit, the critical gauge goes to the vector field is $\omega = (\omega_1,\omega_2,\omega_3)$ given by
\begin{align*}\omega_1 &=  w_K(w_{i,j,k})b_K(b_{i+1,j,k}) =  \frac{(i+1)K(K+1)}{(\ell+1)(\ell+2)(\ell+3)}\to  \frac{x}{(x+y+z)^3}\\
\omega_2 &=  w_K(w_{i,j,k})b_K(b_{i,j+1,k}) =  \frac{(j+1)K(K+1)}{(\ell+1)(\ell+2)(\ell+3)}\to  \frac{y}{(x+y+z)^3}\\
\omega_3 &= w_K(w_{i,j,k})b_K(b_{i,j,k+1}) =  \frac{(k+1)K(K+1)}{(\ell+1)(\ell+2)(\ell+3)}\to  \frac{z}{(x+y+z)^3}.
\end{align*}
Note that $\omega$ divergence free.

\subsection{Tangent plane method for 2D limit shapes}\label{sec:tangent_plane}

The tangent plane method of \cite{KenyonPrause} gives a general method to find the solutions to two-dimensional gradient variational problems, describing the limit shape height function $h$ in the \textit{liquid region} $L = \{(x_1,x_2)\in R : \nabla h(x_1,x_2) \in \mc N^\circ\}$ as an envelope of planes. By Theorem~\ref{thm:no_facets}, $L = R^\circ$ for multinomial limit shapes. We apply the results of \cite{KenyonPrause} to give a general solution for 2D multinomial limit shapes as an envelope of planes satisfying the boundary conditions.

In this section we use the notation that $\nabla h=(s,t)$, and write the relavent functions (surface tension, etc) in these coordinates. If $h: R\to \m R$ is a smooth function, the graph of $h$ is the envelope of its family of tangent planes.
The corresponding family of planes $P_{(x,y)}$ is given by
\begin{align*}
    P_{(x,y)} = \{x_3 = sx+ty + c\}, \qquad c = h(x,y) - (sx+ty); \quad \nabla h(x,y) = (s,t).
\end{align*}
For our functions $h$, it is natural to consider the coefficients $s,t,c$ as functions of the underlying conformal coordinate 
$z$ rather than $x,y$. Then they are each $\kappa$-harmonic real functions of $z$.

Given a position-valued \textit{conductance function} $\kappa: \m R^2\to \m R$, a \emph{$\kappa$-harmonic function} is a function $w$ which solves $\text{div}(\kappa\, \nabla w) = 0$. A $\kappa$-harmonically moving family of planes is one where the coefficients $s=s(z),t=t(z)$, and $c=c(z)$ are $\kappa$-harmonic functions of a complex coordinate $z$. 

As an application of \cite[Theorem 3.2]{KenyonPrause} we get the following.
\begin{thm}
Fix $(R,b)\subset \m R^2$ flexible and let $h_b$ be the corresponding boundary height function. 
Let $\Lambda = \m Z^2$ or the honeycomb lattice. The multinomial limit shape height function with boundary condition $h\mid_{\partial R}=h_b$ is a smooth function which is the envelope of $\kappa$-harmonically moving planes, where $z$ is a conformal coordinate in the Riemannian metric determined by $\sigma$ and $\kappa = \sqrt{\det \text{Hess}(\sigma)}$. Concretely,
    \begin{itemize}
        \item For $\m Z^2$ embedded with edge vectors $(\pm 1/2, \pm 1/2)$, $\mc N = \{(s,t)\in [-1/2,1/2]^2\}$. A conformal coordinate $z = u+i v$ for $(u,v)\in[0,\pi]^2$ is given by $$(u,v) = (\arccos(2s),\arccos(2t))$$ and $\kappa(u,v) = \frac{4}{\sin(u)\sin(v)}$. 
        \item For the honeycomb lattice, $\mc N =\{(s,t): 0\leq s\leq 1, 0\leq t\leq s\}.$ A conformal coordinate $z = u+iv\in\H$ is given by 
        \begin{align}
            (u,v) = \bigg( \frac{s-t(1-s-t)}{(1-t)^2}, \frac{2\sqrt{st(1-s-t}}{(1-t)^2}\bigg)
        \end{align}
        and $\kappa(u,v) = \frac{(1+ \sqrt{u^2 + v^2} + \sqrt{(u-1)^2 + v^2})^2}{2 v}$.
    \end{itemize}
\end{thm}

In the next subsections we explain how to compute the conformal coordinate $z$ for each lattice, and use these coordinates to compute limit shapes in some examples. 
Note that $w(z)$ being $\kappa$-harmonic is equivalent to the statement that $\tilde w(z)$ solves the Schr\"odinger equation 
\begin{align}\label{eq:schrodinger}
    (-\Delta + q)(\tilde{w}) =0,
\end{align}
where $\psi = \kappa^{1/2}$, $q = \frac{\Delta \psi}{\psi}$ and $\tilde{w} = \psi w$. 

\subsubsection{Square lattice and Aztec diamond} 

We embed $\m Z^2$ in the plane so that its edges are parallel to the directions $(\pm 1/2,\pm 1/2)$. In these coordinates the Newton polygon $\mc N= [-1/2,1/2]^2$ and the multinomial dimer surface tension in terms of $\nabla h = (s,t)$ is
\begin{align*}
    \sigma(s,t) = \frac{1+2s}{2} \log\bigg(\frac{1+2s}{2}\bigg) + \frac{1-2s}{2} \log\bigg(\frac{1-2s}{2}\bigg) + \frac{1+2t}{2} \log\bigg(\frac{1+2t}{2}\bigg) +\frac{1-2t}{2} \log\bigg(\frac{1-2t}{2}\bigg).
\end{align*} 
The Hessian matrix is 
%\begin{align*}
 %   \text{Hess}(\sigma) = \bigg(\begin{array}{cc}
  %      \frac{1}{1-s^2} & 0 \\
   %     0 & \frac{1}{1-t^2}
   % \end{array}\bigg).
%\end{align*}
{\begin{align*}
    \text{Hess}(\sigma) = \bigg(\begin{array}{cc}
        \frac{4}{1-4s^2} & 0 \\
        0 & \frac{4}{1-4t^2}
    \end{array}\bigg).
\end{align*}}
The Riemannian metric on $\mc N$ determined by $\sigma$ is: 
%\begin{align*}
 %   m = \sigma_{ss} \dd s^2 + 2 \sigma_{st} \dd s \dd t + \sigma_{tt} \dd t^2 = \frac{\dd s^2}{1-s^2} + \frac{\dd t^2}{1-t^2}. 
%\end{align*}
$$
    m = \sigma_{ss} \dd s^2 + 2 \sigma_{st} \dd s \dd t + \sigma_{tt} \dd t^2 = \frac{4\dd s^2}{1-4s^2} + \frac{4\dd t^2}{1-4t^2}. 
$$
A conformal coordinate is $z = u+iv$ is one which makes $m = e^g(\dd u^2 + \dd v^2)$ for some function $g$. 
By observation we can take
\begin{align*}
    (u,v) = (\arccos(2s), \arccos(2t)) \in [0,\pi]^2. 
\end{align*}
In $(u,v)$ coordinates $\kappa$ is 
%\begin{align*}
 %   \kappa = \sqrt{\det \text{Hess}(\sigma)} = \sqrt{\frac{1}{\sin^2(u) \sin^2(v)}} = \frac{1}{\sin(u) \sin(v)}.
%\end{align*}
\begin{align*}
    \kappa = \sqrt{\det \text{Hess}(\sigma)} = \sqrt{\frac{16}{\sin^2(u) \sin^2(v)}} = \frac{4}{\sin(u) \sin(v)}.
\end{align*}
Therefore the $q$ in \eqref{eq:schrodinger} in this case is 
\begin{align*}
    q(u,v) = -\frac{1}{2} + \frac{3}{4\sin^2(u)} + \frac{3}{4\sin^2(v)}
\end{align*}
The associated Schr\"odinger equation $(-\Delta + q)\tilde{w}=0$ is a linear second-order equation which we can solve using separation of variables. Suppose that $\tilde{w}(u,v) = a(u) b(v)$. For any constant $\alpha$ we can then factor the Schr\"odinger equation as:
\begin{align*}
   \bigg[-a''(u)+\bigg(-\alpha+\frac{3}{4\sin^2(u)}\bigg)a(u) \bigg]b(v) + \bigg[-b''(v) + \bigg(-\frac{1}{2}+\alpha + \frac{3}{4\sin^2(v)}\bigg) b(v)\bigg]a(u) =0.
\end{align*}
Hence $\tilde{w}=a(u)b(v)$ satisfies \eqref{eq:schrodinger} if $a,b$ satisfy for some constant $\lambda$ the equations:
\begin{align*}
    a''(u)+\bigg(\lambda-\frac{3}{4\sin^2(u)}\bigg)a(u)&= 0\\
    b''(v) + \bigg(\frac{1}{2} - \lambda - \frac{3}{4\sin^2(v)}\bigg)b(v) &=0.
\end{align*}
The solutions for arbitrary $\lambda$ can be computed in terms of ${}_2F_1$ hypergeometric functions. Varying $\lambda$ would lead to more solutions, but we note that when $\lambda=1/4$, $a,b$ satisfy the same equation whose solution is 
\begin{align*}
    a(u) = \frac{c_1+c_2 \cos(u)}{\sqrt{\sin(u)}}.
\end{align*}
where $c_1,c_2$ are constants. From this solution we get a family of $\kappa$-harmonic functions $w=\tilde{w}/\kappa^{1/2}$ which can be written
\begin{align*}
    w = c_1 + c_2 \cos(u) + c_3 \cos(v) + c_4 \cos(u) \cos(v)
\end{align*}
for arbitrary constants $c_1,c_2,c_3,c_4$. In particular we verify that $s=\cos(u)/2$ and $t=\cos(v)/2$ are $\kappa$-harmonic. 

If we specialize to the case of the Aztec diamond, the limiting height function is $h(x,y) = {1}/{4}\cdot(2x-1)(2y-1)$. So then $\cos(u) = 2s = 2h_{x} = 2y-1$ and $\cos(v) = 2t = 2h_{y} = 2x-1$, and we get
\begin{align*}
    c = h(x,y)- (sx + ty) = st  - sx - ty= -\frac{1}{4}(\cos(u) + \cos(v)+\cos(u)\cos(v))
\end{align*}
which is indeed also $\kappa$-harmonic, as expected. The graph of the height function $(x,y,h(x,y))$ is hence the envelope of the planes
{\begin{align*}
    P_{z} = \{x_3 = \frac{1}{2}\cos(u) x + \frac{1}{2}\cos(v) y -\frac{1}{4}(\cos(u) + \cos(v)+\cos(u)\cos(v)\}.
\end{align*}}
for $u,v\in [0,\pi]$.

\subsubsection{Honeycomb lattice} 

Here the surface tension as a function of $\nabla h = (s,t)$ is 
\begin{align*}
    \sigma(s,t) = s \log s + t\log t + (1-s-t) \log(1-s-t)
\end{align*}
and hence 
\begin{align*}
    \Hess_\sigma = \bigg(\begin{array}{cc} 1/s + 1/(1-s-t) & 1/(1-s-t) \\ 1/(1-s-t) & 1/t + 1/(1-s-t)\end{array}\bigg).
\end{align*}
{The Euler-Lagrange equation for the height function $h=h(x,y)$, $\nabla h(x,y) = (s,t)$, is
\be\label{lozPDE}(h_y-h_y^2)h_{xx}+2h_xh_yh_{xy} + (h_x-h_x^2)h_{yy}=0.\ee}

We find a conformal coordinate $z=u+iv$ in the upper half plane $\H=\{\Im(z)>0\}$ to be given by:
\begin{align*}
    u(s,t) &= \frac{s-t(1-s-t)}{(1-t)^2}\\
    v(s,t) &= \frac{2\sqrt{st (1-s-t})}{(1-t)^2}.
\end{align*}
This gives 
$$\psi=(\det H)^{1/4} = \frac{1+\sqrt{(u-1)^2+v^2}+\sqrt{u^2+v^2}}{\sqrt{2v}}=\frac{1+|z-1|+|z|}{\sqrt{2v}}$$
from which we find 
\be\label{qhoney}q=\frac{\Delta\psi}{\psi} = \frac3{4v^2}.\ee
We need to solve
$$(-\Delta+\frac3{4v^2})\tilde w(u,v)=0.$$
Looking for a solution of the form $\tilde w(w,u)=a(u)b(v)$ leads to, for a constant $C$, 
\begin{align*}
-a''(u)+Ca(u)&=0\\
-b''(v)+(-C+\frac3{4v^2})b(v)&=0.
\end{align*}
Here $a$ is an exponential function and $b(v)$ is a Bessel function:
$$b(v) = c_1\sqrt{v}J_1(v\sqrt{C}) + c_2\sqrt{v}Y_1(v\sqrt{C})$$
where $J_1,Y_1$ are the Bessel functions of the first and second kind, resp.
In the special case $C=0$, we have
$b(v)=\frac{c_1+c_2v^2}{\sqrt{v}}$ and $a(u)=c_3+c_4u$. This leads to ``elementary'' $\kappa$-harmonic functions
$$w= \frac{c_1+c_2u+c_3v^2+c_4uv^2}{1+|z-1|+|z|}$$
for constants $c_1,c_2,c_3,c_4$.
Another family of $\kappa$-harmonic functions are those of the form 
$$\frac{|z-a|}{1+|z|+|z-1|}$$
for constant $a$ (and linear combinations of such).

The function $s,t$ are the $\kappa$-harmonic functions
\begin{align*} s &= \frac{1-|z-1|+|z|}{1+|z-1|+|z|}\\
t &= \frac{-1+|z-1|+|z|}{1+|z-1|+|z|}.
\end{align*}

As an example limit shape, take $c=-\frac{u}{1+|z-1|+|z|}$. 
Then solving the linear system
\begin{align*}x_3&=sx+ty+c\\
0&=s_ux+t_uy+c_u\\
0&=s_vx+t_vy+c_v
\end{align*}
for $x,y$ and $x_3=h$ gives 
\be\label{hcone}(x,y,h) = \left(\frac{|z|+|z-1|}2,\frac{|z|+u-1}2,\frac{|z-1|+u-1}2\right).\ee
 As $z=u+iv$ runs over $\H$ the upper half plane, this parameterizes the limit
 shape for the truncated orthant (see section \ref{truncorth}) with vertices $(\frac12,-\frac12,0)$ and $(\frac12,\frac12,0)$.
 On this region in (\ref{hcone}) we can solve for $h$ as a function of $x,y$:
 $$h(x,y) = \frac{(2x-1)(2y+1)}{2(2x+1)},$$
 which indeed satisfies (\ref{lozPDE}). 
% Translating so the lower vertex is at the origin the equation becomes $h(x,y)=\frac{xy}{x+1}$.
 %Other similar solutions are $h(x,y) = \frac{axy}{x+(a-1)y+b}.$

\section{Open questions}\label{sec:questions}

\begin{question}\label{q:closed_form_critical_gauge}
    In Section~\ref{sec:limit_shapes} we saw various examples of growing families of graphs where the critical gauge has an explicit (in fact rational) closed form in terms of $n$. These are the only such examples we know. What are other families of graphs where the critical gauge are rational functions of $n$, or otherwise have a closed form in terms of $n$? What are the properties of these graphs? 
\end{question}

\begin{question}\label{q:always_smooth}
    Using the height function, we saw that multinomial limit shapes in 2D are smooth from the fact that they do not have facets (Corollary~\ref{cor:2Dsmooth}). Is it true that multinomial limit shapes are smooth in any dimension? A positive answer to this question would mean that the $C^1$ condition in Theorem~\ref{thm:critical_gauge_scaling_limit} is satisfied in any dimension.
\end{question}

\begin{question}
    On a finite graph as the multiplicity $N$ goes to infinity, the fluctuations of the critical edge weights are a Gaussian process, with covariances determined by a version of the graph Laplacian \cite{KenyonPohoata}. What are the scaling limits of these processes?
\end{question}

%\begin{question}
 %   Fix any region and boundary condition pair $(R,b)$ flexible where the standard dimer limit shape has facets (e.g., the Aztec diamond). As we proved, the $N\to \infty$ multinomial limit shape does not have facets. We conjecture that for all finite $N$, the $N$ multinomial limit shape has facets, with size decreasing as $N$ grows. Can anything precise be said about these limit shapes, e.g.\ for the Aztec diamond? 
%\end{question}

\begin{question}\label{q:multinomial_finiteN}
    Here we have studied the $N\to\infty$ limit of the multinomial dimer model; the $N=1$ case is the standard dimer model. Can anything be said for small $N$, e.g.\ $N=2$ or $N=3$? For example, do the $N$ multinomial limit shapes for the Aztec diamond have facets, which decrease in size as $N$ grows?
\end{question}

\begin{question}
    Overlaying two single dimer covers gives a collection of non-intersecting loops and infinite paths, built by alternating between the two dimer covers. Overlaying two dimer covers of the $N$-fold blow up graph, what can be said about ``geography'' of the infinite paths, either in the blow up graph or in their projection? Can this be used to say anything about ergodic Gibbs measures for this model?
\end{question}
%\begin{question}
 %   Overlaying two single dimer covers of a lattice gives a collection of non-intersecting loops and infinite paths, built by alternating between the two dimer covers. Overlaying $N>2$ independent single dimer covers can still be described as a system of paths, but there is no canonical way to decide how to ``route'' them, i.e.\ which dimer to follow next at each vertex out of the $N-1$ options. On the other hand, the data of a dimer cover of the $N$-fold blow-up graph gives a canonical way to do this, by choosing the dimer adjacent to the same lifting vertex. What can be said about the ``geography'' of infinite paths in an $N$-dimer cover sampled from a multinomial dimer measure? How many infinite paths typically cross through the same point in the downstairs graph? Typically how many connected components are there, as a function of $N$? 
%\end{question}

%\begin{question}
 %   Can the ``infinite path geography'' in the previous question be used to say anything about ergodic Gibbs measures for this model? 
%\end{question}

\begin{question}
    The multinomial tiling model, of which the multinomial dimer model is a special case, allows for any collection of tiles. What can be said about the ``large $N$ limits'' of other statistical mechanics models, e.g.\ the square ice model, using the multinomial framework? 
\end{question}

\bibliographystyle{alpha}
\bibliography{multinomial}

\end{document}